\documentclass[11pt, a4paper,reqno]{amsart}
\pdfoutput=1
\usepackage{tikz,rotating,latexsym,bm,stmaryrd,caption,}
\usetikzlibrary{positioning,intersections,decorations,shadings}
\usetikzlibrary{shapes}
\usetikzlibrary{arrows}
\usetikzlibrary{patterns}
\usepackage{tkz-euclide}
\usetikzlibrary{calc}
\usetikzlibrary{shapes.geometric}
\tikzset{snake it/.style={decorate, decoration=snake}}
\tikzset{zigzag/.style={decorate, decoration=zigzag}}
\usetikzlibrary{fadings}
\usepackage{enumerate}

\usepackage[matrix,arrow]{xy}
\newcommand{\hackcenter}[1]{
 \xy (0,0)*{#1}; \endxy}

\definecolor{ao(english)}{rgb}{0.0, 0.5, 0.0}

\pgfdeclareverticalshading{rainbow}{100bp}
{color(0bp)=(orange);color(25bp)=(orange);
color(50bp)=(magenta); color(75bp)=(cyan!50);
 color(100bp)=(cyan!50)}

\newcommand\bbf{\mathbb{F}}

\newcommand\bbz{\mathbb{Z}}

\newcommand\bfi{\mathbf{i}}
\newcommand\bfj{\mathbf{j}}

\newcommand\sfp{\mathsf{P}}
\newcommand\sfq{\mathsf{Q}}

\newcommand\fkg{\mathfrak{g}}

\newcommand\fkl{\mathfrak{l}}

\newcommand\fkp{\mathfrak{p}}

\newcommand\fks{\mathfrak{s}}

\newcommand\scrp{\mathscr{P}}

\newcommand\scrr{\mathscr{R}}

\newcommand\tta{\mathtt{A}}

\newcommand\ttc{\mathtt{C}}

\newcommand\tts{\mathtt{s}}
\newcommand\ttt{\mathtt{t}}
\newcommand\ttu{\mathtt{u}}
\newcommand\ttv{\mathtt{v}}

\newcommand\sss[1][n]{\mathfrak{S}_{#1}}

\newcommand\spe[1]{\operatorname{S}^{#1}}
\newcommand\D[1]{\operatorname{D}^{#1}}
\newcommand\mptn[2]{\scrp^{#1}_{#2}}

\newcommand{\dom}{\trianglerighteqslant}

\usetikzlibrary{decorations.pathmorphing}

\definecolor{eng}{rgb}{0.0, 0.5, 0.0}
\definecolor{apple}{rgb}{0.55, 0.71, 0.0}
\definecolor{cadmium}{rgb}{0.0, 0.42, 0.24}
\definecolor{darkspringgreen}{rgb}{0.09, 0.45, 0.27}
\definecolor{amethyst}{rgb}{0.6, 0.4, 0.8}
\definecolor{ao}{rgb}{0.0, 0.0, 1.0}
\definecolor{atomictangerine}{rgb}{1.0, 0.6, 0.4}
\definecolor{carmine}{rgb}{0.59, 0.0, 0.09}

\definecolor{toggle}{rgb}{1.0, 0.94, 0.96}

\newcommand{\kapc}{\kappa_\ttc}

\usepackage{standalone}

\usepackage{etex}
 
\tikzset{
  variable line width/.style={
    every variable line width/.append style={#1},
    to path={%
      \pgfextra{%
        \draw[every variable line width/.try,line width=\pgfkeysvalueof{/tikz/thickness}] (\tikztostart) -- (\tikztotarget);
      }%
      (\tikztotarget)
    },
  },
  thickness/.initial=0.6pt,
  every variable line width/.style={line cap=round, line join=round},
}

\usepackage{todonotes}

\usepackage{tikz}
\usetikzlibrary{matrix,  intersections, calc, decorations.pathreplacing} 

\usepackage{tikz-cd}

\newlength{\superthick}
\newlength{\cornerradius}
\tikzstyle{corner}=[rounded corners=\cornerradius]
\tikzstyle{dot}=[circle, inner sep=0pt, minimum size=4.8pt]
\tikzstyle{string}=[line width=\superthick]
\tikzstyle{std}=[string,dash pattern=on 0.9pt off 0.9pt]
\definecolor{realcyan!50}{rgb}{0,1,1}

\tikzstyle{unmarkedshading}=[copperred, pattern=north east lines, pattern color = copperred!40!white, rounded corners]
\tikzstyle{conditionalshading}=[plum, pattern=crosshatch dots, pattern color = plum, rounded corners]

 \usepackage{amsmath,amsthm,amsfonts,amssymb,mathrsfs,pb-diagram}
\usepackage[
bookmarks=true,colorlinks=true,linktoc=page,citecolor=green!50!black,linkcolor=green!50!black,urlcolor=blue]{hyperref}
\usepackage{caption}
\usepackage{lipsum,wasysym}
\usepackage{mathtools}
\usepackage[a4paper,margin=0.85in]{geometry}
\usepackage{cleveref}
 \usepackage{amsmath}
\mathchardef\mhyphen="2D
\usepackage{color}
 
\newcommand{\Rem}{\mathrm{Rem}}
\newcommand{\Add}{\mathrm{Add}}

\renewcommand{\geq}{\geqslant}
\renewcommand{\leq}{\leqslant}

\tikzset{wei/.style= 
{red,double=red,double
distance=0.5pt}}

\DeclareMathOperator{\image}{im}

\tikzset{wei2/.style={red,double=red,double
distance=0.5pt}}

\allowdisplaybreaks
\numberwithin{equation}{section}
\usepackage{scalefnt}

\newtheorem{thm}[equation]{Theorem}
\newtheorem{cor}[equation]{Corollary}

\newtheorem{lem}[equation]{Lemma}
\newtheorem{prop}[equation]{Proposition}
\newtheorem*{Acknowledgements*}{Acknowledgements}

\theoremstyle{definition}
\newtheorem{defn}[equation]{Definition}
\newtheorem*{eg}{Example}

\newtheorem{Remark}[equation]{Remark}

\theoremstyle{remark}

\newtheorem*{rem}{Remark}
\newtheorem*{rmk}{Remarks}

\numberwithin{equation}{section}

\newcommand{\rad}{\mathrm{rad}}
\newcommand{\res}{\mathrm{res}}

\newcommand{\Std}{{\rm Std}}
\newcommand{\SStd}{{\rm SStd}}

\newcommand{\Shape}{\operatorname{Shape}}

\newcommand{\La}{\Lambda}
\newcommand{\la}{\lambda}
\newcommand{\pla}{{\color{magenta}\lambda}}
\newcommand{\bbmu}{{\color{cyan}\mu}}

\newcommand{\SSTS}{\mathtt{S}}

\newcommand{\SSTT}{\mathtt{T}}  
\newcommand{\ZZ}{{\mathbb Z}}

\tikzset{
ultra thin/.style= {line width=0.05pt},
very thin/.style=  {line width=0.2pt},
thin/.style=       {line width=0.1pt},
semithick/.style=  {line width=0.6pt},
thick/.style=      {line width=0.8pt},
very thick/.style= {line width=1.2pt},
ultra thick/.style={line width=1.6pt}
}

\AddToHook{env/lem/begin}{\crefalias{equation}{lem}}
\AddToHook{env/thm/begin}{\crefalias{equation}{thm}}
\AddToHook{env/cor/begin}{\crefalias{equation}{cor}}
\AddToHook{env/prop/begin}{\crefalias{equation}{prop}}
\AddToHook{env/conj/begin}{\crefalias{equation}{conj}}
\AddToHook{env/thmx/begin}{\crefalias{equation}{thmx}}

\AddToHook{env/lemciting/begin}{\crefalias{equation}{lemciting}}
\AddToHook{env/propciting/begin}{\crefalias{equation}{propciting}}
\AddToHook{env/thmciting/begin}{\crefalias{equation}{thmciting}}
\AddToHook{env/corciting/begin}{\crefalias{equation}{corciting}}
\AddToHook{env/conjciting/begin}{\crefalias{equation}{conjciting}}

\AddToHook{env/defn/begin}{\crefalias{equation}{defn}}
\AddToHook{env/Example/begin}{\crefalias{equation}{Example}}
\AddToHook{env/Notation/begin}{\crefalias{equation}{Notation}}
\AddToHook{env/Assumption/begin}{\crefalias{equation}{Assumption}}
\AddToHook{env/Remark/begin}{\crefalias{equation}{Remark}}

\AddToHook{env/Example/begin}{\crefalias{equation}{Example}}
\AddToHook{env/Notation/begin}{\crefalias{equation}{Notation}}
\AddToHook{env/Assumption/begin}{\crefalias{equation}{Assumption}}

\crefname{ques}{Question}{Questions}
\crefname{defn}{Definition}{Definitions}
\crefname{thm}{Theorem}{Theorems}
\crefname{prop}{Proposition}{Propositions}
\crefname{lem}{Lemma}{Lemmas}
\crefname{cor}{Corollary}{Corollaries}
\crefname{conj}{Conjecture}{Conjectures}
\crefname{section}{Section}{Sections}
\crefname{subsection}{Section}{Sections}
\crefname{eg}{Example}{Examples}
\crefname{figure}{Figure}{Figures}
\crefname{rmk}{Remark}{Remarks}
\crefname{Remark}{Remark}{Remarks}
\crefname{rmk}{Remark}{Remarks}
\crefname{equation}{equation}{equation}

\Crefname{ques}{Question}{Questions}
\Crefname{defn}{Definition}{Definitions}
\Crefname{thm}{Theorem}{Theorems}
\Crefname{prop}{Proposition}{Propositions}
\Crefname{lem}{Lemma}{Lemmas}
\Crefname{cor}{Corollary}{Corollaries}
\Crefname{conj}{Conjecture}{Conjectures}
\Crefname{section}{Section}{Sections}
\Crefname{subsection}{Subsection}{Subsections}
\Crefname{eg}{Example}{Examples}
\Crefname{figure}{Figure}{Figures}
\Crefname{rmk}{Remark}{Remarks}
\Crefname{Remark}{Remark}{Remarks}
\Crefname{rmk}{Remark}{Remarks}

\usepackage[hang,flushmargin]{footmisc}

\newcommand{\bla}{\boldsymbol{\la}}
\newcommand{\bmu}{\boldsymbol{\mu}}

\definecolor{zajj}{HTML}{008148}
\definecolor{marker}{HTML}{9448BC}
\definecolor{plum}{HTML}{9448BC}
\definecolor{copperred}{HTML}{DA6244}
\definecolor{orangepeel}{HTML}{FFA62B}
\definecolor{mantis}{HTML}{73BD61}
\definecolor{teal}{HTML}{247BA0}

\tikzstyle over=[draw=white,double=black,line width=2pt, double distance=.4pt]
\tikzstyle{B}=[draw, fill=black, circle, inner sep=0pt, outer sep=0pt, minimum size=5pt]
\tikzstyle{V}=[draw, fill =black, circle, inner sep=0pt, minimum size=1.5pt]
\tikzstyle{M}=[draw, black, fill =marker, circle, double, inner sep=0pt, minimum size=5pt]
\tikzstyle{bV}=[draw, fill =black, circle, inner sep=0pt, minimum size=3.5pt]
\tikzstyle{cV}=[draw, fill =white, circle, inner sep=0pt, minimum size=3.5pt]
\tikzstyle{BoxArr}=[xscale = .2, yscale=-.2]

\def\bi{\text{\boldmath$i$}}

\def\bj{\text{\boldmath$j$}}
\def\bm{\text{\boldmath$m$}}
\def\bk{\text{\boldmath$k$}}

\def\b1{\text{\boldmath$1$}}

\newcommand{\Par}{\mathscr{P}}

\newcommand{\ParblockA}[1][\beta-\omega]{\mathscr{P}_{#1}^{\mathtt{A}}}
\newcommand{\ParblockC}[1][\beta]{\mathscr{P}_{#1}^{\mathtt{C}}}

\definecolor{bittersweet}{rgb}{1.0, 0.44, 0.37}
	\definecolor{lightcoral}{rgb}{0.94, 0.5, 0.5}

\title{}
\author{}

\begin{document} 
  
%
\newcommand{\FourAuthors}{%
  \begin{minipage}{\textwidth}
    \centering
    \begin{tabular}{@{}c@{\qquad}c@{}}
      \begin{tabular}{c}
        Chris Bowman\\
        University of York\\
        Heslington, York, YO10 5DD, UK\\
        \texttt{chris.bowman-scargill@york.ac.uk}
      \end{tabular}
      &
      \begin{tabular}{c}
        Robert Muth\\
        Duquesne University\\
        Pittsburgh PA, USA 15282\\
        \texttt{muthr@duq.edu}
      \end{tabular}
      \\[3em]
      \begin{tabular}{c}
        Liron Speyer\\
        Okinawa Institute of Science and Technology\\
        Okinawa, Japan 904-0495\\
        \texttt{liron.speyer@oist.jp}
      \end{tabular}
      &
      \begin{tabular}{c}
        Louise Sutton\\
        Okinawa Institute of Science and Technology\\
        Okinawa, Japan 904-0495\\
        \texttt{louise.sutton@oist.jp}
      \end{tabular}
    \end{tabular}
  \end{minipage}%
}

\author[Chris Bowman, Robert Muth, Liron Speyer \& Louise Sutton]{\protect\FourAuthors}

 \title{Morita equivalences between cyclotomic KLR algebras in types $\mathtt{C}_\infty$ and $\mathtt{A}_\infty$}

\maketitle

\begin{abstract}
We prove that level one cyclotomic KLR algebras in type $\mathtt{C}_\infty$ are graded Morita equivalent to level two cyclotomic KLR algebras in type $\mathtt{A}_\infty$.
We hence deduce the graded decomposition numbers and full submodule structures of all level one cyclotomic KLR algebras in type $\mathtt{C}_\infty$.
\end{abstract}

\section{Introduction}\label{sec:intro}

Introduced by Khovanov--Lauda and Rouquier, the KLR algebras -- or quiver Hecke algebras -- and their cyclotomic quotients categorify the negative halves of quantum groups and their irreducible highest weight modules, respectively.
These algebras are celebrated for their rich connections with categorical knot theory \cite{kl09,WebsterKnots,Webster2}, the geometry of quiver varieties and perverse sheaves \cite{Rouq,VV11}, crystal and polytope combinatorics \cite{lv11,TW}, and the transfer matrix algebras of statistical mechanics \cite{PR13,bowman17,LibedinskyPlaza20,BCH23}.

Thanks in part to the Brundan--Kleshchev isomorphism, the (finite and affine) type $\mathtt{A}$ cyclotomic KLR algebras have a richly developed structural theory.  
These   algebras possess graded cellular  structures which lift the classical theory of Specht modules and tableaux to the graded setting.
Over the complex field, the LLT algorithm allows us to calculate the graded  characters of simple modules \cite{bk09}.  
Over fields of positive characteristic, the graded  characters of simple modules can be rephrased in terms of $p$-Kazhdan--Lusztig polynomials \cite{el17,BCH23}.

Outside of type $\tt A$, almost nothing is known about the simple modules of the 
(cyclotomic) KLR algebras.  
Recent years have seen a surge of effort to understand the structure of type $\mathtt{C}$ KLR algebras in particular:
the conditions for their semisimplicity have been characterised \cite{ls18};
and their representation types are now known \cite{apc,typecwt1,ahswreptype};
a cellular theory of graded Specht modules has been developed in \cite{aps,mathas22}.

The purpose of this paper is to  construct  a graded Morita equivalence between the level one cyclotomic KLR algebras in type $\mathtt{C}_\infty$ and level two cyclotomic KLR algebras in type $\mathtt{A}_\infty$ and to explicitly determine where simple and Specht modules are sent under this equivalence (see \cref{cor:Morita+grdec}).
This equivalence is constructed as an explicit isomorphism between the latter algebra and an idempotent truncation of the former algebra that preserves the cellular structures of these algebras.
In general this equivalence sends a type $\mathtt{C}_\infty$ Specht module $S(\nu)$ for $\nu=\rho+(\pla,\bbmu)$ to the type $\mathtt{A}_\infty$ Specht module $S(\pla,\bbmu{\color{cyan}'})$, where $\rho$ is rectangular and is the unique partition in a defect zero block in type $\mathtt{C}_\infty$.
This combinatorics is illustrated in \cref{introfig}.

This allows us to  deduce that the graded decomposition matrices of level one cyclotomic KLR algebras in type $\mathtt{C}_\infty$ are characteristic-free 
and equal to anti-spherical ($p$-)Kazhdan--Lusztig polynomials for maximal finite parabolics of finite symmetric groups 
(for which we have explicit combinatorial formulae   \cite{bsIII,Lyle24coreblocks}).
In fact, by \cite{BDHS,BDDHMS,BDDHMS2} we can deduce the full 
 ${\rm Ext}$-quiver presentations of the basic algebras of these KLR algebras, and indeed we can visualise the complete submodule structure of an arbitrary Specht module in terms of its strong Alperin diagram.

Our Morita equivalence unveils a remarkable `folding phenomenon' linking KLR algebras of type $\mathtt{C}$ and type $\mathtt{A}$.
In the level one, type $\mathtt{C}_\infty$ case, we give a complete and explicit characterisation of this phenomenon. 
We predict that extending this phenomenon to higher levels and affine types will be the crux of further understanding of the KLR algebras of type $\mathtt{C}$, however we expect that this will be a very difficult problem.
\begin{figure}[h!]
\[
\scalefont{0.8}
    \begin{tikzpicture}[scale=1]
     \clip(-5,-1.4) rectangle (5,-9.6);  
    
    \draw[very thick,fill=gray!30] (0,-3*0.5)--++(180:6*0.5)--++(-90:9*0.5)--++(0:6*0.5) coordinate (X)--(0,0-3*0.5);   
    
    \draw[very thick](X)--++(180:0.5)--++(90:0.5)     coordinate (X) --++(0:0.5)--++(-90:0.5) coordinate (Y);
\path(Y)--++(180:0.25)--++(90:0.25)     node {$0$};
    \draw[very thick](X)--++(180:0.5)--++(90:0.5)     coordinate (X) --++(0:0.5)--++(-90:0.5) coordinate (Y);
\path(Y)--++(180:0.25)--++(90:0.25)     node {$0$};

\path(X)--++(90:0.25)--++(180:0.25) coordinate (Y) node {$\cdot$}
 --++(90:0.25)--++(180:0.25) coordinate (X);
 
 \path   (Y) --++(135:0.1) node {$\cdot$}--++(-45:0.2) node {$\cdot$};

    \draw[very thick](X)--++(180:0.5)--++(90:0.5)     coordinate (X) --++(0:0.5)--++(-90:0.5) coordinate (Y) --++(180:0.25);
\path(Y)--++(180:0.25)--++(90:0.25)     node {$0$};

  \draw[very thick](X)--++(180:0.5)--++(90:0.5)     coordinate (X) --++(0:0.5)--++(-90:0.5) coordinate (Y) --++(180:0.25);
\path(Y)--++(180:0.25)--++(90:0.25)     node {$0$};

  \draw[very thick](X)--++(180:0.5)--++(90:0.5)     coordinate (X) --++(0:0.5)--++(-90:0.5) coordinate (Y) --++(180:0.25);
\path(Y)--++(180:0.25)--++(90:0.25)     node {$0$};  

 \draw [thick, decorate,decoration={brace,amplitude=5pt,mirror,}]
   (-3.2,-3.1)-- (-3.2,-5.9) ; 
\draw (-3.6,-4.5) node {$a_0$};

 \draw [thick, decorate,decoration={brace,amplitude=5pt,mirror,}]
   (-3.2,-0.1-1.5)-- (-3.2,-2.9) ; 
\draw (-3.6,-2.25) node {$\kapc$};

   \draw[very thick,fill=cyan!30] (0,-6)--++(-90:1*0.5)--++(-180:1*0.5) 
   --++(-90:1*0.5)--++(-180:2*0.5) 
      --++(-90:1*0.5) coordinate (X)
            --++(-90:1*0.5)
      --++(-180:1*0.5)    --++(-90:1*0.5)--++(-180:1*0.5)    --++(-90:2*0.5)--++(-180:1*0.5) --++(90:3.5)--(0,-6);
   ;

    \draw[very thick](X)--++(180:0.5)--++(90:0.5)     coordinate (X) --++(0:0.5)--++(-90:0.5) coordinate (Y) ;
\path(Y)--++(180:0.25)--++(90:0.25)     node {$\kappa_2$};

\path(X)--++(90:0.25)--++(180:0.25) coordinate (Y) node {$\cdot$}
 --++(90:0.25)--++(180:0.25) coordinate (X);
 
 \path   (Y) --++(135:0.1) node {$\cdot$}--++(-45:0.2) node {$\cdot$};

     \draw[very thick](X)--++(180:0.5)--++(90:0.5)     coordinate (X) --++(0:0.5)--++(-90:0.5) coordinate (Y) --++(180:0.5);
\path(Y)--++(180:0.25)--++(90:0.25)     node {$\kappa_2$};

   \draw[very thick,fill=magenta!30] (0,0-3*0.5)--++(0:4)--++(-90:1*0.5)--++(-180:1*0.5) 
   --++(-90:1*0.5)--++(-180:2*0.5) 
      --++(-90:1*0.5)  
            --++(-90:1*0.5)
      --++(-180:1*0.5)  coordinate (X)  --++(-90:1*0.5) --++(-180:1*0.5)    --++(-90:2*0.5)--++(-180:1*0.5)  
  --++(-90:1*0.5)--++(-180:1*0.5)        
  --++(-90:1*0.5)--++(-180:1*0.5)    --++(90:4.5) ;

    \draw[very thick](X)--++(180:0.5)--++(90:0.5)     coordinate (X) --++(0:0.5)--++(-90:0.5) coordinate (Y);
\path(Y)--++(180:0.25)--++(90:0.25)     node {$\kappa_1$};

\path(X)--++(90:0.25)--++(180:0.25) coordinate (Y) node {$\cdot$}
 --++(90:0.25)--++(180:0.25) coordinate (X);
 
 \path   (Y) --++(135:0.1) node {$\cdot$}--++(-45:0.2) node {$\cdot$};

     \draw[very thick](X)--++(180:0.5)--++(90:0.5)     coordinate (X) --++(0:0.5)--++(-90:0.5) coordinate (Y) --++(180:0.5);
\path(Y)--++(180:0.25)--++(90:0.25)     node {$\kappa_1$};

     \draw[very thick](X)--++(180:0.5)--++(90:0.5)     coordinate (X) --++(0:0.5)--++(-90:0.5) coordinate (Y);
\path(Y)--++(180:0.25)--++(90:0.25)     node {$\kappa_1$};
\end{tikzpicture}
\]

\caption{All the combinatorics of our main theorem is illustrated in this figure and can be found in \cref{sec:background}.
 An arbitrary partition $\nu=\rho+(\pla,\bbmu )$ labelling a Specht module of the level 1 type $\ttc_\infty$ KLR algebra.
In grey we highlight the rectangular subpartition $\rho$ and in pink and blue we highlight the partitions $\pla$ and $\bbmu$ which label a Specht module $S(\pla,\bbmu{\color{cyan}'})$ of the level 2 type $\tta_\infty$ KLR algebra.
The charges for the Specht modules are $\kapc \in \bbz_{\geq 0}$ and $({\color{magenta}\kappa_1},{\color{cyan}\kappa_2}) \in \bbz_{>0}^2$, respectively.
}
\label{introfig}
\end{figure}

\begin{Acknowledgements*}
Firstly, we would like to thank Andrew Mathas for helpful comments.  
This work was funded by EPSRC grant EP/V0090X/1, JSPS Kakenhi grants 23K03043 and 23K12964, and the Royal Society.
This project started while the first author was visiting OIST as part of OIST's Theoretical Sciences Visiting Program and concluded while all four authors were at the ICERM program {``Categorification and Computation in Algebraic Combinatorics''} (supported by the National Science Foundation under Grant No.~DMS-1929284).
We thank the referee for their helpful comments.

%

\end{Acknowledgements*}

\section{Background}\label{sec:background}

\subsection{Cyclotomic Khovanov--Lauda--Rouquier algebras}\label{subsec:klr}

We adopt standard notation from \linebreak \cite{Kac} for the root datum of types
 $\mathtt{A}_\infty$ and $\mathtt{C}_\infty$, that is $\fkg = \fks\fkl_\infty$ or $\fks\fkp_\infty$.
We set $I = I_\fkg = I_\tta:=\bbz$ if $\fkg=\fks\fkl_\infty$, or $I = I_\fkg = I_\ttc := \bbz_{\geq0}$ if $\fkg=\fks\fkp_\infty$.
We orient our Dynkin diagrams as in \cref{dynkin}.
In particular, we have {\sf simple roots} $\{\alpha_i \mid i\in I\}$, {\sf simple coroots} $\{\alpha^\vee \mid i \in I\}$, and we have {\sf fundamental weights} $\{\Lambda_i \mid i \in I\}$ in the {\sf weight lattice} $\sfp$.
We let $\sfq^+:= \bigoplus_{i\in I} \bbz_{\geq 0} \alpha_i$ be the {\sf positive cone of the root lattice} and $\sfp^+:= \{\La\in \sfp\mid \langle \alpha^\vee_i, \La \rangle \geq 0 \text{ for all } i\in I\}$ the {\sf positive weight lattice}, where $\langle - , -\rangle$ is the natural pairing (i.e.~$\langle \alpha^\vee_i, \Lambda_j\rangle = \delta_{ij}$).
There is also an invariant symmetric bilinear form $( -, - )$ on $\sfp$ satisfying $(\alpha_i, \Lambda_j) = d_i \delta_{ij}$ and $(\alpha_i, \alpha_j) = d_i a_{ij}$, with $d=(1,1,\dots)$ if $\fkg=\fks\fkl_\infty$, and $d=(2,1,1,\dots)$ if $\fkg=\fks\fkp_\infty$.  
We say that $\beta = \sum_{i\in I} a_i \alpha_i \in \sfq^+$ has {\sf height} $\operatorname{ht}(\beta) = \sum_{i\in I} a_i$, and $\La = \sum_{i\in I} b_i \La_i \in \sfp^+$ has {\sf level} $\sum_{i\in I} b_i$.
Set $\sfq^+_n:=\{\beta \in \sfq^+ \mid \operatorname{ht}(\beta) = n\}$.
For any $\beta \in \sfq^+$ of height $n$, we set $I^\beta = \{\bfi \in I^n \mid \alpha_{i_1} + \dots + \alpha_{i_n} = \beta\}$.
The symmetric group $\sss$ acts on elements of $I^n$ by place permutation.
This paper  broadly follows the classical notations and definitions of \cite[Sections 2 and 3]{kmr}, but with additional marker $\mathfrak{g}$ to differentiate between our wider family of types.

\begin{figure}[ht!]
\[
  \begin{tikzpicture}[scale=0.55]
 \draw[very thick]( -1,3)--++(180:0.8); 
 \draw[very thick]( -4,3)--++(180:0.8)
coordinate(hi);
\draw[very thick]( -4,3)--++(0:0.8)
coordinate(hi2);
 
\foreach \i in  {0,2,4,6,8,10}{
 \path(hi)--++(0:0+\i) node {\scalefont{1.5}$\mathbf <$}; 
}
 
 \draw[very thick] ( -7,3)--( 7,3); 
  \draw[very thick,densely dotted]  ( 7,3)--++(0:0.85);; 
  \draw[very thick,densely dotted]  ( -7,3)--++(180:0.85);; 

 \path ( -8,3) coordinate (hi) circle (6pt); 

\draw[very thick, fill=black ] ( -6,3) coordinate (hi) circle (6pt); 
 \path(hi)--++(-90:0.62) node {\scalefont{0.9} $-3$}; 

\draw[very thick, fill=black ]  ( -4,3) coordinate (hi) circle (6pt); 
 \path(hi)--++(-90:0.62) node {\scalefont{0.9} $-2$}; 

\draw[very thick, fill=black ]  ( 0,3) coordinate (hi) circle (6pt); 
 \path(hi)--++(-90:0.62) node {\scalefont{0.9} $0$}; 

\draw[very thick, fill=black ]  (  2,3) coordinate (hi) circle (6pt); 
 \path(hi)--++(-90:0.62) node {\scalefont{0.9} $1$}; 

\draw[very thick, fill=black ]  ( -2,3) coordinate (hi) circle (6pt); 
 \path(hi)--++(-90:0.62) node {\scalefont{0.9} $-1$}; 
 
\draw[very thick, fill=black ]  ( 4,3) coordinate (hi) circle (6pt); 
 \path(hi)--++(-90:0.62) node {\scalefont{0.9} $2$}; 
 
\draw[very thick, fill=black ]  ( 6,3) coordinate (hi) circle (6pt); 
 \path(hi)--++(-90:0.62) node {\scalefont{0.9} $3$}; 
 
\end{tikzpicture}
\]
\[
  \begin{tikzpicture}[scale=0.55]
 \draw[very thick]( -1,3)--++(180:0.8); 
 \draw[very thick]( -4,3)--++(180:0.8)
coordinate(hi);
\draw[very thick]( -4,3)--++(0:0.8)
coordinate(hi2);
 \draw[very thick] ( -6,3)--( 7,3); 
 \draw[very thick,densely dotted]  ( 7,3)--++(0:0.85);; 
 \draw[very thick] ( -6,3.1)--( -8,3.1); 
 \draw[very thick] ( -6,2.9)--( -8,2.9); 

\draw[very thick, fill=black ] ( -8,3) coordinate (hi) circle (6pt); 
 \path(hi)--++(-90:0.62) node {\scalefont{0.9} $0$}; 
 \path(hi)--++(0:1) node {\scalefont{2}$\mathbf >$}; 

\foreach \i in  {0,2,4,6,8,10}{
 \path(hi)--++(0:3+\i) node {\scalefont{1.5}$\mathbf <$}; 
}

\draw[very thick, fill=black ] ( -6,3) coordinate (hi) circle (6pt); 
 \path(hi)--++(-90:0.62) node {\scalefont{0.9} $1$}; 

\draw[very thick, fill=black ]  ( -4,3) coordinate (hi) circle (6pt); 
 \path(hi)--++(-90:0.62) node {\scalefont{0.9} $2$}; 

\draw[very thick, fill=black ]  ( 0,3) coordinate (hi) circle (6pt); 
 \path(hi)--++(-90:0.62) node {\scalefont{0.9} $4$}; 

\draw[very thick, fill=black ]  (  2,3) coordinate (hi) circle (6pt); 
 \path(hi)--++(-90:0.62) node {\scalefont{0.9} $5$}; 

\draw[very thick, fill=black ]  ( -2,3) coordinate (hi) circle (6pt); 
 \path(hi)--++(-90:0.62) node {\scalefont{0.9} $3$}; 
 
\draw[very thick, fill=black ]  ( 4,3) coordinate (hi) circle (6pt); 
 \path(hi)--++(-90:0.62) node {\scalefont{0.9} $6$}; 
 
\draw[very thick, fill=black ]  ( 6,3) coordinate (hi) circle (6pt); 
 \path(hi)--++(-90:0.62) node {\scalefont{0.9} $7$}; 

\end{tikzpicture}
\]
\caption{The Dynkin diagrams of types $\mathtt{A}_\infty$ (above) and $\mathtt{C}_\infty$ (below).}
\label{dynkin}
\end{figure}

For a field $\bbf$, and $\beta \in \sfq^+$ of height $n$, the {\sf Khovanov--Lauda--Rouquier (KLR) algebra} $\scrr_\beta = \scrr_\beta(\fkg)$ is the unital associative $\bbf$-algebra with generators
\[
\{e(\bfi) \mid \bfi \in I^\beta\} \cup \{y_1, \dots, y_n\} \cup \{\psi_1, \dots, \psi_{n-1}\},
\]
subject to the following relations.
{\allowdisplaybreaks
\begin{alignat}{2}
e(\bfi)e(\bfj)&=\delta_{\bfi, \bfj} e(\bfi); \! \qquad\qquad\qquad\qquad\qquad & y_r e(\bfi) &= e(\bfi) y_r;\nonumber\\
\sum_{\bfi \in I^\beta} e(\bfi)&=1;  & y_r y_s &= y_s y_r;\nonumber\\
\psi_r e(\bfi) &= e(s_r\bfi) \psi_r; & \psi_r y_s &= \mathrlap{y_s \psi_r}\hphantom{\smash{\psi_s\psi_r}} \quad \text{if } s\neq r,r+1;\\
y_r \psi_r e(\bfi) &=(\psi_r y_{r+1} - \delta_{i_r,i_{r+1}})e(\bfi); & \psi_r \psi_s &= \psi_s\psi_r \quad \text{if } |r-s|>1;\label{rel:dotcrossbun1}\\
y_{r+1} \psi_r e(\bfi) &=(\psi_r y_r + \delta_{i_r,i_{r+1}})e(\bfi)\label{rel:dotcrossbun2};
\end{alignat}
}
\vspace{-3ex}
{\allowdisplaybreaks
\begin{align}
\psi_r^2 e(\bfi)&=\begin{cases}
\mathrlap0\phantom{(\psi_{r+1}\psi_r\psi_{r+1}+y_r+y_{r+2})e(\bfi)}& \text{if }i_r=i_{r+1},\\
e(\bfi) & \text{if }i_{r+1}\neq i_r, i_r\pm1,\\
(y_{r+1} - y_r) e(\bfi) & \text{if }i_r \rightarrow i_{r+1},\\
(y_r - y_{r+1}) e(\bfi) & \text{if }i_r\leftarrow i_{r+1},\\
(y_r - y_{r+1}^2) e(\bfi) & \text{if }i_r\Rightarrow i_{r+1};\\
(y_r^2 - y_{r+1}) e(\bfi) & \text{if }i_r\Leftarrow i_{r+1};
\end{cases}\label{rel:quadr}\\
\psi_r\psi_{r+1}\psi_re(\bfi)&=\begin{cases}
(\psi_{r+1}\psi_r\psi_{r+1}+1)e(\bfi)& \text{if }i_{r+2}=i_r\rightarrow i_{r+1},\\
(\psi_{r+1}\psi_r\psi_{r+1}-1)e(\bfi)& \text{if }i_{r+2}=i_r\leftarrow i_{r+1} \text{ or } i_{r+2}=i_r\Rightarrow i_{r+1},\\
(\psi_{r+1}\psi_r\psi_{r+1}+y_r+y_{r+2})e(\bfi)& \text{if }i_{r+2}=i_r\Leftarrow i_{r+1},\\
(\psi_{r+1}\psi_r\psi_{r+1})e(\bfi)& \text{otherwise.}
\end{cases}\label{rel:braid}
\end{align}
As is standard in the literature, our notation $i \rightarrow j$ indicates that there is an arrow from $i$ to $j$ in the Dynkin diagram (\cref{dynkin}), while $i \Rightarrow j$ indicates that there is a double edge from $i$ to $j$ in the Dynkin diagram, which only happens in type $\mathtt{C}_\infty$ with $i=0$ and $j=1$.

We have natural inclusion maps $\scrr_ \beta(\fkg) \otimes \scrr _\gamma(\fkg) \hookrightarrow \scrr _{\beta+\gamma}(\fkg)$ for any $\beta, \gamma \in \sfq^+$, and will often abuse notation and write $x\otimes y$ for the images in $\scrr _{\beta+\gamma}(\fkg)$ of elements under such inclusions.
These algebras have cyclotomic quotients, which are our primary interest here.
For $\Lambda \in \sfp^+$, the {\sf cyclotomic KLR algebra} $\scrr^\Lambda_\beta = \scrr^\Lambda_\beta(\fkg)$ is the quotient of $\scrr_\beta$ by the additional {\sf cyclotomic relations}
\[
y_1^{\langle \alpha^\vee_{i_1} , \Lambda \rangle} e(\bfi) = 0 \text{ for all } \bfi\in I^\beta.
\]
We set $\scrr^\La(\fkg) = \bigoplus_{\beta\in \sfq^+} \scrr^\La_\beta(\fkg)$ noting that there are natural homomorphisms $\scrr^\La_\beta(\fkg) \rightarrow \scrr^\La_{\beta+\gamma}(\fkg)$ for any $\beta, \gamma \in \sfq^+$.
The KLR algebras and their cyclotomic quotients may be $\bbz$-graded by
\[
\deg e(\bfi) = 0, \qquad \deg y_r e(\bfi) = (\alpha_{i_r}, \alpha_{i_r}), \qquad \deg \psi_r e(\bfi) = (\alpha_{i_r}, \alpha_{i_{r+1}}),
\]
where $(- ,-)$ is the invariant symmetric bilinear form on $\sfp$. 
We associate the following braid diagrams to the KLR generators in the usual fashion.
We associate braid diagrams to the KLR generators  in the usual fashion, as illustrated in \cref{usualfashion}. The KLR relations can be entirely rewritten with this diagrammatic calculus, as in \cite{kl09,kl11}, for example.

\begin{figure}[ht!]
\[
\begin{tikzpicture}[scale=0.33,yscale=1]
\draw[  rounded corners](-1,0) rectangle (6.5*2,5);
\foreach \i in {0,2,3,4,6}
{\fill (\i*2,0) circle (3pt);
\fill (\i*2,5) circle (3pt);
\path(\i*2,5) coordinate (X\i);
\path(\i*2,0) coordinate (Y\i);
 }
\draw (Y0) node [below] {\scalefont{0.7}$i_1$}; 
\draw (Y2) node [below] {\scalefont{0.7}$i_{r\text{--}1}$}; 
\draw (Y3) node [below] {\scalefont{0.7}$i_{r}$}; 
\draw (Y4) node [below] {\scalefont{0.7}$i_{r\text{+}1}$}; 
 \draw (Y6) node [below] {\scalefont{0.7}$i_{n}$};  

\draw  (0*2,0) --++(90:5);
\draw  (2*2,0) --++(90:5);
 \draw  (6*2,0) --++(90:5);
 
  \draw  (3*2,0) --(3*2,5);
    \draw  (4*2,0) --(4*2,5);

 \draw [thick, white, densely dotted] (1,0)--(3,0);
  \draw [thick, white, densely dotted] (1,5)--(3,5);
   \draw [thick, white, densely dotted] (11-2,0)--(13-2,0);
      \draw [thick, white, densely dotted] (11-2,5)--(13-2,5);

\end{tikzpicture}
\qquad
\begin{tikzpicture}[scale=0.33,yscale=1]
\draw[  rounded corners](-1,0) rectangle (7.5*2,5);
\foreach \i in {0,2,3,4,5,7}
{\fill (\i*2,0) circle (3pt);
\fill (\i*2,5) circle (3pt);
\path(\i*2,5) coordinate (X\i);
\path(\i*2,0) coordinate (Y\i);
 }
\draw (Y0) node [below] {\scalefont{0.7}$i_1$}; 
\draw (Y2) node [below] {\scalefont{0.7}$i_{r\text{--}1}$}; 
\draw (Y3) node [below] {\scalefont{0.7}$i_{r}$}; 
\draw (Y4) node [below] {\scalefont{0.7}$i_{r\text{+}1}$}; 
\draw (Y5) node [below] {\scalefont{0.7}$i_{r\text{+}2}$}; 
\draw (Y7) node [below] {\scalefont{0.7}$i_{n}$};  

\draw  (0*2,0) --++(90:5);
\draw  (2*2,0) --++(90:5);
\draw  (5*2,0) --++(90:5);
 \draw  (7*2,0) --++(90:5);
 
  \draw  (3*2,0) --(4*2,5);
    \draw  (4*2,0) --(3*2,5);

 \draw [thick, white, densely dotted] (1,0)--(3,0);
  \draw [thick, white, densely dotted] (1,5)--(3,5);
   \draw [thick, white, densely dotted] (11,0)--(13,0);
      \draw [thick, white, densely dotted] (11,5)--(13,5);
\end{tikzpicture}
\qquad
\begin{tikzpicture}[scale=0.33,yscale=1]
\draw[  rounded corners](-1,0) rectangle (6.5*2,5);
\foreach \i in {0,2,3,4,6}
{\fill (\i*2,0) circle (3pt);
\fill (\i*2,5) circle (3pt);
\path(\i*2,5) coordinate (X\i);
\path(\i*2,0) coordinate (Y\i);
 }
\draw (Y0) node [below] {\scalefont{0.7}$i_1$}; 
\draw (Y2) node [below] {\scalefont{0.7}$i_{r\text{--}1}$}; 
\draw (Y3) node [below] {\scalefont{0.7}$i_{r}$}; 
\draw (Y4) node [below] {\scalefont{0.7}$i_{r\text{+}1}$}; 
 \draw (Y6) node [below] {\scalefont{0.7}$i_{n}$};  

\draw  (0*2,0) --++(90:5);
\draw  (2*2,0) --++(90:5);
 \draw  (6*2,0) --++(90:5);
 
  \draw  (3*2,0) --(3*2,5);
    \draw  (4*2,0) --(4*2,5);

 \draw [thick, white, densely dotted] (1,0)--(3,0);
  \draw [thick, white, densely dotted] (1,5)--(3,5);
   \draw [thick, white, densely dotted] (11-2,0)--(13-2,0);
      \draw [thick, white, densely dotted] (11-2,5)--(13-2,5);

     \draw[thick, fill] (6,2.5) circle (8pt); 
\end{tikzpicture}
\]
\caption{The diagrammatic visualisation of the elements $e(\bfi)$, $\psi_r e(\bfi )$, $y_re(\bfi )$ for $\bfi=(i_1,\dots i_n)$.}
\label{usualfashion}
\end{figure}

\begin{rmk}
Strictly speaking, we have made a choice of certain polynomials in our definition of the KLR algebras.
If $\bbf$ is a quadratically closed field (i.e.~it contains the square roots of all its elements), then all choices of polynomials yield isomorphic algebras in type $\mathtt{C}$, so we have lost nothing in making this choice.
In type $\mathtt{A}$, there are non-isomorphic algebras obtained by different choices of polynomial -- see the discussion before Lemma~2.2 in~\cite{apa1}.
Our chosen polynomials are the most common in the literature, and ensure that the well-known `Brundan--Kleshchev isomorphism' applies.
\end{rmk}

\subsection{$\ell$-partitions and residues}\label{subsec:multis}
We now recall the combinatorics of partitions and the associated theories of residues in both types 
$\tt A_\infty$ and $\tt C_\infty$.

\begin{defn}
For $n\geq 0$, a {\sf partition} of $n$ is a weakly decreasing sequence of non-negative integers $\la = (\la_1, \la_2, \dots)$ such that the sum $|\la|=\la_1+\la_2+\cdots$ is equal to $n$.
We write $\varnothing$ for the unique partition of 0 and we let $\la'$ denote the conjugate partition of $\la$ determined by 
$\la_i’ = |\{j \mid \la_j \geq i\}|$ for $i \geq 1$.
Note that we will in general omit trailing zeroes from partitions.
An  $\ell$-{\sf partition} of $n$ is an $\ell$-tuple of partitions $\bla = (\la^{(1)}, \dots, \la^{(\ell)})$ such that the total size is $\sum_{i=1}^\ell |\la^{(i)}| = n$.
We also write $\varnothing$ for the unique $\ell$-partition of 0.
We denote the set of $\ell$-partitions of $n$ by $\Par^\ell_n$ and set $\Par^\ell= \cup_{n\geq 0} \Par^\ell_n$.
\end{defn}

This paper will mostly look at the cases $\ell = 1$ or $2$.

\begin{defn}
Let $\rho \in \Par^1$ be a rectangular partition (i.e.~$\rho = (a^b)$ for some $a,b$).
For  $\la  $ a partition, we let  $L(\la) $ denote the number of non-zero rows in $\la$. 
For $\pla \in \Par^1$ with $L(\rho)\geq L(\pla)$, we define $\rho + \pla = (\rho_1 + \pla_1, \rho_2 + \pla_2,\dots) \in \Par^1_{|\rho| + |\pla|}$.
For $(\pla, \bbmu) \in \Par^2$ with $L(\rho) \geq L(\pla)$, we define 
\[
\rho + (\pla, \bbmu) = (\rho_1 + \pla_{\color{magenta}1}, \rho_2 + \pla_{\color{magenta}2}, \dots, \rho_{L(\rho)} + \pla_{{\color{magenta}L(\rho)}}, \bbmu_{\color{cyan}1}, \bbmu_{\color{cyan}2}, \dots, \bbmu_{{\color{cyan}L(\bbmu)}}) \in \Par^1_{|\rho| + |\pla| + |\bbmu|}.
\]
This is illustrated in \cref{introfig}.
\end{defn}

If $\bla$ and $\bmu$ are $\ell$-partitions of $n$, we say that $\bla$ {\sf dominates} $\bmu$, and write $\bla\dom\bmu$ if
\[
|\la^{(1)}| + \dots + |\la^{(m-1)}| + \sum_{j=1}^{r} \la^{(m)}_j \geq |\mu^{(1)}| + \dots + |\mu^{(m-1)}| + \sum_{j=1}^{r} \mu^{(m)}_j
\]
for all $1\leq m \leq \ell$ and $r\geq 1$.
For any $\ell$-partition $\bla$, we define its {\sf Young diagram} $[\bla]$ to be the set
\[
\{[r,c,m] \in \bbz_{\geq 1} \times \bbz_{\geq 1} \times \{1,\dots,\ell\} \mid c\leq \lambda^{(m)}_r\}.
\] For the purposes of graded tableau combinatorics, we will require the notion of a {\sf multicharge} $\kappa = (\kappa_1, \ldots, \kappa_\ell) \in \bbz^\ell$.
In type $\mathtt{A}_\infty$ (resp.~$\mathtt{C}_\infty$), we have an associated dominant weight $\La
 = \La_\kappa = \La_{\kappa_1} + \dots + \La_{\kappa_\ell}$ (resp.~$\La_{|\kappa_1|} + \dots + \La_{|\kappa_\ell|}$) for each multicharge. 
Let $\bla$ be an $\ell$-partition.
Then to any node $A=[r,c,m] \in [\bla]$ we may associate its ($\fkg$-){\sf residues}
\begin{align}\label{labelforchrispy}
\res_{\mathtt{A}_\infty} (A) =  \kappa_m +c- r 
\qquad \text{and} \qquad
\res_{\mathtt{C}_\infty}  (A) =  |  \kappa_m + c - r |.
\end{align}
We also associate ($\fkg$-){\sf contents} to the $\ell$-partition $\bla$
\begin{align}
\mathsf{cont}_{\mathtt{A}_\infty}  (\bla) =  \sum_{A\in [\bla]} \alpha_{\res_{\mathtt{A}_\infty}  (A) }
\qquad \text{and} \qquad
\mathsf{cont}_{\mathtt{C}_\infty}  (\bla) =  \sum_{A\in [\bla]} \alpha_{\res_{\mathtt{C}_\infty}  (A) }.
\end{align}
In this paper we will be mostly concerned with \emph{bipartitions} indexing Specht modules for the level 2 cyclotomic KLR algebras $\scrr^{\La_{{\color{magenta}\kappa_1}}+\La_{{\color{cyan}\kappa_2}}}_{\beta - \omega}(\mathfrak{sl}_\infty)$, for a bicharge $({\color{magenta}\kappa_1},{\color{cyan}\kappa_2}) \in \bbz^2$, and \emph{partitions} indexing Specht modules for the level 1 cyclotomic KLR algebras $\scrr^{\La_{\kapc}}_{\beta}(\mathfrak{sp}_\infty)$, for $\kapc \in \bbz$ a charge.
Here, $\omega \in \sfq^+$ will denote $\mathsf{cont}_{\mathtt{C}_\infty} (\rho)$, where $\rho$ will be of the form of the grey rectangle in \cref{introfig}.
Thus we set $\ParblockA[\beta]$ to be the set of bipartitions $\bla$ such that $\mathsf{cont}_{\mathtt{A}_\infty}(\bla) = \beta$, and set $\ParblockC[\beta]$ to be the set of partitions $\la$ such that $\mathsf{cont}_{\mathtt{C}_\infty}(\la) = \beta$, where we are taking $\beta \in \sfq^+$ for the corresponding types.
We will assume throughout that $\kapc \geq 0$ -- choosing $\kapc<0$ would result in the same algebra $\scrr^{\La_{|\kapc|}}_{\beta}(\mathfrak{sp}_\infty)$, but with slightly different residue combinatorics.
All of our arguments go through with minimal changes, but our choice makes for a more streamlined exposition.

If $\res(A) = i$, we call $A$ an $i$-node for $i \in I$.
We say that a node $A$ is {\sf removable} (resp.\ {\sf addable}) if $[\bla]\setminus \{A\}$ (resp.\ $[\bla]\cup \{A\}$) is a valid Young diagram for an $\ell$-partition of $n-1$ (resp.\ $n+1$).
We let $\Rem_i(\bla)$ denote the set of removable $i$-nodes of $\bla$, and similarly denote by $\Add_i(\bla)$ the set of addable $i$-nodes of $\bla$.
For $i,j\in I_{\ttc}$ we write
\[
i\nearrow j := (i,i+1,i+2,\dots, j-1,j),
\qquad
j\searrow i := (j,j-1,j-2,\dots, i+1,i)
\]
and we let 
\[
i\searrow\nearrow j := (i,i-1,\dots, 2, 1, 0,1,2, \dots, j-1,j).
\]
 By convention, if $i>j$, then $i\nearrow j$ and $j\searrow i$ are empty.
Given two multisets of residues $R_1\subseteq I_\fkg^{r_1}$ and $R_2\subseteq I_\fkg^{r_2}$ we say that $R_1$ and $R_2$ are {\sf well separated} if no residue from $R_1$ is adjacent or equal to any residue from $R_2$.

%

\subsection{Standard tableaux and Specht modules}\label{subsec:tableaux}

Let $\bla$ be an $\ell$-partition of $n$.
A {\sf $\bla$-tableau} is a bijection $\ttt :[\bla] \rightarrow \{1,\dots,n\}$.
We depict $\ttt$ by filling each node $[r,c,m]\in [\bla]$ with $\ttt[r,c,m]$.
We say that a tableau $\ttt$ is {\sf row-strict} if the entries increase along the rows of each component of $\ttt$, and {\sf column-strict} if the entries increase down the columns of each component of $\ttt$.
If $\ttt$ is both row- and column-strict, we call it {\sf standard}.
We denote the set of standard $\bla$-tableaux by $\Std(\bla)$.
Note that the symmetric group $\sss$ acts naturally on the left on the set of tableaux.
For each $\bla$-tableau $\ttt$, we have the associated residue sequence
\begin{align}\label{analgoty}
\bfi^\ttt := (  \res_\fkg { \ttt^{-1}(1) } ,\   \res_\fkg { \ttt^{-1}(2)} , \dots,  \res_\fkg { \ttt^{-1}(n)} ).
\end{align}
We set $e_\ttt = e(\bfi^\ttt)$.
For $\bfi \in I^\beta$, we also set $\Std(\bfi)= \Std_\fkg(\bfi) = \{\ttt \in \Std(\bla) \mid \bla \in \Par^\fkg_\beta \text{ and } \bfi^\ttt = \bfi\}$.
In other words, $\Std(\bfi)$ is the set of standard tableaux of any shape with residue sequence $\bfi$.

Let $\ttt^{\bla}$ be the {\sf (row-)initial tableau}, which is the distinguished tableau where we fill the nodes with $1,\dots, n$ first along successive rows in $\la^{(1)}$, then $\la^{(2)}$, and so on, ending with $\la^{(\ell)}$.
Then for a $\bla$-tableau $\ttt$, we define the permutation $w^\ttt$ by $w^\ttt \ttt^{\bla} = \ttt$.
We set $\bfi^{\bla} = \bfi^{\ttt^{\bla}}$ and $\bfi_{\bla} = \bfi^{\ttt_{\bla}}$.
Fixing a reduced expression $w^\ttt = s_{i_1} \dots s_{i_r}$, we define $\psi_\ttt = \psi_{i_1} \dots \psi_{i_r} \in \scrr^\La_\beta(\fkg)$.
We also define the {\sf final tableau}, or {\sf column-initial tableau} $\ttt_\la$, for a \emph{partition} $\la$, to be the tableau where we fill the nodes instead in order down columns starting with the first column, then the second, and so on.

Let $\ttt$ be a $\bla$-tableau and $0\leq m\leq n$.
We denote by $\ttt_{\leq m} $ the set of nodes of $[\bla]$ whose entries are less than or equal to $m$.
If $\ttt\in \Std(\bla)$, then $\ttt_{\leq m}$ is a tableau for some $\ell$-partition, which we denote $\Shape(\ttt _{\leq m})$.

Specht modules for $\mathscr{R}_\beta^\La(\fkg)$ were constructed in \cite{bkw11,kmr} when $\fkg = \mathfrak{sl}_\infty$ (or $\mathfrak{\widehat{sl}}_e$),  and coincide with the cell modules constructed in \cite{hm10}.
When $\fkg = \mathfrak{sp}_\infty$ (or $\mathfrak{\widehat{sp}}_{2e}$), Specht modules were constructed in \cite{aps,mathas22}.

\begin{defn}\label{def:uglovv1}
Fix $\mathfrak{g}$ to be either $\mathtt{C}_{\infty}$ or $\mathtt{A}_{\infty}$ and $\kappa \in I_\fkg^\ell$.
Given $\bla \in \mptn \ell n$ and $i \in I_\fkg$, we define the {\sf $i$-sequence} of $\bla$ to be the sequence of addable and removable $i$-nodes (recorded by $a$ and $r$ respectively) from top-to-bottom of the Young diagram $[\bla]$, in order from component $1$ to component $\ell$.
We define the {\sf reduced $i$-sequence} to be the sequence of the form 
$a,a,\dots,a, r,r,\dots,r$ obtained from the $i$-sequence by repeatedly removing all consecutive pairs of the form $(r,a)$.
We say that a removable (resp.~addable) $i$-node of $\lambda$ is $\mathfrak{g}$-{\sf good} (resp.~$\mathfrak{g}$-{\sf cogood}) if it corresponds to the leftmost $r$ (resp.~rightmost $a$) in the reduced $i$-sequence.
For $\bfi = \bfi_\fkg = (i_1, i_2, \dots, i_m) \in I^m$, we write
\[
\bla  \xrightarrow { \ \bf i \ } \bmu
\]
if $\bmu$ can be obtained from $\bla$ by adding a sequence of cogood nodes of residues $i_1, i_2, \dots, i_m$ in order.
\end{defn}

We will drop the $\fkg$ and simply say that a node is good when uniformly handling $\mathtt{A}_\infty$ and $\mathtt{C}_\infty$.

\begin{defn}\label{uglovv}
Given a fixed $\kappa \in I_\fkg^\ell$, the set of {\sf Kleshchev} $\ell$-partitions $\mathscr{K}^\ell_\beta (\fkg)\subseteq \mathscr{P}_{\beta}^{\fkg}$ is defined recursively as follows.  
We have that $\varnothing\in \mathscr{K}^\ell_0(\fkg)$.  
For $\bla\in \mathscr{P}_{\beta}^{\fkg}$, we have that $\bla \in \mathscr{K}^\ell_\beta(\fkg)$ if and only if there exists  $i\in I_\fkg$ and a good $i$-node $A \in {\rm Rem}_i(\bla)$ such that $\bla\setminus \{A\}  \in \mathscr{K}^\ell_{\beta-\alpha_i}(\fkg)$.
\end{defn}

We now introduce the cellular bases which we will use to prove the isomorphism in \cref{thm:homomorphism,thm:isomorphism}.
Suppose $\ttt[r_k,c_k,m_k] = k$ for $1\leq k \leq n$ and $\res_\fkg[r_k,c_k,m_k] = i \in I_\fkg$.
Then we define
\[
{\deg}_{\ttt} ([r_k,c_k,m_k]) =  | \{A \mid A  \in \Add_i( \ttt_{\leq k-1})
  \text{ and } A \text{ is below }  [r_k,c_k,m_k]\}|,
\]
and we set 
\[
y_{\ttt}=\prod_{1\leq k \leq n} y_k^{{\deg}_{\ttt} ([r_k,c_k,m_k])}e_ {\ttt}. 
\]

\begin{thm}[{\cite[Theorem A]{mathas22}}]\label{thm:cellular}
Let $\fkg = \mathfrak{sl}_\infty$ or $\mathfrak{sp}_\infty$.
Then $\mathscr{R}_\beta^\La(\fkg)$ is a graded cellular algebra with cellular basis 
\begin{align}\label{skdjfhgdkljfghdfks}
\{c_{\tts\ttt} = \psi_\tts y_{\ttt^{\bla}}
\psi_\ttt^\ast \mid \tts,\ttt \in \Std(\bla), \bla \in \mptn \fkg \beta\} 
\end{align}
with respect to the ordering $\dom$ and the anti-involution $\ast$.
\end{thm}

Given $\bla \in \mptn \ell n$, we define the following left cell ideals in $\mathscr{R}_\beta^\La = \mathscr{R}_\beta^\La(\fkg)$.
\[
\mathscr{R}^{\dom \bla}_\beta  =  \mathscr{R}^\La_\beta y_{\ttt^{\bla}},
\qquad 
\mathscr{R}^{\rhd \bla}_\beta  =   
\mathscr{R}^{\dom \bla}_\beta \cap
\mathbb Z \{ c_{\tts \ttt} \mid 
\tts,\ttt \in \Std  ( \bmu ), \bmu \rhd \bla \}.
\]
The {\sf cell modules} coming from the above cellular basis are the {\sf Specht modules} $\spe\bla$, and can be described by an explicit homogeneous presentation -- see \cite[Definition~3.8 and Remark~3.6]{aps}.
 We recall that the cellular structure also allows us to define, for each $\bla \in \mptn \ell n$ a bilinear form 
$\langle\ ,\ \rangle^{\bla} _\Bbbk$ on $\spe \bla$ which is determined by
\[
c_{\tts \ttt} c_{\ttu \ttv}\equiv
\langle c _\ttt,c _\ttu \rangle^{\bla}  c_{\tts\ttv} \pmod{ \mathscr{R}^{\rhd \bla}}
\]
for any $\tts,\ttt,\ttu,\ttv \in \Std(\bla)$.

\begin{thm}{\cite[Theorem C]{mathas22}}
When $\bbf$ is a field, we have that 
\begin{align}
\{ \D\bla  = \spe \bla  / 
  \rad(\langle\ ,\ \rangle^{\bla})
  \mid  \bla  \in \mathscr{K}^\ell_\beta (\fkg)\}
\end{align}
is a complete and irredundant set of non-isomorphic simple $\scrr^\La_\beta(\fkg)$-modules.
\end{thm}

\begin{Remark}\label{remforzero}
It can be seen immediately from \cref{thm:cellular} that $e(\bfi) = 0$ if and only $\Std(\bfi) = \emptyset$.
This fact will be used several times without further reference.
\end{Remark}

\begin{lem}\label{lem:conjugatedomorder}
Let $\fkg = \mathfrak{sl}_\infty$ and $\La = \La_{{\color{magenta}\kappa_1}}+\La_{{\color{cyan}\kappa_2}}$.
We define a new ordering $\dom'$ on $\ParblockA[\beta]$ by $(\la,\mu) \dom ' (\alpha,\beta)$ if and only if $(\la,\mu') \dom (\alpha,\beta')$.
Then $\mathscr{R}_\beta^\La(\fkg)$ is a graded cellular algebra with cellular basis as in \eqref{skdjfhgdkljfghdfks} and with respect to the ordering $\dom'$ and the anti-involution $\ast$.
\end{lem}

\begin{proof}
For this proof only, we  will  require from  \cite[Definition 1.1 and 1.3]{bowman17} a residue-coarsened version of the dominance ordering.
We write $[r,c,m] <  [r',c',m']$ if 
$(i)$ $\kappa_m+c - r <  \kappa_{m'}+c' - r'$
 or $(ii)$  $\kappa_m+c - r =  \kappa_{m'}+c' - r'$ and 
$m' < m$.
We write 
 $[r,c,m] \blacktriangleleft  [r',c',m']$ if  $[r,c,m] <  [r',c',m']$ and $\res_{\mathtt{A}_\infty}([r,c,m]) = \res_{\mathtt{A}_\infty}([r',c',m'])$.   
In type ${\mathtt{A}_\infty}$ this order simplifies to 
 $[r,c,m] \blacktriangleleft  [r',c',m']$ if $m'<m$ and $\res_{\mathtt{A}_\infty}([r,c,m]) = \res_{\mathtt{A}_\infty}([r',c',m'])$, by definition of residue (see \eqref{labelforchrispy}). 
Now, we write $\la \geq \mu$ if there is a bijective map
$A : [\la]  \to [\mu]$ such that for each $[r,c,m] \in [\la]$, either 
$A([r,c,m]) <  [r,c,m]$ or 
$A([r,c,m]) = [r,c,m]$.
In \cite{bowman17,Webster} it is proven that the algebra $\mathscr{R}_\beta^\La(\fkg)$ is a graded cellular algebra with respect to the order $\blacktriangleright$.

Now, in level two type $\tt A$, it is clear from the definition that $\la \blacktriangleright \mu$ if and only if the latter is obtained from the former by moving nodes from the first component to the second component.
In particular, both $\dom$ and $\dom'$ are refinements of $\blacktriangleright$  and we remark that these two refinements only differ by comparing nodes within the latter component.
\end{proof}

\subsection{Semistandard tableaux}
For this section we let
$\nu = \rho +{\color{magenta}\la} \in \Par^1_n$, where $\rho$ is a rectangular partition and $L(\rho)\geq L({\color{magenta}\la})$, and we set $\ell=L(\rho) +L({\color{magenta}\la})$.
We define a {\sf semistandard tableau} of shape $\rho +{\color{magenta}\la}$ to be a map $\SSTT :[\nu] \rightarrow \{1,\dots,\ell \}$ whose entries weakly increase along the rows of $\SSTT$ and strictly increase down the columns of $\SSTT$.
 We depict $\SSTT$ by filling each node $[r,c]\in [\rho +{\color{magenta}\la}]$ with $\SSTT[r,c]$.
We denote the set of all semistandard $(\rho +{\color{magenta}\la}) $-tableaux by $\SStd ( \rho +{\color{magenta}\la})$.
We will only be interested in semistandard tableaux of a very particular form, namely those for which $\SSTT[r,c]=\SSTT[r ,c']$ if
$[r,c]$ and $[r,c']$ are both in $\rho$ or if they are both in $\color{magenta}\la$. 
We denote the subset of all such semistandard $\la$-tableaux by $\SStd_+( \rho +{\color{magenta}\la} )
\subseteq \SStd( \rho +{\color{magenta}\la} )$.  
Examples are depicted in \cref{bigeg}.
Note that the symmetric group $\sss[\ell]$ acts naturally on the left on the set $\SStd_+( \rho +{\color{magenta}\la})$.

We let $\SSTT^{\rho +{\color{magenta}\la} }\in \SStd_+ ( \rho +{\color{magenta}\la} )$  
(respectively $\SSTT_{\rho +{\color{magenta}\la} }\in \SStd_+ ( \rho +{\color{magenta}\la} )$) 
be the  distinguished semistandard tableau where we fill the nodes with $1,\dots, \ell$   along successive rows
(respectively columns)  in  
$\rho +{\color{magenta}\la} $.  
For $\SSTS,\SSTT\in \SStd_+( \rho +{\color{magenta}\la} )$ we define the permutation $w_\SSTS^\SSTT \in \sss[\ell]$ by $w_\SSTS^\SSTT \SSTS = \SSTT$. 
If $\SSTS = \SSTT^{ \rho +{\color{magenta}\la} }$, we set $w^\SSTT = w_\SSTS^\SSTT$.

For each $\SSTT\in \SStd_+ ( \rho + \pla)$ 
we have a unique tableau $\ttt\in \Std(\rho + \pla)$ such that 
$\ttt[r,c]\leq \ttt(r',c')$ if and only if $\SSTT[r,c]\leq \SSTT[r',c']$ and we let 
$\varphi: \SStd_+ ( \rho + \pla ) \to \Std  ( \rho + \pla)$ denote the injective map defined by $\varphi (\SSTT)= \ttt$.
For $s_k \in \sss[\ell]$ and $\SSTT= s_k \SSTS$, we have that $w^\SSTT_\SSTS$ is fully commutative (and so any two reduced expressions for  $w^\SSTT_\SSTS$ differ only by commutation relations); we let $\psi^\SSTT_\SSTS$ denote the unique (modulo commutation relations) lift of this element to the KLR algebra.
For $\SSTS,\SSTT \in \SStd_+ ( \rho + \pla)$ and a fixed reduced word $s_{i_1}s_{i_2}\dots s_{i_k} \in \sss[\ell]$ such that
$\SSTT= s_{i_1}s_{i_2}\dots s_{i_k}\SSTS$, we define $\psi^\SSTT_\SSTS$ in a similar fashion.
We set $e_\SSTT = e_{\varphi(\SSTT)}$ and $y_\SSTT =  y_{\varphi(\SSTT)}$.
We refine this by letting
\[
\SSTT^{-1}(k) = \{a_1, a_2, \dots, a_m\},
\]
where we order the $a_i \in \rho + \pla$ from left-to-right, setting $a_1 < a_2 <\dots <a_m$.
We define the residue sequence 
\[
\res(\SSTT^{-1}(k)) = (\res(a_1) , \res(a_2),\dots, \res(a_m) )
\]
and we define the degree
\[
\deg(\SSTT^{-1}(k)) =\textstyle \sum_{1\leq i\leq m }  {\deg}_{\varphi(\SSTT)} (a_i).
\]

\begin{rmk}
When we depict a diagram $\psi^\SSTT_\SSTS \in  \SStd_+ ( \rho +{\color{magenta}\la} )$  we highlight the `thick strand permutation' by colouring each individual 
row  of $\rho$ and $\la$ in grey and pink (as per the colouring in \cref{introfig}).
See \cref{bigeg,bigeg2} for an example.
\end{rmk}

\begin{figure}[ht!]
\[
\begin{tikzpicture}[scale=0.575]
\draw[thick] (0,0)--(7,0)--++(-90:1)
--++(180:1)--++(-90:1)--++(180:1)--++(-90:1)--++(180:1)--++(-90:1)--++(180:4)--(0,0); 

\clip (0,0)--(7,0)--++(-90:1)
--++(180:1)--++(-90:1)--++(180:1)--++(-90:1)--++(180:1)--++(-90:1)--++(180:4)--(0,0);

\fill [opacity=0.2] (0,0) rectangle (4,-4);

\fill [opacity=0.3,magenta] (4,0) rectangle (8,-4);
\draw[thick] (0,0) rectangle (3,-3);

\draw[thick] (0,-1) --++(0:14);
\draw[thick] (0,-2) --++(0:15);
\draw[thick] (0,-3) --++(0:15);

\draw[thick] (1,0) --++(-90:5);
\draw[thick] (2,0) --++(-90:5);
\draw[thick] (4,0) --++(-90:5);
\draw[thick] (5,0) --++(-90:5);
\draw[thick] (6,0) --++(-90:5);
\draw[thick] (3,0) --++(-90:5);

\draw (0.5,-0.5) node {$0$};
\draw (1.5,-0.5) node {$1$};
\draw (2.5,-0.5) node {$2$};
\draw (3.5,-0.5) node {$3$};
\draw (4.5,-0.5) node {$4$};
\draw (5.5,-0.5) node {$5$};
\draw (6.5,-0.5) node {$6$};

\draw (0.5,-1.5) node {$1$};
\draw (1.5,-1.5) node {$0$};
\draw (2.5,-1.5) node {$1$};
\draw (3.5,-1.5) node {$2$};
\draw (4.5,-1.5) node {$3$};
\draw (5.5,-1.5) node {$4$};

\draw (0.5,-2.5) node {$2$};
\draw (1.5,-2.5) node {$1$};
\draw (2.5,-2.5) node {$0$};
\draw (3.5,-2.5) node {$1$};
\draw (4.5,-2.5) node {$2$};

\draw (0.5,-3.5) node {$3$};
\draw (1.5,-3.5) node {$2$};
\draw (2.5,-3.5) node {$1$};
\draw (3.5,-3.5) node {$0$};
\draw (4.5,-3.5) node {$1$};
\end{tikzpicture}
\qquad
\begin{tikzpicture}[scale=0.575]
\draw[thick] (0,0)--(7,0)--++(-90:1)
--++(180:1)--++(-90:1)--++(180:1)--++(-90:1)--++(180:1)--++(-90:1)--++(180:4)--(0,0);

\clip (0,0)--(7,0)--++(-90:1)
--++(180:1)--++(-90:1)--++(180:1)--++(-90:1)--++(180:1)--++(-90:1)--++(180:4)--(0,0);

\fill [opacity=0.2] (0,0) rectangle (4,-4);

\fill [opacity=0.3,magenta] (4,0) rectangle (8,-4);
\draw[thick] (0,0) rectangle (3,-3);

\draw[thick] (0,-1) --++(0:14);
\draw[thick] (0,-2) --++(0:15);
\draw[thick] (0,-3) --++(0:15);

\draw[thick] (1,0) --++(-90:5);
\draw[thick] (2,0) --++(-90:5);
\draw[thick] (4,0) --++(-90:5);
\draw[thick] (5,0) --++(-90:5);
\draw[thick] (6,0) --++(-90:5);
\draw[thick] (3,0) --++(-90:5);

\draw (0.5,-0.5) node {$1$};
\draw (1.5,-0.5) node {$1$};
\draw (2.5,-0.5) node {$1$};
\draw (3.5,-0.5) node {$1$};
\draw (4.5,-0.5) node {$2$};
\draw (5.5,-0.5) node {$2$};
\draw (6.5,-0.5) node {$2$};

\draw (0.5,-1.5) node {$3$};
\draw (1.5,-1.5) node {$3$};
\draw (2.5,-1.5) node {$3$};
\draw (3.5,-1.5) node {$3$};
\draw (4.5,-1.5) node {$4$};
\draw (5.5,-1.5) node {$4$};

\draw (0.5,-2.5) node {$5$};
\draw (1.5,-2.5) node {$5$};
\draw (2.5,-2.5) node {$5$};
\draw (3.5,-2.5) node {$5$};
\draw (4.5,-2.5) node {$6$};

\draw (0.5,-3.5) node {$7$};
\draw (1.5,-3.5) node {$7$};
\draw (2.5,-3.5) node {$7$};
\draw (3.5,-3.5) node {$7$};
\draw (4.5,-3.5) node {$7$};
\end{tikzpicture}
\qquad
\begin{tikzpicture}[scale=0.575]
\draw[thick] (0,0)--(7,0)--++(-90:1)
--++(180:1)--++(-90:1)--++(180:1)--++(-90:1)--++(180:1)--++(-90:1)--++(180:4)--(0,0);

\clip (0,0)--(7,0)--++(-90:1)
--++(180:1)--++(-90:1)--++(180:1)--++(-90:1)--++(180:1)--++(-90:1)--++(180:4)--(0,0); 

\fill [opacity=0.2] (0,0) rectangle (4,-4);

\fill [opacity=0.3,magenta] (4,0) rectangle (8,-4);
\draw[thick] (0,0) rectangle (3,-3);

\draw[thick] (0,-1) --++(0:14);
\draw[thick] (0,-2) --++(0:15);
\draw[thick] (0,-3) --++(0:15);

\draw[thick] (1,0) --++(-90:5);
\draw[thick] (2,0) --++(-90:5);
\draw[thick] (4,0) --++(-90:5);
\draw[thick] (5,0) --++(-90:5);
\draw[thick] (6,0) --++(-90:5);
\draw[thick] (3,0) --++(-90:5);

\draw (0.5,-0.5) node {$1$};
\draw (1.5,-0.5) node {$1$};
\draw (2.5,-0.5) node {$1$};
\draw (3.5,-0.5) node {$1$};
\draw (4.5,-0.5) node {$5$};
\draw (5.5,-0.5) node {$5$};
\draw (6.5,-0.5) node {$5$};

\draw (0.5,-1.5) node {$2$};
\draw (1.5,-1.5) node {$2$};
\draw (2.5,-1.5) node {$2$};
\draw (3.5,-1.5) node {$2$};
\draw (4.5,-1.5) node {$6$};
\draw (5.5,-1.5) node {$6$};

\draw (0.5,-2.5) node {$3$};
\draw (1.5,-2.5) node {$3$};
\draw (2.5,-2.5) node {$3$};
\draw (3.5,-2.5) node {$3$};
\draw (4.5,-2.5) node {$7$};

\draw (0.5,-3.5) node {$4$};
\draw (1.5,-3.5) node {$4$};
\draw (2.5,-3.5) node {$4$};
\draw (3.5,-3.5) node {$4$};
\draw (4.5,-3.5) node {$4$};
\end{tikzpicture}
\]
\caption{For $\kapc=0$, $\rho = (4^4)$, and $\pla = {\color{magenta}(3,2,1)}$, we depict examples of the $\tt C_\infty$-residues of $\rho + \pla$, the tableau 
$\SSTT^{\rho + \pla}$, and the tableau $\SSTT_{\rho + \pla}$ respectively.}
\label{bigeg}
\end{figure}

\begin{figure}[ht!]
\[
  \begin{tikzpicture}[scale=0.22,yscale=-1]
 \draw(-1,0) rectangle (21.5*2,8);
\foreach \i in {0,1,2,...,21}
{\fill (\i*2,0) circle (4pt);
\fill (\i*2,8) circle (4pt);
\path(\i*2,8) coordinate (X\i);
\path(\i*2,0) coordinate (Y\i);
 }

 \fill[opacity =0.2] (X0)to (Y0) to (Y3) to (X3);
 \draw[thick](X0)to (Y0);
  \draw[thick](X1)to (Y1);
   \draw[thick](X2)to (Y2);
 \draw[thick](X3)to (Y3);

\fill[opacity =0.3,magenta] (X4)to [out=-60,in=90] (Y8) to (Y10) to [out=90,in=-60] (X6);

 \draw[thick](X4)to [out=-60,in=90] (Y8);
  \draw[thick](X5)to [out=-60,in=90] (Y9);
   \draw[thick](X6)to [out=-60,in=90] (Y10);

  \fill[opacity =0.2] (X7)to [out=-120,in=90] (Y4) to (Y7) to [out=90,in=-120] (X10);

 \draw[thick](X7)to [out=-120,in=90] (Y4);
  \draw[thick](X8)to [out=-120,in=90] (Y5);
   \draw[thick](X9)to [out=-120,in=90] (Y6);
   \draw[thick](X10)to [out=-120,in=90] (Y7);

 \fill[opacity =0.3,magenta] (X11)to[out=-90,in=90]  (Y15) to (Y16) to [out=90,in=-90] (X12);
 \draw[thick](X11)to[out=-90,in=90]  (Y15);
  \draw[thick](X12)to [out=-90,in=90] (Y16);

  \fill[opacity =0.2] (X13)to [out=-90,in=90] (Y11) to (Y14) to [out=90,in=-90] (X16);

 \draw[thick](X13)to [out=-90,in=90] (Y11);
  \draw[thick](X14)to [out=-90,in=90] (Y12);
   \draw[thick](X15)to [out=-90,in=90] (Y13);
   \draw[thick](X16)to [out=-90,in=90] (Y14);

 \draw[line width =4 ,magenta!30] (X17)to[out=-90,in=90]  (Y21)--++(-90:16);
 \draw[thick](X17)to[out=-90,in=90]  (Y21)--++(-90:16);

\fill[opacity =0.2] (X18)to [out=-90,in=90] (Y17) to (Y20) to [out=90,in=-90] (X21);

 \draw[thick](X18)to [out=-90,in=90] (Y17);
  \draw[thick](X19)to [out=-90,in=90] (Y18);
   \draw[thick](X20)to [out=-90,in=90] (Y19);
   \draw[thick](X21)to [out=-90,in=90] (Y20);


\draw(-1,0-8) rectangle (21.5*2,8-8);
\foreach \i in {0,1,2,...,21}
{\fill (\i*2,0-8) circle (4pt);
\fill (\i*2,8-8) circle (4pt);
\path(\i*2,8-8) coordinate (X\i);
\path(\i*2,0-8) coordinate (Y\i);
 }

\fill[opacity =0.2] (X0)to (Y0) to (Y3) to (X3);
 \draw[thick](X0)to (Y0);
  \draw[thick](X1)to (Y1);
   \draw[thick](X2)to (Y2);
 \draw[thick](X3)to (Y3);

\fill[opacity =0.2] (X4)to [out=-90,in=90] (Y4) to (Y7) to [out=90,in=-90] (X7);

 \draw[thick](X4)to [out=-90,in=90] (Y4);
  \draw[thick](X5)to [out=-90,in=90] (Y5);
   \draw[thick](X6)to [out=-90,in=90] (Y6);
 \draw[thick](X7)to [out=-90,in=90] (Y7);

\fill[opacity =0.3,magenta] (X8)to [out=-90,in=90] (Y12) to (Y14) to [out=90,in=-90] (X10);

 \draw[thick](X8)to [out=-90,in=90] (Y12);
  \draw[thick](X9)to [out=-90,in=90] (Y13);
   \draw[thick](X10)to [out=-90,in=90] (Y14);

\fill[opacity =0.2] (X11)to [out=-90,in=90] (Y8) to (Y11) to [out=90,in=-90] (X14);

 \draw[thick](X11)to [out=-90,in=90] (Y8);
  \draw[thick](X12)to [out=-90,in=90] (Y9);
   \draw[thick](X13)to [out=-90,in=90] (Y10);
   \draw[thick](X14)to [out=-90,in=90] (Y11);

\fill[opacity =0.3,magenta] (X15)to [out=-90,in=90] (Y19) to (Y20) to [out=90,in=-90] (X16);

 \draw[thick](X15)to [out=-90,in=90] (Y19);
  \draw[thick](X16)to [out=-90,in=90] (Y20);

\fill[opacity =0.2] (X17)to [out=-90,in=90] (Y15) to (Y18) to [out=90,in=-90] (X20);

 \draw[thick](X17)to [out=-90,in=90] (Y15);
  \draw[thick](X18)to [out=-90,in=90] (Y16);
   \draw[thick](X19)to [out=-90,in=90] (Y17);
   \draw[thick](X20)to [out=-90,in=90] (Y18);



\draw(-1,0-8-8) rectangle (21.5*2,8-8-8);
\foreach \i in {0,1,2,...,21}
{\fill (\i*2,0-8-8) circle (4pt);
\fill (\i*2,8-8-8) circle (4pt);
\path(\i*2,8-8-8) coordinate (X\i);
\path(\i*2,0-8-8) coordinate (Y\i);
 }

\fill[opacity =0.2] (X0)to (Y0) to (Y3) to (X3);
 \draw[thick](X0)to (Y0);
  \draw[thick](X1)to (Y1);
   \draw[thick](X2)to (Y2);
 \draw[thick](X3)to (Y3);

  \fill[opacity =0.2] (X4)to [out=-90,in=90] (Y4) to (Y7) to [out=90,in=-90] (X7);

 \draw[thick](X4)to [out=-90,in=90] (Y4);
  \draw[thick](X5)to [out=-90,in=90] (Y5);
   \draw[thick](X6)to [out=-90,in=90] (Y6);
 \draw[thick](X7)to [out=-90,in=90] (Y7);

  \fill[opacity =0.2] (X8)to [out=-90,in=90] (Y8) to (Y11) to [out=90,in=-90] (X11);

 \draw[thick](X8)to [out=-90,in=90] (Y8);
  \draw[thick](X9)to [out=-90,in=90] (Y9);
   \draw[thick](X10)to [out=-90,in=90] (Y10);
 \draw[thick](X11)to [out=-90,in=90] (Y11);

\fill[opacity =0.2] (X15)to [out=-90,in=90] (Y12) to (Y15) to [out=90,in=-90] (X18);

 \draw[thick](X15)to [out=-90,in=90] (Y12);
  \draw[thick](X16)to [out=-90,in=90] (Y13);
   \draw[thick](X17)to [out=-90,in=90] (Y14);
 \draw[thick](X18)to [out=-90,in=90] (Y15);

\fill[opacity =0.3,magenta] (X12)to [out=-90,in=90] (Y16) to (Y18) to [out=90,in=-90] (X14);

 \draw[thick](X12)to [out=-90,in=90] (Y16);
  \draw[thick](X13)to [out=-90,in=90] (Y17);
   \draw[thick](X14)to [out=-90,in=90] (Y18);

\fill[opacity =0.3,magenta] (X19)to [out=-90,in=90] (Y19) to (Y20) to [out=90,in=-90] (X20);

 \draw[thick](X19)to [out=-90,in=90] (Y19);
  \draw[thick](X20)to [out=-90,in=90] (Y20);

 \draw (Y0) node [above] {\scalefont{0.7}$0$};
  \draw (Y1) node [above] {\scalefont{0.7}$1$};
   \draw (Y2) node [above] {\scalefont{0.7}$2$};
  \draw (Y3) node [above] {\scalefont{0.7}$3$};

  \draw (Y4) node [above] {\scalefont{0.7}$1$};
   \draw (Y5) node [above] {\scalefont{0.7}$0$};
      \draw (Y6) node [above] {\scalefont{0.7}$1$};

  \draw (Y7) node [above] {\scalefont{0.7}$2$};
  \draw (Y8) node [above] {\scalefont{0.7}$2$};
   \draw (Y9) node [above] {\scalefont{0.7}$1$};
  \draw (Y10) node [above] {\scalefont{0.7}$0$};

   \draw (Y11) node [above] {\scalefont{0.7}$1$};
   \draw (Y12) node [above] {\scalefont{0.7}$3$};

   \draw (Y13) node [above] {\scalefont{0.7}$2$};
  \draw (Y14) node [above] {\scalefont{0.7}$1$};
   \draw (Y15) node [above] {\scalefont{0.7}$0$};

  \draw (Y16) node [above] {\scalefont{0.7}$\color{magenta}4$};
     \draw (Y17) node [above] {\scalefont{0.7}$\color{magenta}5$};
  \draw (Y18) node [above] {\scalefont{0.7}$\color{magenta}6$};

   \draw (Y19) node [above] {\scalefont{0.7}$\color{magenta}3$};
   \draw (Y20) node [above] {\scalefont{0.7}$\color{magenta}4$};

  \draw (Y21) node [above] {\scalefont{0.7}$\color{magenta}2$};

 \draw(-1,0+8) rectangle (21.5*2,8+8);
\foreach \i in {0,1,2,...,21}
{\fill (\i*2,0+8) circle (4pt);
\fill (\i*2,8+8) circle (4pt);
\path(\i*2,8+8) coordinate (X\i);
\path(\i*2,0+8) coordinate (Y\i);
\path(\i*2,0+8+8-4) coordinate (Z\i);
 }

\foreach \i in {1,10}
{
 \draw[thick,fill=black](Z\i) circle (10pt);
}

 \fill[opacity =0.2] (X0)to (Y0) to (Y3) to (X3);
 \draw[thick](X0)to (Y0);
  \draw[thick](X1)to (Y1);
   \draw[thick](X2)to (Y2);
 \draw[thick](X3)to (Y3);

 \draw[line width =4 ,magenta!30] (X17)--(Y17);

   \fill[opacity =0.2,magenta] (X11)to [out=-90,in=90] (Y11) to (Y12) to [out=90,in=-90] (X12);

   \fill[opacity =0.2,magenta] (X4)to [out=-90,in=90] (Y4) to (Y6) to [out=90,in=-90] (X6);

 \draw[thick](X4)to [out=-90,in=90] (Y4);
  \draw[thick](X5)to [out=-90,in=90] (Y5);
   \draw[thick](X6)to [out=-90,in=90] (Y6);
 \draw[thick](X7)to [out=-90,in=90] (Y7);

  \fill[opacity =0.2] (X7)to [out=-90,in=90] (Y7) to (Y10) to [out=90,in=-90] (X10);

 \draw[thick](X8)to [out=-90,in=90] (Y8);
  \draw[thick](X9)to [out=-90,in=90] (Y9);
   \draw[thick](X10)to [out=-90,in=90] (Y10);
 \draw[thick](X11)to [out=-90,in=90] (Y11);

  \fill[opacity =0.2] (X13)to [out=-90,in=90] (Y13) to (Y16) to [out=90,in=-90] (X16);

  \fill[opacity =0.2] (X21)to [out=-90,in=90] (Y21) to (Y18) to [out=90,in=-90] (X18);

 \draw[thick](X21)to [out=-90,in=90] (Y21);
 \draw[thick](X12)to [out=-90,in=90] (Y12);
  \draw[thick](X13)to [out=-90,in=90] (Y13);
   \draw[thick](X14)to [out=-90,in=90] (Y14);
 \draw[thick](X15)to [out=-90,in=90] (Y15);

 \draw[thick](X16)to [out=-90,in=90] (Y16);
  \draw[thick](X17)to [out=-90,in=90] (Y17);
   \draw[thick](X18)to [out=-90,in=90] (Y18);

 \draw[thick](X19)to [out=-90,in=90] (Y19);
  \draw[thick](X20)to [out=-90,in=90] (Y20);

 \draw (X0) node [below] {\scalefont{0.7}$0$};
  \draw (X1) node [below] {\scalefont{0.7}$1$};
   \draw (X2) node [below] {\scalefont{0.7}$2$};
  \draw (X3) node [below] {\scalefont{0.7}$3$};

 \draw (X4) node [below] {\scalefont{0.7}$\color{magenta}4$};
   \draw (X5) node [below] {\scalefont{0.7}$\color{magenta}5$};
      \draw (X6) node [below] {\scalefont{0.7}$\color{magenta}6$};

 \draw (X7) node [below] {\scalefont{0.7}$1$};
  \draw (X8) node [below] {\scalefont{0.7}$0$};
   \draw (X9) node [below] {\scalefont{0.7}$1$};
  \draw (X10) node [below] {\scalefont{0.7}$2$};

  \draw (X11) node [below] {\scalefont{0.7}$\color{magenta}3$};
   \draw (X12) node [below] {\scalefont{0.7}$\color{magenta}4$};

\draw (X13) node [below] {\scalefont{0.7}$2$};
  \draw (X14) node [below] {\scalefont{0.7}$1$};
   \draw (X15) node [below] {\scalefont{0.7}$0$};
  \draw (X16) node [below] {\scalefont{0.7}$1$};

\draw (X17) node [below] {\scalefont{0.7}$\color{magenta}2$};

\draw (X18) node [below] {\scalefont{0.7}$3$};
  \draw (X19) node [below] {\scalefont{0.7}$2$};
   \draw (X20) node [below] {\scalefont{0.7}$1$};
  \draw (X21) node [below] {\scalefont{0.7}$0$};
\end{tikzpicture}
\]
\caption{The element $\psi^{\SSTT_{\rho +\pla}}_{\SSTT^{\rho +\pla}} y_{\SSTT^{\rho + \pla}}$ corresponding to the tableaux in \cref{bigeg}.
We have $\SSTT_{\rho + \pla} =  s_4 (s_3s_5)(s_2s_4s_6) \SSTT^{\rho + \pla}$, where the bracketing indicates commuting elements (which can be reordered within the brackets at will to produce 12 distinct reduced words).
}
\label{bigeg2}
\end{figure}

\section{The isomorphism theorem}\label{sec:isomthm}

For the remainder of the paper we fix $\beta = \sum_{i\in I} a_i \alpha_i  \in \sfq^+_n$, and $\La = \La_{\kapc} \in \sfp^+$.
Let $\rho = ((a_0)^{\kapc+a_0})$ denote the minimal rectangle containing $a_0$ $0$-nodes, and let $\omega \in \sfq^+_n$ be such that $\rho$ is the unique partition in $\ParblockC[\omega]$.

\begin{eg}
Examples of $\rho$ are given by the grey partitions in \cref{introfig,bigeg}.  
In \cref{bigeg}, we have that $\rho=(4^4)$ and $\omega=4\alpha_0+6\alpha_1+4\alpha_2+2\alpha_3$. 
\end{eg}

We define the idempotent
\[
1_{\omega, \beta-\omega} :=  1_{\scrr_{\omega}} \otimes 1_{\scrr_{\beta-\omega}} \in \scrr^\Lambda_{\beta}(\mathtt{C}_\infty).
\]
In \cref{thm:homomorphism,thm:isomorphism}, we will truncate by this idempotent, and construct an isomorphism between this and the tensor product of a simple algebra and $\scrr^{\La_{{\color{magenta}\kappa_1}} + \Lambda_{{\color{cyan}\kappa_2}}}_{\beta-\omega}(\tta_\infty)$. 
In order to see that $1_{\omega, \beta-\omega} \scrr^\La_\beta(\mathtt{C}_\infty) 1_{\omega, \beta-\omega}$ and $\scrr^\La_\beta(\mathtt{C}_\infty)$ are Morita equivalent, we have the following \lcnamecref{prop:erhopreservessimples}.

\begin{prop}\label{prop:erhopreservessimples}
If $\nu \in \mathscr{K}^1_\beta(\mathtt{C}_\infty)$, then $1_{\omega, \beta-\omega} \D\nu\neq 0$.
\end{prop}

\begin{proof}
Recall that we have set $\beta = \mathsf{cont}(\nu) = \sum_{i\in I} a_i \alpha_i$. 
By \cite[Corollary~6G.10]{mathas22}, if there is a sequence of cogood nodes from $\varnothing$ to $\nu$, with underlying residue sequence $\bfj$, then $e(\bfj) \D\nu \neq 0$. 
We will construct such a sequence of the form $\bfj = \bfi^{\ttt_\rho} \otimes \bfi \in I^\beta$ for some $\bfi \in I^{\beta - \omega}$ and hence deduce that $e (\bfi^{\ttt_\rho} \otimes \bfi) \D\nu\neq 0$.
We will argue by induction on $a_0$.
For the base case $a_0= 1$, we work backwards from $\nu$ removing good nodes.
The required factorisation amounts to not removing any nodes of $\rho$ until all nodes of $\nu \setminus \rho$ have been removed.
Equivalently, we need a good node sequence that doesn't remove the unique $0$-node of $\rho=(1^{\kapc+1})$ until we have removed all nodes outside of $\rho$.
The $0$-node is only removable once all other nodes in its row and lower rows have been removed, by which point all other remaining removable nodes will have residue $2$ or larger.
Since there cannot also be any addable nodes of such residues, and $\nu$ is Kleshchev by assumption, it follows that other removable nodes are always good, and the necessary factorisation is thus possible.

We now turn to the inductive step, we will prove the contrapositive statement.
We let $a_0 \geq 2$, so that $\rho=(a_0^{({\kapc}+a_0)})$.
We suppose there does not exist a residue sequence 
$\bfj = \bfj' \otimes \bfj'' \in I^\beta$ giving a sequence of good nodes $\varnothing \xrightarrow{\bf j} \nu$, we will prove that $\nu$ is not Kleshchev.
Since we only have two diagonals with any given residue in $\nu$, we have that any $i$-sequence containing a removable node must be of the form $(r,r)$, $(a,r)$, $(r,a)$ or $(r)$ (with only $(r,a)$ being a bad $i$-sequence).

Recall that $\nu= \rho + (\pla, \bbmu)$. 
We first note that if  $\bbmu$ has a removable $i$-node, then the corresponding $i$-sequence for $\nu$ cannot be of the form  $(r,a)$ (and hence is good). Therefore we can remove all  such sequences to obtain a partition of the form $\rho+(\widehat{\pla},\varnothing)$. 
Now, let $i\neq 0, a_0$ then $|\Rem_i(\widehat{\nu})|\leq 1$ and so the corresponding $i$-sequence for $\rho+(\widehat{\pla},\varnothing)$ cannot be of the form $(r,a)$.
Therefore we can remove all  such sequences to obtain a partition of the form 
\[
\bar\nu =
  \rho + ({\color{magenta}(b^{(b+{\kapc})})},\varnothing)
\]
for some $b\geq 1$. 
(Otherwise $\rho=\bar\nu$ and we have constructed our desired factorisation of the sequence $\bfj = \bfj' \otimes \bfj''$ with
 $\varnothing \xrightarrow{\bf j'} \rho = \bar\nu \xrightarrow{\bf j''} \nu $.) 
Now, $\bar\nu$ only has  a removable node of residue $0$ and one of residue $a_0$; only the $0$-node is good. 
The final row of $\rho$ has residue sequence $0,1,2, \dots a_0-1$, read from right-to-left.
For $i=0,1,\dots, a_0-1$ we remove all  removable $i$-nodes in turn until we obtain 
\[
(a_0-1)^{\kapc+a_0-1} + ({\color{magenta}((b+1)^{(b+{\kapc})})},\varnothing).
\]Finally, we can remove all nodes in the final column (all of which are  good) and hence  obtain 
  \[
 \bar{\bar\nu }
 =
 (a_0-1)^{\kapc+a_0-1} + ({\color{magenta}(b^{(b+{\kapc})})},\varnothing) 
\]
Now, $\bar{\bar{\nu}}$ only has  a removable node of residue $0$ and one of residue $a_0-1$; only the $0$-node is good.  
This cannot be factorised through the rectangle $ (a_0-1)^{\kapc+a_0-1} $ and hence, by the inductive assumption, is not Kleshchev.  It follows that $\nu$ was not Kleshchev and the result follows.
\end{proof}

\begin{eg}
We give a small example of the base case $a_0=1$ instance of the proof of \cref{prop:erhopreservessimples}.
Let $\kapc = 2$, so that $\rho=(1^3)$ and $\omega = \alpha_0 + \alpha_1 + \alpha_2$.
Let $\nu = (4,2^2,1^2)$.
Then $\nu$ has a good $2$-node.
After removing it, we may also successively remove two good $1$-nodes, leaving the partition $(4,2,1)$.
This partition now has a good $0$-node, but we may remove all nodes outside of $\rho$ before that, since those residues will never simultaneously have an addable and a removable node.
This yields several choices of the necessary factorisation $\bfj$, such as $\bfj = (2,1,0) \otimes (3,4,5,2,1,1,2)$.
\end{eg}

\begin{defn}
For $\nu \in \ParblockC$, we define
\[
\Std_{\omega,\beta-\omega}(\nu):=
\{\ttt \in \Std(\nu) \mid \bfi^\ttt = \bfi \otimes \bfj \text{ such that } \bfi \in I^\omega \text{ and } \bfj\in I^{\beta-\omega}\}.
\]
\end{defn}

\begin{cor}\label{neededlater}
The algebra $1_{\omega, \beta-\omega} \scrr^{\La_{\kapc}}_{\beta}(\mathtt{C}_\infty) 1_{\omega, \beta-\omega}$ is 
a graded cellular algebra with basis 
\[
\{ c_{\tts\ttt} = \psi_\tts y_{\ttt^{\nu}}
\psi_\ttt^\ast \mid \tts,\ttt \in \Std_{\omega,\beta-\omega}(\nu),
 \nu \in \ParblockC\}.
\]
The algebras 
$\scrr^{\La_{\kapc}}_{\beta}(\mathtt{C}_\infty)$ and  
$1_{\omega, \beta-\omega} \scrr^{\La_{\kapc}}_{\beta}(\mathtt{C}_\infty) 1_{\omega, \beta-\omega}$ are graded Morita equivalent and this preserves the cellular structure. 
\end{cor}

\begin{lem}\label{lem:maxdeg}
For $\rho = ((a_0)^{\kapc+a_0}) \in \ParblockC[\omega]$, we have that 
\[
\dim_q(e(\bfi^\rho) \spe\rho) = (q+q^{-1})^{\lfloor \tfrac{a_0}{2} \rfloor}.
\]
In particular, the tableau 
$\ttt^\rho \in \Std(\bfi^\rho)$ (respectively $\ttt_\rho \in \Std(\bfi^\rho)$) is the unique tableau of maximal degree $\lfloor \tfrac{a_0}{2} \rfloor$
 (respectively minimal degree, $-\lfloor \tfrac{a_0}{2} \rfloor$)  in its residue class.
\end{lem}

\begin{proof}
We  construct all possible tableaux in $\Std(\bfi^\rho)$. 
It suffices to consider the case that $\rho$ is a square with $\kapc=0$  (the earlier rows in the rectangular case with $\kapc>0$ contribute nothing to the graded dimension).  
For $1\leq k \leq a_0^2$ we let $i_k$ denote the $k$th residue in the sequence $\bfi^\rho\in I^{a_0^2}$ and we  let  $r_k$ denote the row in 
$\ttt^\rho$ in which the $k$th node appears. 
 For $1\leq k \leq a_0^2$ with     $1\leq r_k \leq \lfloor a_0/2\rfloor $ 
there are two places to add the 
$k$th node to ${\rm Shape}(\ttt_{\leq k-1})$ if and only if 
$k= (r_k-1) a_0 +2r_k$; otherwise we have exactly one
  place  where we can add the $k$th node
(in particular, the ungraded dimension of $e(\bfi^\rho) \spe\rho$ is $2^{{\lfloor \tfrac{a_0}{2} \rfloor}}$).  
In fact, the only nodes which contribute to the graded dimension are
the $k$th and $(a_0^2- k+1)$th nodes for   $k= (r_k-1) a_0 +2r_k$  and 
  $1\leq r_k \leq \lfloor a_0/2\rfloor $.  
 (See \cref{seeme!} for an example.)

For  $k= (r_k-1) a_0 +2r_k$ the $k$th node is placed above (respectively below) the $0$-diagonal 
if and only if the $(a_0^2- k+1)$th node is below (respectively above) the $0$-diagonal.
For such $1\leq k \leq a_0^2$, the $k$th node contributes degree $+1$ (respectively $0$) 
if and only if the $(a_0^2- k+1)$th node contributes $0$ (respectively $-1$) to the degree, respectively. 
\end{proof}

\begin{figure}[ht!]
\[
\scalefont{0.9}
\begin{tikzpicture}[scale=0.55]
\draw[thick] (0,0)--(6,0)--++(-90:6)
--++(180:6)--++(90:6);

\clip (0,0)--(6,0)--++(-90:6)
--++(180:6)--++(90:6);


 \draw[thick] (0,0) rectangle (3,-3);

\draw[thick] (0,-1) --++(0:6);
\draw[thick] (0,-2) --++(0:6);
\draw[thick] (0,-3) --++(0:6);
\draw[thick] (0,-4) --++(0:6);
\draw[thick] (0,-5) --++(0:6);

\draw[thick] (1,0) --++(-90:6);
\draw[thick] (2,0) --++(-90:6);
\draw[thick] (3,0) --++(-90:6);
\draw[thick] (4,0) --++(-90:6);
\draw[thick] (5,0) --++(-90:6);

\path(1,0) coordinate (X);
 \fill[orange,opacity=0.3](X)--++(0:1)--++(-90:1)--++(180:1);

\path(3,-1) coordinate (X);
 \fill[orange,opacity=0.3](X)--++(0:1)--++(-90:1)--++(180:1);

\path(5,-2) coordinate (X);
 \fill[orange,opacity=0.3](X)--++(0:1)--++(-90:1)--++(180:1);

\path(0,-3) coordinate (X);
 \fill[green!50!black,opacity=0.3](X)--++(0:1)--++(-90:1)--++(180:1);

\path(2,-4) coordinate (X);
 \fill[green!50!black,opacity=0.3](X)--++(0:1)--++(-90:1)--++(180:1);

\path(4,-5) coordinate (X);
 \fill[green!50!black,opacity=0.3](X)--++(0:1)--++(-90:1)--++(180:1);

\draw (0.5,-0.5) node {$1$};
\draw (1.5,-0.5) node {$2$};
\draw (2.5,-0.5) node {$3$};
\draw (3.5,-0.5) node {$4$};
\draw (4.5,-0.5) node {$5$};
\draw (5.5,-0.5) node {$6$};

\draw (0.5,-1.5) node {$7$};
\draw (1.5,-1.5) node {$8$};
\draw (2.5,-1.5) node {$9$};
\draw (3.5,-1.5) node {$10$};
\draw (4.5,-1.5) node {$11$};
\draw (5.5,-1.5) node {$12$};

\draw (0.5,-2.5) node {$13$};
\draw (1.5,-2.5) node {$14$};
\draw (2.5,-2.5) node {$15$};
\draw (3.5,-2.5) node {$16$};
\draw (4.5,-2.5) node {$17$};
\draw (5.5,-2.5) node {$18$};

\draw (0.5,-3.5) node {$19$};
\draw (1.5,-3.5) node {$20$};
\draw (2.5,-3.5) node {$21$};
\draw (3.5,-3.5) node {$22$};
\draw (4.5,-3.5) node {$23$};
\draw (5.5,-3.5) node {$24$};

\draw (0.5,-4.5) node {$25$};
\draw (1.5,-4.5) node {$26$};
\draw (2.5,-4.5) node {$27$};
\draw (3.5,-4.5) node {$28$};
\draw (4.5,-4.5) node {$29$};
\draw (5.5,-4.5) node {$30$};

\draw (0.5,-5.5) node {$31$};
\draw (1.5,-5.5) node {$32$};
\draw (2.5,-5.5) node {$33$};
\draw (3.5,-5.5) node {$34$};
\draw (4.5,-5.5) node {$35$};
\draw (5.5,-5.5) node {$36$};
\end{tikzpicture}
\qquad 
\begin{tikzpicture}[scale=0.55]
\draw[thick] (0,0)--(6,0)--++(-90:6)
--++(180:6)--++(90:6);

\clip (0,0)--(6,0)--++(-90:6)
--++(180:6)--++(90:6);


 \draw[thick] (0,0) rectangle (3,-3);

\draw[thick] (0,-1) --++(0:6);
\draw[thick] (0,-2) --++(0:6);
\draw[thick] (0,-3) --++(0:6);
\draw[thick] (0,-4) --++(0:6);
\draw[thick] (0,-5) --++(0:6);

\draw[thick] (1,0) --++(-90:6);
\draw[thick] (2,0) --++(-90:6);
\draw[thick] (3,0) --++(-90:6);
\draw[thick] (4,0) --++(-90:6);
\draw[thick] (5,0) --++(-90:6);

\path(1,0) coordinate (X);
 \fill[orange,opacity=0.3](X)--++(0:1)--++(-90:1)--++(180:1);

\path(3,-1) coordinate (X);
 \fill[orange,opacity=0.3](X)--++(0:1)--++(-90:1)--++(180:1);

\path(5,-2) coordinate (X);
 \fill[green!50!black,opacity=0.3](X)--++(0:1)--++(-90:1)--++(180:1);

\path(0,-3) coordinate (X);
 \fill[orange,opacity=0.3](X)--++(0:1)--++(-90:1)--++(180:1);

\path(2,-4) coordinate (X);
 \fill[green!50!black,opacity=0.3](X)--++(0:1)--++(-90:1)--++(180:1);

\path(4,-5) coordinate (X);
 \fill[green!50!black,opacity=0.3](X)--++(0:1)--++(-90:1)--++(180:1);

\draw (0.5,-0.5) node {$1$};
\draw (1.5,-0.5) node {$2$};
\draw (2.5,-0.5) node {$3$};
\draw (3.5,-0.5) node {$4$};
\draw (4.5,-0.5) node {$5$};
\draw (5.5,-0.5) node {$6$};

\draw (0.5,-1.5) node {$7$};
\draw (1.5,-1.5) node {$8$};
\draw (2.5,-1.5) node {$9$};
\draw (3.5,-1.5) node {$10$};
\draw (4.5,-1.5) node {$11$};
\draw (5.5,-1.5) node {$12$};

\draw (0.5,-2.5) node {$13$};
\draw (1.5,-2.5) node {$14$};
\draw (2.5,-2.5) node {$15$};
\draw (3.5,-2.5) node {$16$};
\draw (4.5,-2.5) node {$17$};
\draw (5.5,-2.5) node {$19$};

\draw (0.5,-3.5) node {$18$};
\draw (1.5,-3.5) node {$20$};
\draw (2.5,-3.5) node {$21$};
\draw (3.5,-3.5) node {$22$};
\draw (4.5,-3.5) node {$23$};
\draw (5.5,-3.5) node {$24$};

\draw (0.5,-4.5) node {$25$};
\draw (1.5,-4.5) node {$26$};
\draw (2.5,-4.5) node {$27$};
\draw (3.5,-4.5) node {$28$};
\draw (4.5,-4.5) node {$29$};
\draw (5.5,-4.5) node {$30$};

\draw (0.5,-5.5) node {$31$};
\draw (1.5,-5.5) node {$32$};
\draw (2.5,-5.5) node {$33$};
\draw (3.5,-5.5) node {$34$};
\draw (4.5,-5.5) node {$35$};
\draw (5.5,-5.5) node {$36$};
\end{tikzpicture}
\qquad 
\begin{tikzpicture}[scale=0.55]
\draw[thick] (0,0)--(6,0)--++(-90:6)
--++(180:6)--++(90:6);

\clip (0,0)--(6,0)--++(-90:6)
--++(180:6)--++(90:6); 


 \draw[thick] (0,0) rectangle (3,-3);

\draw[thick] (0,-1) --++(0:6);
\draw[thick] (0,-2) --++(0:6);
\draw[thick] (0,-3) --++(0:6);
\draw[thick] (0,-4) --++(0:6);
\draw[thick] (0,-5) --++(0:6);

\draw[thick] (1,0) --++(-90:6);
\draw[thick] (2,0) --++(-90:6);
\draw[thick] (3,0) --++(-90:6);
\draw[thick] (4,0) --++(-90:6);
\draw[thick] (5,0) --++(-90:6);

\path(1,0) coordinate (X);
 \fill[orange,opacity=0.3](X)--++(0:1)--++(-90:1)--++(180:1);

\path(4,-1) coordinate (X);
 \fill[orange,opacity=0.3](X)--++(0:1)--++(-90:1)--++(180:1);

\path(0,-2) coordinate (X);
 \fill[orange,opacity=0.3](X)--++(0:1)--++(-90:1)--++(180:1);

\path(1,-4) coordinate (X);
 \fill[green!50!black,opacity=0.3](X)--++(0:1)--++(-90:1)--++(180:1);
\path(5,-3) coordinate (X);
 \fill[green!50!black,opacity=0.3](X)--++(0:1)--++(-90:1)--++(180:1);

\path(4,-5) coordinate (X);
 \fill[green!50!black,opacity=0.3](X)--++(0:1)--++(-90:1)--++(180:1);

\draw (0.5,-0.5) node {$1$};
\draw (1.5,-0.5) node {$2$};
\draw (2.5,-0.5) node {$3$};
\draw (3.5,-0.5) node {$4$};
\draw (4.5,-0.5) node {$5$};
\draw (5.5,-0.5) node {$6$};

\draw (0.5,-1.5) node {$7$};
\draw (1.5,-1.5) node {$8$};
\draw (2.5,-1.5) node {$9$};
\draw (3.5,-1.5) node {$13 $};
\draw (4.5,-1.5) node {$18 $};
\draw (5.5,-1.5) node {$25 $};

\draw (0.5,-2.5) node {$10$};
 \draw (1.5,-2.5) node {$14$};
 \draw (2.5,-2.5) node {$15$};
 \draw (3.5,-2.5) node {$16$};
 \draw (4.5,-2.5) node {$20$};
 \draw (5.5,-2.5) node {$26$};

 \draw (0.5,-3.5) node {$11$};
 \draw (1.5,-3.5) node {$17$};
\draw (2.5,-3.5) node {$21$};
\draw (3.5,-3.5) node {$22$};
\draw (4.5,-3.5) node {$23$};
 \draw (5.5,-3.5) node {$27$};

 \draw (0.5,-4.5) node {$12$};
 \draw (1.5,-4.5) node {$19$};
 \draw (2.5,-4.5) node {$24$};
\draw (3.5,-4.5) node {$28$};
\draw (4.5,-4.5) node {$29$};
\draw (5.5,-4.5) node {$30$};

\draw (0.5,-5.5) node {$31$};
\draw (1.5,-5.5) node {$32$};
\draw (2.5,-5.5) node {$33$};
\draw (3.5,-5.5) node {$34$};
\draw (4.5,-5.5) node {$35$};
\draw (5.5,-5.5) node {$36$};
\end{tikzpicture}
\qquad 
\begin{tikzpicture}[scale=0.55]
\draw[thick] (0,0)--(6,0)--++(-90:6)
--++(180:6)--++(90:6);

\clip (0,0)--(6,0)--++(-90:6)
--++(180:6)--++(90:6);


 \draw[thick] (0,0) rectangle (3,-3);

\draw[thick] (0,-1) --++(0:6);
\draw[thick] (0,-2) --++(0:6);
\draw[thick] (0,-3) --++(0:6);
\draw[thick] (0,-4) --++(0:6);
\draw[thick] (0,-5) --++(0:6);

\draw[thick] (1,0) --++(-90:6);
\draw[thick] (2,0) --++(-90:6);
\draw[thick] (3,0) --++(-90:6);
\draw[thick] (4,0) --++(-90:6);
\draw[thick] (5,0) --++(-90:6);

\path(0,-1) coordinate (X);
 \fill[orange,opacity=0.3](X)--++(0:1)--++(-90:1)--++(180:1);

\path(2,0) coordinate (X);
 \fill[orange,opacity=0.3](X)--++(0:1)--++(-90:1)--++(180:1);
\path(4,-1) coordinate (X);
 \fill[orange,opacity=0.3](X)--++(0:1)--++(-90:1)--++(180:1);
\path(1,-4) coordinate (X);
 \fill[green!50!black,opacity=0.3](X)--++(0:1)--++(-90:1)--++(180:1);
\path(3,-5) coordinate (X);
 \fill[green!50!black,opacity=0.3](X)--++(0:1)--++(-90:1)--++(180:1);
   \path(5,-4) coordinate (X);
      \fill[green!50!black,opacity=0.3](X)--++(0:1)--++(-90:1)--++(180:1);

\draw (0.5,-0.5) node {$1$};
 \draw (1.5,-0.5) node {$7$};
 \draw (2.5,-0.5) node {$10$};
 \draw (3.5,-0.5) node {$11$};
 \draw (4.5,-0.5) node {$12$};
 \draw (5.5,-0.5) node {$31$};
 \draw (0.5,-1.5) node {$2$};
 \draw (1.5,-1.5) node {$8$};
 \draw (2.5,-1.5) node {$14$};
 \draw (3.5,-1.5) node {$17$};
 \draw (4.5,-1.5) node {$18$};
 \draw (5.5,-1.5) node {$32$};
 \draw (0.5,-2.5) node {$3$};
 \draw (1.5,-2.5) node {$9$};
 \draw (2.5,-2.5) node {$15$};
 \draw (3.5,-2.5) node {$21$};
 \draw (4.5,-2.5) node {$24$};
 \draw (5.5,-2.5) node {$33$};
 \draw (0.5,-3.5) node {$4$};
 \draw (1.5,-3.5) node {$13$};
 \draw (2.5,-3.5) node {$16$};
 \draw (3.5,-3.5) node {$22$};
 \draw (4.5,-3.5) node {$28$};
 \draw (5.5,-3.5) node {$34$};
 \draw (0.5,-4.5) node {$5$};
 \draw (1.5,-4.5) node {$19$};
 \draw (2.5,-4.5) node {$20$};
 \draw (3.5,-4.5) node {$23$};
 \draw (4.5,-4.5) node {$29$};
 \draw (5.5,-4.5) node {$35$};
 \draw (0.5,-5.5) node {$6$};
\draw (1.5,-5.5) node {$25$};
\draw (2.5,-5.5) node {$26$};
\draw (3.5,-5.5) node {$27$};
\draw (4.5,-5.5) node {$30$};
\draw (5.5,-5.5) node {$36$};
\end{tikzpicture}
\]
\caption{The leftmost tableau is of maximal degree, $3$ (note that every orange tile has degree $+1$ and every green has degree 0). 
The next three tableaux are all possible tableaux of degree $1$; 
 in each case there is a unique pair of  orange/green nodes of degree $0$/$-1$; these are $18$/$19$,  $10$/$27$, and $2$/$35$ respectively.}
\label{seeme!}
\end{figure}

\begin{lem}\label{easypeasylemon}
For $\rho = ((a_0)^{\kapc+a_0}) \in \ParblockC[\omega]$, we have that 
$y_k y_{\ttt ^\rho} = 0 \in \scrr^{\La_{\kapc}}_{\omega}(\mathfrak{sp}_\infty)$
for all $1\leq k \leq |\rho|$.
\end{lem}

\begin{proof}
The tableau $\ttt^\rho$ is of maximal degree in its residue class, by \cref{lem:maxdeg}.
Therefore $y_k y_{\ttt^\rho} =0$ by degree considerations (there is no tableau of degree 
greater than $\deg(\ttt^\rho)$ and so the result follows from \cref{thm:cellular}).
\end{proof}

We are ready to introduce the map needed for our main result.
We first show that it is an algebra homomorphism and will prove in \cref{thm:isomorphism} that it is, in fact, an isomorphism of graded algebras.

\begin{thm}\label{thm:homomorphism}
Let $\beta = \sum_{i\in I} a_i \alpha_i  \in \sfq^+_n$, $\La = \La_{\kapc} \in \sfp^+$, and define $\rho \in \ParblockC[\omega]$ as above.
Set ${\color{magenta} \kappa_1} = \kapc + a_0$ and ${\color{cyan} \kappa_2} = a_0$.
We have a homomorphism of algebras
\[
\varphi: \scrr^{\La_{\kapc}}_{\omega}(\mathfrak{sp}_\infty) \otimes
\scrr^{\La_{{\color{magenta}\kappa_1}} + \La_{{\color{cyan}\kappa_2}}}_{\beta - \omega}(\mathfrak{sl}_\infty)
\longrightarrow
1_{\omega, \beta-\omega} \scrr^{\La_{\kapc}}_{\beta}(\mathfrak{sp}_\infty) 1_{\omega, \beta-\omega}
\]
given by 
\begin{equation}\label{sakljghdlfjhgjldskhgjdflk}
\varphi: x_1 \otimes x_2 \longmapsto 1_{\omega, \beta-\omega} (x_1 \otimes x_2) 1_{\omega, \beta-\omega}
\end{equation}
for $x_1 \in \scrr^{\Lambda_{\kapc}}_{\omega}(\mathfrak{sp}_\infty)$ and $x_2 \in \scrr^{\La_{{\color{magenta} \kappa_1}}+\La_{{\color{cyan}\kappa_2}}}_{\beta - \omega}(\mathfrak{sl}_\infty)$.
\end{thm}

\begin{proof}
First note that, as $\beta - \omega$ is not supported in $\alpha_0$, the local relations defining $ \scrr_{\beta - \omega}(\mathfrak{sl}_\infty)$ are a subset of those in $\scrr_{\beta}(\mathfrak{sp}_\infty)$. 
Thus we merely need to verify the cyclotomic relations to see that this map is a well-defined homomorphism.
Let $\bfi \in I^\omega$, $\bfj \in I^{\beta-\omega}$.
Note that 
\begin{align*}
\varphi\big(y_1^{\langle \alpha_{i_1}^\vee, \La_{\kapc}\rangle}e(\bfi) \otimes e(\bfj)\big) = 0
\end{align*}
by the cyclotomic relation on $\scrr^{\La_{\kapc}}_{\beta}(\mathfrak{sp}_\infty)$.
We now consider the other cyclotomic relation.  We will show that 
\begin{align*}
\varphi\big( e(\bfi) \otimes
y_1^{\langle \alpha_{j_1}^\vee,
\La_{{\color{magenta}\kappa_1}} + \La_{{\color{cyan}\kappa_2}}
\rangle}
 e(\bfj)\big) 
 = 0 
\end{align*}for all  $\bfi \in I^\omega$, $\bfj \in I^{\beta-\omega}$. 
Recall that $\rho$ is the unique partition in $\ParblockC[\omega]$.
Therefore by \cref{thm:cellular}, the algebra  $\scrr^{\La_{\kapc}}_{\omega}(\mathfrak{sp}_\infty)$ is simple and  generated by the element $y_{\ttt^\rho}$.
It thus suffices to prove that 
\begin{align*}
\varphi\big( y_{\ttt^\rho} \otimes
y_1^{\langle \alpha_{j_1}^\vee,  \La_{{\color{magenta}\kappa_1}} + \La_{{\color{cyan}\kappa_2}}\rangle}
 e(\bfj)\big) 
 =  0.
\end{align*}
There are three  cases to 
consider: $|\Add_{j_1}(\rho)|=0,1,2$. 
In the case that $|\Add_{j_1}(\rho)|=0$ there is no standard tableau with residue sequence $\bfi^\rho  \otimes   j_1$ and so 
$y_{\ttt^\rho} \otimes e(j_1)=0$ by \cref{remforzero}, as required. 
In the case that $|\Add_{j_1}(\rho)|=1$ we can pull the $j_1$-strand in 
$y_{\ttt^\rho} \otimes y_1e(j_1)$ to the left using the commutativity relations until we encounter a $(j_1-1)$-strand;
here we apply the relation $y_2 e(j_1-1,j_1) = \pm (\psi_1^2 e(j_1-1,j_1) - y_1 e(j_1-1,j_1))$; 
the first term on the righthand-side is zero by \cref{remforzero} and the second term is zero by \cref{easypeasylemon} (via the natural homomorphism $\scrr^\La_\omega(\fkg) \rightarrow \scrr^\La_{\beta}(\fkg)$).

We now consider the case that $|\Add_{j_1}(\rho)|=2$. We can pull the $j_1$-strand in 
$y_{\ttt^\rho} \otimes y_1^2e(j_1)$ to the left using the commutativity relations until we encounter a $(j_1-1)$-strand;
here we apply the relation $y_2^2 e(j_1-1,j_1) = \pm (\psi_1y_1 e(j_1,j_1-1) \psi_1 - y_1y_2 e(j_1-1,j_1)) $; 
the   second term is zero by \cref{easypeasylemon}. We now consider the first term: we can continue to pull the (now singly dotted) $j_1$-strand to the left until it encounters {\em the other} $(j_1-1)$-strand; we again apply the relation 
$y_2 e(j_1-1,j_1) = \pm (\psi_1^2 e(j_1-1,j_1) - y_1 e(j_1-1,j_1))$ and the first term on the righthand-side is again zero by \cref{remforzero} and the second term is again zero by \cref{easypeasylemon}.
\end{proof}

\begin{thm}\label{thm:isomorphism}
The homomorphism $\varphi$ from \cref{thm:homomorphism} is an isomorphism. 
\end{thm}

\begin{proof}
Write \(r = \operatorname{ht}(\omega)\), \(b = \operatorname{ht}(\beta- \omega)\) and \(m = \textup{ht}(\beta)\).
For any \(\sigma \in \mathfrak{S}_m\), we may write \(\sigma = w X u\), where \(w = (w_1, w_2), u = (u_1, u_2) \in \mathfrak{S}_r \times \mathfrak{S}_b \subseteq \mathfrak{S}_m\), and where \(X\) is a block transposition of the form:
\begin{align*}
X = (c, r+1) (c+1, r+2) \dots (r, 2r+ 1-c)
\end{align*}
for some \(c \in [1,r+1]\), noting that such \(X\) are representatives for the double cosets \(\mathfrak{S}_r \times \mathfrak{S}_b \backslash \mathfrak{S}_m / \mathfrak{S}_r \times \mathfrak{S}_b\). 
Then by \cite[Theorem~2.5]{kl09}, \cite[Theorem~3.7]{Rouq}, we have that \(1_{\omega , \beta-\omega} \scrr^{\La_{\kapc}}_{\beta}(\mathfrak{sp}_\infty) 1_{\omega , \beta-\omega}\) is spanned by elements of the form
\begin{align}\label{spanR}
1_{\omega, \beta-\omega} (\psi_{w_1} \otimes \psi_{w_2}) 
e(\bi \otimes \bk \otimes \bj \otimes \bm)
\psi_X
e(\bi\otimes  \bj \otimes  \bk \otimes \bm)
(\psi_{u_1} \otimes \psi_{u_2}) y_1^{f_1} \dots y_m^{f_m} 1_{\omega, \beta-\omega},
\end{align}
where \(u_1, u_2, w_1,w_2, X\) are as above, \(f_1, \dots, f_m \in \ZZ_{\geq 0}\), \(\bi \in I^{c-1}\), \(\bj, \bk \in I^{r-c+1}\), \(\bm \in I^{b-c+1}\), with 
 \(\bi \otimes \bj, \bi  \otimes  \bk \in I^{\omega}\), \(\bk \otimes  \bm, \bj  \otimes \bm \in I^{\beta-\omega}\). 
It is most convenient to view elements of the form (\ref{spanR}) using the diagrammatic presentation of \(\scrr^{\La_{\kapc}}_{\beta}(\mathfrak{sp}_\infty)\) (see \cite{kl09}) where they look like the following.
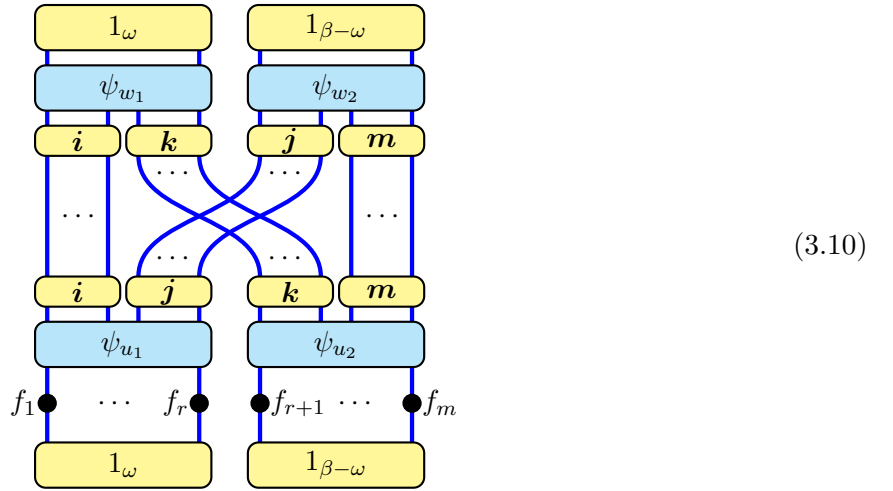
\begin{align}\label{KLRsepdiag}
\hackcenter{
{}
}
\hackcenter{
\begin{tikzpicture}[scale=.8]
  \draw[ultra thick,blue] (-0.5,0) .. controls ++(0,1) and ++(0,-1) .. (1.5,2);
    \draw[ultra thick,blue] (-1.5,0) .. controls ++(0,1) and ++(0,-1) .. (0.5,2);
      \draw[ultra thick,blue] (0.5,0) .. controls ++(0,1) and ++(0,-1) .. (-1.5,2);
    \draw[ultra thick,blue] (1.5,0) .. controls ++(0,1) and ++(0,-1) .. (-0.5,2);
     \draw[ultra thick,blue] (-3,-3)--(-3,0);
  \draw[ultra thick,blue] (-0.5,-3)--(-0.5,0);
    \draw[ultra thick,blue] (-2,-1)--(-2,0);
       \draw[ultra thick,blue] (-1.5,-1)--(-1.5,0);
   \draw[ultra thick,blue] (3,-3)--(3,0);
  \draw[ultra thick,blue] (0.5,-3)--(0.5,0);
      \draw[ultra thick,blue] (2,-1)--(2,0);
       \draw[ultra thick,blue] (1.5,-1)--(1.5,0);
           \draw[ultra thick,blue] (-3,2)--(-3,4);
           \draw[ultra thick,blue] (-2,2)--(-2,3);
                 \draw[ultra thick,blue] (-1.5,2)--(-1.5,3);
           \draw[ultra thick,blue] (-0.5,2)--(-0.5,4);
             \draw[ultra thick,blue] (3,2)--(3,4);
           \draw[ultra thick,blue] (2,2)--(2,3);
                 \draw[ultra thick,blue] (1.5,2)--(1.5,3);
           \draw[ultra thick,blue] (0.5,2)--(0.5,4);
     \draw[ultra thick,blue] (-3,0)--(-3,2);
      \draw[ultra thick,blue] (-2,0)--(-2,2);
      \draw[ultra thick,blue] (2,0)--(2,2);
     \draw[ultra thick,blue] (3,0)--(3,2);
     \draw[thick, rounded corners, fill=yellow!50] (-3.2, -0.5) rectangle (-1.8, 0) {};
       \draw[thick, rounded corners, fill=yellow!50] (-1.7, -0.5) rectangle (-0.3, 0) {};
            \draw[thick, rounded corners, fill=yellow!50] (3.2, -0.5) rectangle (1.8, 0) {};
       \draw[thick, rounded corners, fill=yellow!50] (1.7, -0.5) rectangle (0.3, 0) {};
      \node[] at (-2.5, -0.25) {$\bi$};
      \node[] at (-1, -0.25) {$\bj$};
       \node[] at (2.5, -0.25) {$\bm$};
      \node[] at (1, -0.25) {$\bk$};
     \draw[thick, rounded corners, fill=yellow!50] (-3.2, -0.5+2.5) rectangle (-1.8, 0+2.5) {};
       \draw[thick, rounded corners, fill=yellow!50] (-1.7, -0.5+2.5) rectangle (-0.3, 0+2.5) {};
            \draw[thick, rounded corners, fill=yellow!50] (3.2, -0.5+2.5) rectangle (1.8, 0+2.5) {};
       \draw[thick, rounded corners, fill=yellow!50] (1.7, -0.5+2.5) rectangle (0.3, 0+2.5) {};
      \node[] at (-2.5, -0.25+2.5) {$\bi$};
      \node[] at (-1, -0.25+2.5) {$\bk$};
       \node[] at (2.5, -0.25+2.5) {$\bm$};
      \node[] at (1, -0.25+2.5) {$\bj$};
     \draw[thick, rounded corners, fill=cyan!25] (-3.2, -1.5) rectangle (-0.3, -0.75) {};
               \node[] at (-1.75, -1.125) {$\psi_{u_1}$};
         \draw[thick, rounded corners, fill=cyan!25] (3.2, -1.5) rectangle (0.3, -0.75) {};
           \node[] at (1.75, -1.125) {$\psi_{u_2}$};
     \draw[thick, rounded corners, fill=cyan!25] (-3.2, -1.5+4.25) rectangle (-0.3, -0.75+4.25) {};
      \node[] at (-1.75, -1.125+4.25) {$\psi_{w_1}$};
         \draw[thick, rounded corners, fill=cyan!25] (3.2, -1.5+4.25) rectangle (0.3, -0.75+4.25) {};
            \node[] at (1.75, -1.125+4.25) {$\psi_{w_2}$};
              \draw[thick, fill=black]  (-3,-2.1) circle (4pt);
               \node[left] at (-3,-2.1) {$f_1$};
                \draw[thick, fill=black]  (-0.5,-2.1) circle (4pt);
               \node[left] at (-0.5,-2.1) {$f_{r}$};
                \draw[thick, fill=black]  (0.5,-2.1) circle (4pt);
               \node[right] at (0.5,-2.1) {$f_{r+1}$};
                 \draw[thick, fill=black]  (3,-2.1) circle (4pt);
               \node[right] at  (3,-2.1) {$f_{m}$};
                   \draw[thick, rounded corners, fill=yellow!50] (-3.2, -1.5-2) rectangle (-0.3, -0.75-2) {};
               \node[] at (-1.75, -1.125-2) {$1_{\omega}$};
         \draw[thick, rounded corners, fill=yellow!50] (3.2, -1.5-2) rectangle (0.3, -0.75-2) {};
           \node[] at (1.75, -1.125-2) {$1_{\beta-\omega}$};
                   \draw[thick, rounded corners, fill=yellow!50] (-3.2, -1.5+5.25) rectangle (-0.3, -0.75+5.25) {};
               \node[] at (-1.75, -1.125+5.25) {$1_{\omega}$};
         \draw[thick, rounded corners, fill=yellow!50] (3.2, -1.5+5.25) rectangle (0.3, -0.75+5.25) {};
           \node[] at (1.75, -1.125+5.25) {$1_{\beta-\omega}$};
            \node[] at (-2.45,1) {$\cdots$};
             \node[] at (2.55,1) {$\cdots$};
             \node[] at (-0.9,0.3) {$\cdots$};
             \node[] at (0.95,0.3) {$\cdots$};
             \node[] at (-0.9,1.7) {$\cdots$};
             \node[] at (0.95,1.7) {$\cdots$};
              \node[] at (-1.85,-2.1) {$\cdots$};
              \node[] at (2.1,-2.1) {$\cdots$};
\end{tikzpicture}}
\end{align}
Consider a {\em nonzero} element \(z\) of the form (\ref{spanR}, \ref{KLRsepdiag}).
Then since \(0 \neq z \in 1_{\omega , \beta-\omega} \scrr^{\La_{\kapc}}_{\beta}(\mathfrak{sp}_\infty) 1_{\omega , \beta-\omega}\), it must be that \(e(\bi \otimes \bj \otimes \bk \otimes \bm)\) is nonzero in \(\scrr^{\La_{\kapc}}_{\beta}(\mathfrak{sp}_\infty)\).
By \cref{remforzero}, \(e(\bi\otimes  \bj \otimes  \bk \otimes  \bm)\) is nonzero only if there exists \(\nu = \rho + (\pla,\bbmu) \in \ParblockC \) and a standard tableau \({\SSTT} \in \SStd_+(\nu)\) with \(\bfi^{\varphi ({\SSTT})} = \bi \otimes \bj \otimes  \bk  \otimes \bm\).
 But note then that \(\bi\otimes  \bj\) is a word in \(\rho = \Shape({\SSTT} \downarrow_r)\), so it follows that the last entry in \(\bj\) is 0. But, as \(\beta - \omega\) is not supported in \(\alpha_0\), we must have then that \(X\) is trivial (else \(X\) would carry the 0-coloured strand up into the second component as in (\ref{KLRsepdiag}) where it would be annihilated by \(1_{\beta - \omega}\)). Therefore in view of the spanning set (\ref{spanR}), we have that \(1_{\omega , \beta-\omega} \scrr^{\La_{\kapc}}_{\beta}(\mathfrak{sp}_\infty) 1_{\omega , \beta-\omega}\) is spanned by elements of the form
\begin{align*}
1_{\omega, \beta - \omega} (\psi_{u_1} \otimes \psi_{u_2}) y_1^{f_1} \cdots y_m^{f_m} 1_{\omega, \beta - \omega} \in \image \varphi,
\end{align*}
and therefore \(\varphi\) is a surjection.

By \cref{thm:cellular} and the fact that $e(\bfi) c_{\tts \ttt} e(\bfj) = \delta_{\bfi,\bfi^\tts} \delta_{\bfj,\bfi^\ttt} c_{\tts \ttt}$ it is clear that the dimensions of the domain and codomain match, and $\varphi$ is therefore an isomorphism.
\end{proof}

\subsection{The cellular basis under the isomorphism}
We are almost ready to complete the proof the main result of this paper, but first we require a few lemmas which break up the cases of the proof.
 These lemmas appear a bit unmotivated  at this point, so the reader is invited to skip to 
\cref{BOOM} and come back to these as and when they are needed in the proof.

\begin{lem}\label{lem:maxres}
Let $\nu\in \ParblockC$ with 2 addable $r$-nodes and $\tts\in \Std(\nu)$, then $e_\tts \otimes e(r\pm1) = 0 \in \scrr^{\La_{\kapc}}(\mathtt{C}_\infty)$.
\end{lem}

\begin{proof}
First we note that $r\neq 0$ (as $\nu$ has at most one addable $0$-node).
If $r> \kapc$ and $\nu\in\mathscr{P}_n$ has 2 addable $r$-nodes if and only if 
\begin{equation}\label{coutningres}
\sharp \{\text{$(r-1)$-nodes in } \nu\}-2
=
\sharp \{\text{$r$-nodes in } \nu\}
=
\sharp \{\text{$(r+1)$-nodes in } \nu\}.
\end{equation}
Therefore if $\tts \in \Std(\nu) $ and $\ttt \in \Std(\bfi^\ttt)$ then ${\rm Shape}(\tts)=\nu'$ also has 2 addable $r$-nodes.
In particular, $\nu'$ does not have an addable ($r\pm1$)-node.
Therefore there is no tableau whose residue sequence is $\bfi^\tts \otimes (r\pm1)$, and the result follows. 
The cases $r=\kapc$ and $r<\kapc$ can be argued in an identical fashion, except that \eqref{coutningres} must be replaced with the condition 
\begin{equation}\label{coutningres2}
\sharp \{\text{$(r-1)$-nodes in } \nu\}-1
=
\sharp \{\text{$r$-nodes in } \nu\}
=
\sharp \{\text{$(r+1)$-nodes in } \nu\} 
\end{equation}
for $r=\kapc$ and with the condition 
\begin{equation}\label{coutningres3}
\sharp \{\text{$(r-1)$-nodes in } \nu\}-1
=
\sharp \{\text{$r$-nodes in } \nu\}
=
\sharp \{\text{$(r+1)$-nodes in } \nu\}-1 
\end{equation}for $r < \kapc$. 
\end{proof}

\begin{cor}\label{lem:res}
Let $\rho$ and $\pla$ be such that $\rho$ is a rectangular partition with $L(\rho) > L(\pla)$.
Let ${\bf i} = {\bf i}^{\rho+\pla}$.
Let $i \in I$ denote the $\mathtt C$-residue of the lowest addable node of $\rho+\pla$ and let $l\in I$ denote the $\mathtt C$-residue of the second lowest addable node of $\rho+\pla$.
Then
\[
e({\bf i}) \otimes e(l \nearrow m-1) \otimes e(i\searrow m+1) \otimes  e(m-1)=0 \in \scrr^{\La_{\kapc}}(\mathtt{C}_\infty)
\]
for any $l \leq m \leq i$, where we ignore any empty residue sequences.
\end{cor}

\begin{lem}\label{case1lem}
Let $\nu \in \ParblockC$ be a partition with 2 addable $i$-nodes, and let $\ttt \in \Std(\nu)$. 
In $\scrr^{\La_{\kapc}}(\mathtt{C}_\infty)$ we have that
\[
e_{\ttt} \otimes \psi_2\psi_1y_1e(i,i,i-1)\psi_1\psi_2 
= e_{\ttt} \otimes  e(i,i-1,i).
\]
\end{lem}

\begin{proof}
We have that 
\begin{align*}
e_{\ttt} \otimes  (\psi_2\psi_1y_1e(i,i,i-1)) \psi_1\psi_2 
& =
e_{\ttt} \otimes  
( \psi_2y_2\psi_1 ^2 \psi_2 - \psi_2\psi_1   \psi_2 ) e(i,i-1,i) \\
& =
  - e_{\ttt} \otimes 
  \psi_2\psi_1   \psi_2 e(i,i-1,i)  \\
&= 
e_{\ttt} \otimes (1 - \psi_1\psi_2 \psi_1 ) e(i,i-1,i)\\
&  =  
e_{\ttt} \otimes  e(i,i-1,i),
\end{align*}
where the first equality follows by relation (\ref{rel:dotcrossbun2});
the second equality follows by relation (\ref{rel:quadr}); 
the third equality follows from the second case of relation (\ref{rel:braid}); the final equality follows by \cref{lem:maxres} and our assumption that $\nu$ has two addable $i$-nodes.
\end{proof}

\begin{lem}\label{small-lemma}
Let $\nu \in \ParblockC$ be a partition with no addable $(i-1)$-nodes, and let $\ttt \in \Std(\nu)$.
For any $i>0$ we have that 
\[
e_{\ttt} \otimes 
\psi_3 \psi_2 \psi_1 
(e(i) \otimes e(i+1,i,i-1) )
\psi_1 \psi_2 \psi_3 
=
- e_{\ttt} \otimes
 e(i+1,i,i-1) \otimes e(i)
\]
in $\scrr^{\La_{\kapc}}(\mathtt{C}_\infty)$.
\end{lem}

\begin{proof}
This follows by applying relation (\ref{rel:quadr}) to obtain  
\[
e_{\ttt} \otimes
\psi_3 \psi_2 \psi_1 
e(i,i+1,i,i-1)  \psi_1 \psi_2 \psi_3 
=
e_{\ttt} \otimes
\psi_3 \psi_2  
 (y_2 - y_1) e(i+1,i,i,i-1)
  \psi_2 \psi_3,
\]
where the second term on the righthand-side is zero by relation \eqref{rel:quadr}.
By relation (\ref{rel:dotcrossbun2}), we have that the first term is
\[
e_{\ttt} \otimes
 \psi_3 (\psi_2  
y_2 e(i+1,i,i,i-1) )
  \psi_2 \psi_3
  =
e_{\ttt} \otimes  \psi_3  (y_3 \psi_2 -1)e(i+1,i,i,i-1 ) \psi_2 \psi_3
\]
where the first term is zero by relation \eqref{rel:quadr}. 
Finally, by relation \eqref{rel:braid}, we have that 
\[
 e_\ttt \otimes
 e(i+1,i,i-1,i )   
 \psi_3 \psi_2 \psi_3
   =
 e_\ttt \otimes
 e(i+1,i,i-1,i )   
( \psi_2 \psi_3 \psi_2 - 1)
\]
where the first term on the righthand-side is zero since
\begin{align*}
e_\ttt \otimes
e(i+1,i,i-1,i )  \psi_2 \psi_3 \psi_2
&=
e_\ttt \otimes 
\psi_2 e(i+1,i-1,i,i )\psi_3 \psi_2\\
&= 
e_\ttt \otimes 
\psi_2 \psi_1^2e(i+1,i-1,i,i )\psi_3 \psi_2\\
&=
e_\ttt \otimes 
\psi_2 \psi_1e(i-1,i+1,i,i ) \psi_1\psi_3 \psi_2
\end{align*}
by our assumption that $\nu$ has no addable $(i-1)$-nodes.
The result follows.
\end{proof}

\begin{lem}\label{toomanydots} 
Let $\nu=\rho+\pla \in  \ParblockC$ with $|\rho|= r$, and let $\SSTT \in \SStd_+(\nu)$. 
We have that 
\[
y_k 
 \psi^{\SSTT_\nu}_{\SSTT}y_{\SSTT}
  \psi_{\SSTT_\nu}^{\SSTT}= 0 \in \scrr^{\La_{\kapc}}(\mathtt{C}_\infty)
\]
for all $1\leq k \leq r$.
\end{lem}

\begin{proof}
If $\la=\varnothing$ then this follows by \cref{easypeasylemon}.  
If $\pla \neq \varnothing$ then we have that
\[
y_k 
 \psi^{\SSTT_\nu}_{\SSTT}y_{\SSTT }
  \psi_{\SSTT_\nu}^{\SSTT}
  =
  y_k 
 \psi^{\SSTT_\nu}_{\SSTT}y_{\SSTT_\rho \otimes \SSTT_\pla }
  \psi_{\SSTT_\nu}^{\SSTT}= 
  y_k (y_{\SSTT_\rho } \otimes e_{\SSTT_\pla}) \psi^{\SSTT_\nu}_{\SSTT}(e_{\SSTT^\rho} \otimes y_{ \SSTT_\pla })
  \psi_{\SSTT_\nu}^{\SSTT} =
  0 
\]
where the first equality follows by the definition, the second  by the commutativity relations, and the third is by the $\pla = \varnothing$ case.
\end{proof}

The following result will be crucial in matching up the cellular structures in types $\ttc_\infty$ and $\tta_\infty$, because our map $\varphi$ in \cref{thm:homomorphism} sends the natural cell module generator to an element which is not obviously an element of the intended cell ideal. \cref{BOOM} will help us to rewrite this complicated-looking element so that it is transparently inside that cell ideal, and in turn allow us to match those cellular structures.
The proof of the proposition requires some delicate case analysis, as the computation is highly dependent on the residues appearing involved.

\begin{prop}\label{BOOM}
Let $\nu=\rho+\pla \in  \ParblockC$ and let $\SSTS= s_k(\SSTT)$ with $\SSTS,\SSTT \in \SStd_+(\nu)$ and $\ell(w^\SSTS) \leq \ell(w^\SSTT)$.
In $\scrr^{\La_{\kapc}}(\mathtt{C}_\infty)$ we have that 
\[
\psi^{\SSTT_\nu}_{\SSTS}y_{\SSTS }
  \psi_{\SSTT_\nu}^{\SSTS}
  =
  \pm 
  \psi^{\SSTT_\nu}_{\SSTT}y_{\SSTT }
  \psi_{\SSTT_\nu}^{\SSTT}.
\]
\end{prop}

\begin{proof}
Our assumption that $\SSTS,\SSTT$ are semistandard and that $\ell(w^\SSTS) \leq \ell(w^\SSTT)$ implies that  
$\SSTS^{-1}(k)$ is in $\pla$ and that $\SSTS^{-1}(k+1)$ is in $\rho$.
Thus we can assume that 
\[
{\color{magenta}[i_2,j_2]} := \res(\SSTS^{-1}(k))= ({\color{magenta}  i_2\nearrow j_2})
 \qquad
 {\color{gray} [i_1,j_1]} := \res(\SSTS^{-1}(k+1))=({\color{gray} i_1\searrow\nearrow j_1}) 
\]
for some residues ${\color{gray} i_1,j_1}$, $\color{magenta}i_2,j_2$.
There are three subcases to consider, depending on whether
\[
 \deg(\SSTS )  - 
\deg(\SSTT )   =-1,0,1.
\]
In each case we will zoom into a local region of the diagram via the following observation:
\begin{align}\label{whatisD}
  \psi^{\SSTT_\nu}_{\SSTS}y_{\SSTS}
  \psi_{\SSTS_\nu}^{\SSTS}=  
  \psi^{\SSTT_\nu}_{\SSTT}
 ( \psi^{\SSTT }_{\SSTS}
 y_{\SSTS}
  \psi_{\SSTT}^{\SSTS})
  \psi_{\SSTT_\nu}^{\SSTT} 
=
  \psi^{\SSTT_\nu}_{\SSTT}
( y_{\SSTS_{<k} }\otimes D \otimes y_{\SSTS_{>k+1} })
  \psi_{\SSTT_\nu}^{\SSTT} 
\end{align}where the diagram $D$ 
  is equal to a (possibly decorated) double-crossing of a thick $\color{gray}[i_1,j_1]$-strand 
and a thick $\color{magenta}[i_2,j_2]$-strand as shown in \cref{thediagD}. 

\begin{figure}[ht!]
\[
\begin{tikzpicture}[scale=0.285] 
 
 \draw(-1,0) rectangle (11.5*2,8);
\foreach \i in {0,1,2,...,11}
{
\path(\i*2,8) coordinate (X\i);
\path(\i*2,0) coordinate (Y\i);
 }

\path(2.5*2,8) coordinate (X2.5);
\path(8.5*2,8) coordinate (Y8.5);
 
 \draw (X0) node [above] {\scalefont{0.7}$i_1$};
  \draw (X1) node [above] {\scalefont{0.7}$i_1{-}1$};
 
  \draw (X4) node [above] {\scalefont{0.7}$j_1{-}1$};
  \draw (X5) node [above] {\scalefont{0.7}$j_1$};
   \draw (X6) node [above] {\scalefont{0.7}$\color{magenta}i_2$};
      \draw (X7) node [above] {\scalefont{0.7}$\color{magenta}i_2{+}1$};
      
         \draw (X10) node [above] {\scalefont{0.7}$\color{magenta}j_2{-}1$};
                  \draw (X11) node [above] {\scalefont{0.7}$\color{magenta}j_2$};
  
   \filldraw[  fill ,opacity=0.2 ]  (X0)to (Y6) to (Y11) to (X5);
 
  \draw[thick](X0)to (Y6);
  \draw[thick,densely dotted](X1)to (Y7);
    \draw[thick,densely dotted](X2)to (Y8);
        \draw[thick,densely dotted](X3)to (Y9);
 \draw[thick](X5)to (Y11);
  \draw[thick,densely dotted](X4)to (Y10);

   \filldraw[  fill ,opacity=0.3, magenta ]  (X6)to (Y0) to (Y5) to (X11);
   \draw[thick](X6)to (Y0);
  \draw[thick](X11)to (Y5);
 \draw[thick,densely dotted](X7)to (Y1);
 \draw[thick,densely dotted](X10)to (Y4);

 \draw[thick,densely dotted](X8)to (Y2);
 \draw[thick,densely dotted](X9)to (Y3);

\draw(-1,0) rectangle (11.5*2,-6);
\foreach \i in {0,1,2,...,11}
{
\path(\i*2,0) coordinate (X\i);
\path(\i*2,0-6) coordinate (Y\i);
 }
  
   \filldraw[  fill ,opacity=0.2 ]  (X6)to (Y6) to (Y11) to (X11);
 
  \draw[thick](X6)to (Y6);
  \draw[thick,densely dotted](X7)to (Y7);
    \draw[thick,densely dotted](X8)to (Y8);
        \draw[thick,densely dotted](X9)to (Y9);
 \draw[thick](X11)to (Y11);
  \draw[thick,densely dotted](X10)to (Y10);

   \filldraw[  fill ,opacity=0.3, magenta ]  (X0)to (Y0) to (Y5) to (X5);
   \draw[thick](X0)to (Y0);
  \draw[thick](X5)to (Y5);
 \draw[thick,densely dotted](X1)to (Y1);
 \draw[thick,densely dotted](X4)to (Y4);

 \draw[thick,densely dotted](X2)to (Y2);
 \draw[thick,densely dotted](X3)to (Y3);

\draw[line width =3 ,rounded corners,fill =white] (-0.5,-1.5) rectangle (10.5,-4.5) node  [midway]{$y_{\SSTS^{-1}(k)}$} ;

\draw[line width =3 ,rounded corners,fill =white] (11.5,-1.5) rectangle (22.5,-4.5) node  [midway]{$y_{\SSTS^{-1}(k+1)}$} ;

 \draw(-1,-6) rectangle (11.5*2,-6-8);
\foreach \i in {0,1,2,...,11}
{
\path(\i*2,-6-8) coordinate (X\i);
\path(\i*2,-6) coordinate (Y\i);
 }

 
 \draw (X0) node [below] {\scalefont{0.7}$i_1$};
  \draw (X1) node [below] {\scalefont{0.7}$i_1{-}1$};
 
  \draw (X4) node [below] {\scalefont{0.7}$j_1{-}1$};
  \draw (X5) node [below] {\scalefont{0.7}$j_1$};
   \draw (X6) node [below] {\scalefont{0.7}$\color{magenta}i_2$};
      \draw (X7) node [below] {\scalefont{0.7}$\color{magenta}i_2{+}1$};
      
         \draw (X10) node [below] {\scalefont{0.7}$\color{magenta}j_2{-}1$};
                  \draw (X11) node [below] {\scalefont{0.7}$\color{magenta}j_2$};
  
   \filldraw[  fill ,opacity=0.2 ]  (X0)to (Y6) to (Y11) to (X5);
 
  \draw[thick](X0)to (Y6);
  \draw[thick,densely dotted](X1)to (Y7);
    \draw[thick,densely dotted](X2)to (Y8);
        \draw[thick,densely dotted](X3)to (Y9);
 \draw[thick](X5)to (Y11);
  \draw[thick,densely dotted](X4)to (Y10);

   \filldraw[  fill ,opacity=0.3, magenta ]  (X6)to (Y0) to (Y5) to (X11);
   \draw[thick](X6)to (Y0);
  \draw[thick](X11)to (Y5);
 \draw[thick,densely dotted](X7)to (Y1);
 \draw[thick,densely dotted](X10)to (Y4);

 \draw[thick,densely dotted](X8)to (Y2);
 \draw[thick,densely dotted](X9)to (Y3);
\end{tikzpicture} 
\]
 \caption{The diagram $D$ formed of a pair of (possibly decorated) thick strands, coloured grey and pink, which double-cross each other.
}
\label{thediagD}
\end{figure}
 
\smallskip \noindent
{\bf Case 1:  $\deg(\SSTS)= \deg(\SSTT)+1$.}
By the definition of the degree function, this is equivalent to 
$\deg(\SSTT^{-1}(k+1))=0$ (equivalently $i_1+1  \not  \in\color{magenta} [i_2,j_2]$)
and $\deg(\SSTS^{-1}(k))=1$ (equivalently $i_1 \in \color{magenta} [i_2,j_2]$).
We also have $\deg(\SSTT^{-1}(k)) = 0 = \deg(\SSTS^{-1}(k+1))$ in this case.
The conditions   $i_1 \in \color{magenta}[i_2,j_2]$ and $i_1+1 \not \in \color{magenta}[i_2,j_2]$ together imply that 
 $i_1=j_2$. 
The diagram $D$ in this case is equal to a  double-crossing of a thick $\color{gray}[i_1,j_1]$-strand 
and a thick $\color{magenta}[i_2,j_2]$-strand with a single dot in the middle on the $j_2$-strand.  
Thus we can straighten the $j_2$-strand in $D$ by a single application of \cref{case1lem} (note that in so doing, we lose the decoration on this strand);
each of the remaining strands can be straightened using \cref{small-lemma} and the commutativity relations to obtain (up to sign) the undecorated diagram $e_{\SSTT^{-1}(k)} \otimes e_{\SSTT^{-1}(k+1)}$ (as required, since  $\deg(\SSTT^{-1}(k+1))= 0 = \deg(\SSTT^{-1}(k))$). 
This is illustrated via an example in \cref{tableautileeg,case1pic}.

\begin{figure}[ht!]
\[
\begin{tikzpicture}[scale=0.6]
\draw[thick] (0,0)--(5,0)--++(-90:1)
--++(180:1)--++(-90:2)--++(180:4)--(0,0);

\clip (0,0)--(5,0)--++(-90:1)
--++(180:1)--++(-90:2)--++(180:4)--(0,0); 

\fill [opacity=0.2] (0,0) rectangle (3,-3);

\fill [opacity=0.3,magenta] (3,0) rectangle (5,-3);
\draw[thick] (0,0) rectangle (3,-3);

\draw[thick] (0,-1) --++(0:5);
\draw[thick] (0,-2) --++(0:5);

\draw[thick] (1,0) --++(-90:5);
\draw[thick] (2,0) --++(-90:5);
\draw[thick] (4,0) --++(-90:5);

\draw (0.5,-0.5) node {$0$};
\draw (1.5,-0.5) node {$1$};
\draw (2.5,-0.5) node {$2$};
\draw (3.5,-0.5) node {$3$};
\draw (4.5,-0.5) node {$4$};

\draw (0.5,-1.5) node {$1$};
\draw (1.5,-1.5) node {$0$};
\draw (2.5,-1.5) node {$1$};
\draw (3.5,-1.5) node {$2$};
\draw (4.5,-1.5) node {$3$};

\draw (0.5,-2.5) node {$2$};
\draw (1.5,-2.5) node {$1$};
\draw (2.5,-2.5) node {$0$};
\draw (3.5,-2.5) node {$1$};
\draw (4.5,-2.5) node {$2$};
\end{tikzpicture}
\qquad
\begin{tikzpicture}[scale=0.6]
\draw[thick] (0,0)--(5,0)--++(-90:1)
--++(180:1)--++(-90:2)--++(180:4)--(0,0);

\clip (0,0)--(5,0)--++(-90:1)
--++(180:1)--++(-90:2)--++(180:4)--(0,0);

\fill [opacity=0.2] (0,0) rectangle (3,-3);

\fill [opacity=0.3,magenta] (3,0) rectangle (5,-3);
\draw[thick] (0,0) rectangle (3,-3);

\draw[thick] (0,-1) --++(0:5);
\draw[thick] (0,-2) --++(0:5);

\draw[thick] (1,0) --++(-90:5);
\draw[thick] (2,0) --++(-90:5);
\draw[thick] (4,0) --++(-90:5);

 \draw (0.5,-0.5) node {$1$};
\draw (1.5,-0.5) node {$1$};
 \draw (2.5,-0.5) node {$1$};
 \draw (3.5,-0.5) node {$2$};
 \draw (4.5,-0.5) node {$2$};

 \draw (0.5,-1.5) node {$3$};
\draw (1.5,-1.5) node {$3$};
 \draw (2.5,-1.5) node {$3$};
\draw (3.5,-1.5) node {$4$};

\draw (0.5,-2.5) node {$5$};
\draw (1.5,-2.5) node {$5$};
\draw (2.5,-2.5) node {$5$};
\draw (3.5,-2.5) node {$6$};
\draw (4.5,-2.5) node {$2$};
\end{tikzpicture}
\qquad
\begin{tikzpicture}[scale=0.6]
\draw[thick] (0,0)--(5,0)--++(-90:1)
--++(180:1)--++(-90:2)--++(180:4)--(0,0);

\clip (0,0)--(5,0)--++(-90:1)
--++(180:1)--++(-90:2)--++(180:4)--(0,0);

\fill [opacity=0.2] (0,0) rectangle (3,-3);

\fill [opacity=0.3,magenta] (3,0) rectangle (5,-3);
\draw[thick] (0,0) rectangle (3,-3);

\draw[thick] (0,-1) --++(0:5);
\draw[thick] (0,-2) --++(0:5);

\draw[thick] (1,0) --++(-90:5);
\draw[thick] (2,0) --++(-90:5);
\draw[thick] (4,0) --++(-90:5);

 \draw (0.5,-0.5) node {$1$};
\draw (1.5,-0.5) node {$1$};
 \draw (2.5,-0.5) node {$1$};
 \draw (3.5,-0.5) node {$2$};
 \draw (4.5,-0.5) node {$2$};

 \draw (0.5,-1.5) node {$3$};
\draw (1.5,-1.5) node {$3$};
 \draw (2.5,-1.5) node {$3$};
\draw (3.5,-1.5) node {$5$};

\draw (0.5,-2.5) node {$4$};
\draw (1.5,-2.5) node {$4$};
\draw (2.5,-2.5) node {$4$};
\draw (3.5,-2.5) node {$6$};
\draw (4.5,-2.5) node {$2$};
\end{tikzpicture}
\]
\caption{On the left we depict the residues of $\rho +\pla$.
The two semistandard tableaux are examples of $\SSTS = \SSTT^{\rho+\pla}$ and $\SSTT = s_4(\SSTS)$ as in Case 1.
Here $y_\SSTS = y_2y_9$ and $y_\SSTT = y_2$.}
\label{tableautileeg}
\end{figure}
 
\begin{figure}[ht!]
\[
\begin{minipage}{5.9cm}
 \begin{tikzpicture}[scale=0.225]
 \draw(-1,0) rectangle (12.5*2,8);
\foreach \i in {0,1,2,...,12}
{\fill (\i*2,0) circle (4pt);
\fill (\i*2,8) circle (4pt);
\path(\i*2,8) coordinate (X\i);
\path(\i*2,0) coordinate (Y\i);
 }
 
 \draw (X0) node [above] {\scalefont{0.7}$0$};
  \draw (X1) node [above] {\scalefont{0.7}$1$};
   \draw (X2) node [above] {\scalefont{0.7}$2$};

 \draw (X3) node [above] {\scalefont{0.7}$1$};
  \draw (X4) node [above] {\scalefont{0.7}$0$};
   \draw (X5) node [above] {\scalefont{0.7}$1$};

  \draw (X6) node [above] {\scalefont{0.7}$2$};
  \draw (X7) node [above] {\scalefont{0.7}$1$};
   \draw (X8) node [above] {\scalefont{0.7}$0$};

  \draw (X9) node [above] {\scalefont{0.7}$3$};
  \draw (X10) node [above] {\scalefont{0.7}$4$};
   \draw (X11) node [above] {\scalefont{0.7}$2$};
   \draw (X12) node [above] {\scalefont{0.7}$1$};

 \fill[opacity =0.2] (X0)to (Y0) to (Y2) to (X2);
 \draw[thick](X0)to (Y0);
  \draw[thick](X1)to (Y1);
   \draw[thick](X2)to (Y2);

  \fill[opacity =0.2] (X3)to (Y5) to (Y7) to (X5);

 \draw[thick](X3)to (Y5);
  \draw[thick](X4)to (Y6);
   \draw[thick](X5)to (Y7);

  \fill[opacity =0.2] (X6)to (Y8) to (Y10) to (X8);
 \draw[thick](X6)to (Y8);
  \draw[thick](X7)to (Y9);
   \draw[thick](X8)to (Y10);

  \fill[opacity =0.3,magenta] (X9)to (Y3) to (Y4) to (X10);

 \draw[thick](X9) to  (Y3);
 \draw[thick](X10) to  (Y4);

  \draw[thick][opacity =0.3,magenta,line width =4] (X11) to (Y11);

 \draw[thick](X11) to  (Y11);

  \draw[thick][opacity =0.3,magenta,line width =4] (X12) to (Y12);
 \draw[thick](X12) to  (Y12);


 \draw(-1,0) rectangle (12.5*2,-8);
\foreach \i in {0,1,2,...,12}
{\fill (\i*2,0) circle (4pt);
\fill (\i*2,-8) circle (4pt);
\path(\i*2,-8) coordinate (X\i);
\path(\i*2,0) coordinate (Y\i);
\path(\i*2,-4) coordinate (Z\i);
 }

\foreach \i in {0,1,2,3,4,5,6,7,12}
{
 \draw[thick](X\i) to  (Y\i);
}

\foreach \i in {1,8}
{
 \draw[thick,fill=black](Z\i) circle (10pt);
}

  \fill[opacity =0.3,magenta] (X3) rectangle (Y4);
  \fill[opacity =0.2] (X0) rectangle (Y2);
  \fill[opacity =0.2] (X5) rectangle (Y7);

  \draw[thick][opacity =0.3,magenta,line width =4] (X12) to [out=-90,in=90] (Z12) to [out=-90,in=90] (Y12);
 \draw[thick](X12) to [out=-90,in=90]  (Y12);

  \draw[thick][opacity =0.3,magenta,line width =4] (Y11) to [out=-90,in=90] (Z8) to [out=-90,in=90] (X11);
 \draw[thick](Y11) to [out=-90,in=90] (Z8) to [out=-90,in=90] (X11);

\fill[opacity =0.2, line width =3] (Y8) to [out=-60,in=90] (Z9) to [out=-90,in=60] (X8)--(X10)
  to [out=60,in=-90] (Z11)  to [out=90,in=-60] (Y10)
 ;
 \draw[thick](Y8) to [out=-60,in=90] (Z9) to [out=-90,in=60] (X8);
 \draw[thick](Y9) to [out=-60,in=90] (Z10) to [out=-90,in=60] (X9);
  \draw[thick](Y10) to [out=-60,in=90] (Z11) to [out=-90,in=60] (X10);


 \draw(-1,-8) rectangle (12.5*2,-16);
\foreach \i in {0,1,2,...,12}
{\fill (\i*2,-16) circle (4pt);
\fill (\i*2,-8) circle (4pt);
\path(\i*2,-16) coordinate (X\i);
\path(\i*2,0-8) coordinate (Y\i);
 }
 
 \draw (X0) node [below] {\scalefont{0.7}$0$};
  \draw (X1) node [below] {\scalefont{0.7}$1$};
   \draw (X2) node [below] {\scalefont{0.7}$2$};

 \draw (X3) node [below] {\scalefont{0.7}$1$};
  \draw (X4) node [below] {\scalefont{0.7}$0$};
   \draw (X5) node [below] {\scalefont{0.7}$1$};

  \draw (X6) node [below] {\scalefont{0.7}$2$};
  \draw (X7) node [below] {\scalefont{0.7}$1$};
   \draw (X8) node [below] {\scalefont{0.7}$0$};

  \draw (X9) node [below] {\scalefont{0.7}$3$};
  \draw (X10) node [below] {\scalefont{0.7}$4$};
   \draw (X11) node [below] {\scalefont{0.7}$2$};
   \draw (X12) node [below] {\scalefont{0.7}$1$};

 \fill[opacity =0.2] (X0)to (Y0) to (Y2) to (X2);
 \draw[thick](X0)to (Y0);
  \draw[thick](X1)to (Y1);
   \draw[thick](X2)to (Y2);

  \fill[opacity =0.2] (X3)to (Y5) to (Y7) to (X5);

 \draw[thick](X3)to (Y5);
  \draw[thick](X4)to (Y6);
   \draw[thick](X5)to (Y7);

  \fill[opacity =0.2] (X6)to (Y8) to (Y10) to (X8);
 \draw[thick](X6)to (Y8);
  \draw[thick](X7)to (Y9);
   \draw[thick](X8)to (Y10);

  \fill[opacity =0.3,magenta] (X9)to (Y3) to (Y4) to (X10);

 \draw[thick](X9) to  (Y3);
 \draw[thick](X10) to  (Y4);

  \draw[thick][opacity =0.3,magenta,line width =4] (X11) to (Y11);

 \draw[thick](X11) to  (Y11);

  \draw[thick][opacity =0.3,magenta,line width =4] (X12) to (Y12);
 \draw[thick](X12) to  (Y12);
\end{tikzpicture}
\end{minipage}
\;\;= \;\;
\begin{minipage}{5.625cm}
 \begin{tikzpicture}[scale=0.225]
 \draw(-1,0) rectangle (12.5*2,8);
\foreach \i in {0,1,2,...,12}
{\fill (\i*2,0) circle (4pt);
\fill (\i*2,8) circle (4pt);
\path(\i*2,8) coordinate (X\i);
\path(\i*2,0) coordinate (Y\i);
 }
 
 \draw (X0) node [above] {\scalefont{0.7}$0$};
  \draw (X1) node [above] {\scalefont{0.7}$1$};
   \draw (X2) node [above] {\scalefont{0.7}$2$};

 \draw (X3) node [above] {\scalefont{0.7}$1$};
  \draw (X4) node [above] {\scalefont{0.7}$0$};
   \draw (X5) node [above] {\scalefont{0.7}$1$};

  \draw (X6) node [above] {\scalefont{0.7}$2$};
  \draw (X7) node [above] {\scalefont{0.7}$1$};
   \draw (X8) node [above] {\scalefont{0.7}$0$};

  \draw (X9) node [above] {\scalefont{0.7}$3$};
  \draw (X10) node [above] {\scalefont{0.7}$4$};
   \draw (X11) node [above] {\scalefont{0.7}$2$};
   \draw (X12) node [above] {\scalefont{0.7}$1$};

 \fill[opacity =0.2] (X0)to (Y0) to (Y2) to (X2);
 \draw[thick](X0)to (Y0);
  \draw[thick](X1)to (Y1);
   \draw[thick](X2)to (Y2);

  \fill[opacity =0.2] (X3)to (Y5) to (Y7) to (X5);

 \draw[thick](X3)to (Y5);
  \draw[thick](X4)to (Y6);
   \draw[thick](X5)to (Y7);

  \fill[opacity =0.2] (X6)to (Y8) to (Y10) to (X8);
 \draw[thick](X6)to (Y8);
  \draw[thick](X7)to (Y9);
   \draw[thick](X8)to (Y10);

  \fill[opacity =0.3,magenta] (X9)to (Y3) to (Y4) to (X10);

 \draw[thick](X9) to  (Y3);
 \draw[thick](X10) to  (Y4);

  \draw[thick][opacity =0.3,magenta,line width =4] (X11) to (Y11);

 \draw[thick](X11) to  (Y11);

  \draw[thick][opacity =0.3,magenta,line width =4] (X12) to (Y12);
 \draw[thick](X12) to  (Y12);


 \draw(-1,0) rectangle (12.5*2,-8);
\foreach \i in {0,1,2,...,12}
{\fill (\i*2,0) circle (4pt);
\fill (\i*2,-8) circle (4pt);
\path(\i*2,-8) coordinate (X\i);
\path(\i*2,0) coordinate (Y\i);
\path(\i*2,-4) coordinate (Z\i);
 }

\foreach \i in {0,1,2,3,4,5,6,7,12}
{
 \draw[thick](X\i) to  (Y\i);
}

\foreach \i in {1}
{
 \draw[thick,fill=black](Z\i) circle (10pt);
}

  \fill[opacity =0.3,magenta] (X3) rectangle (Y4);
  \fill[opacity =0.2] (X0) rectangle (Y2);
  \fill[opacity =0.2] (X5) rectangle (Y7);

  \draw[thick][opacity =0.3,magenta,line width =4] (X12) to [out=-90,in=90] (Z12) to [out=-90,in=90] (Y12);
 \draw[thick](X12) to [out=-90,in=90]  (Y12);

  \draw[thick][opacity =0.3,magenta,line width =4] (Y11) to [out=-90,in=90] (Z11) to [out=-90,in=90] (X11);
 \draw[thick](Y11) to [out=-90,in=90] (Z11) to [out=-90,in=90] (X11);

\fill[opacity =0.2, line width =3] (Y8) --(Z8)  --(X8)
--(X10)  -- (Y10)
 ;
 \draw[thick](Y8) --(X8);
 \draw[thick](Y9) --(X9);
  \draw[thick](Y10)--(X10);


 \draw(-1,-8) rectangle (12.5*2,-16);
\foreach \i in {0,1,2,...,12}
{\fill (\i*2,-16) circle (4pt);
\fill (\i*2,-8) circle (4pt);
\path(\i*2,-16) coordinate (X\i);
\path(\i*2,0-8) coordinate (Y\i);
 }
 
 \draw (X0) node [below] {\scalefont{0.7}$0$};
  \draw (X1) node [below] {\scalefont{0.7}$1$};
   \draw (X2) node [below] {\scalefont{0.7}$2$};

 \draw (X3) node [below] {\scalefont{0.7}$1$};
  \draw (X4) node [below] {\scalefont{0.7}$0$};
   \draw (X5) node [below] {\scalefont{0.7}$1$};

  \draw (X6) node [below] {\scalefont{0.7}$2$};
  \draw (X7) node [below] {\scalefont{0.7}$1$};
   \draw (X8) node [below] {\scalefont{0.7}$0$};

  \draw (X9) node [below] {\scalefont{0.7}$3$};
  \draw (X10) node [below] {\scalefont{0.7}$4$};
   \draw (X11) node [below] {\scalefont{0.7}$2$};
   \draw (X12) node [below] {\scalefont{0.7}$1$};

 \fill[opacity =0.2] (X0)to (Y0) to (Y2) to (X2);
 \draw[thick](X0)to (Y0);
  \draw[thick](X1)to (Y1);
   \draw[thick](X2)to (Y2);

  \fill[opacity =0.2] (X3)to (Y5) to (Y7) to (X5);

 \draw[thick](X3)to (Y5);
  \draw[thick](X4)to (Y6);
   \draw[thick](X5)to (Y7);

  \fill[opacity =0.2] (X6)to (Y8) to (Y10) to (X8);
 \draw[thick](X6)to (Y8);
  \draw[thick](X7)to (Y9);
   \draw[thick](X8)to (Y10);

  \fill[opacity =0.3,magenta] (X9)to (Y3) to (Y4) to (X10);

 \draw[thick](X9) to  (Y3);
 \draw[thick](X10) to  (Y4);

  \draw[thick][opacity =0.3,magenta,line width =4] (X11) to (Y11);

 \draw[thick](X11) to  (Y11);

  \draw[thick][opacity =0.3,magenta,line width =4] (X12) to (Y12);
 \draw[thick](X12) to  (Y12);
\end{tikzpicture}\end{minipage}
\]
\caption{We depict 
$ \psi^{\SSTT_\nu}_{\SSTT}
 ( \psi^{\SSTT }_{\SSTS}
 y_{\SSTS}
  \psi_{\SSTT}^{\SSTS})
  \psi_{\SSTT_\nu}^{\SSTT}
=
\psi^{\SSTT_\nu}_{\SSTT}y_{\SSTT}
\psi_{\SSTT_\nu}^{\SSTT}$ for $\SSTS,\SSTT$ as in \cref{tableautileeg}.
This is an example of Case 1 where $\deg(\SSTS)= \deg(\SSTT) + 1$. 
These diagrams differ only in the region between the 9th, 10th, 11th, and 12th strands in the central rectangle;
on the left this region is equal to the diagram $D=\psi_3
 \psi_2\psi_1y_1e(2,2,1,0)\psi_1\psi_2\psi_3$.
On the right this region is equal to $e(2,1,0,2)$.
This equality can be deduced by a single application of \cref{case1lem} and the commutativity relations.}
\label{case1pic}
\end{figure}

\smallskip \noindent
{\bf Case 2:  $\deg(\SSTS)= \deg(\SSTT)-1$.} 
By the definition of the degree function, this is equivalent to 
$\deg(\SSTT^{-1}(k+1))=1$ (equivalently  $i_1+1     \in\color{magenta} [i_2,j_2]$)
and $\deg(\SSTS^{-1}(k))=0$ (equivalently  $i_1 \not \in \color{magenta} [i_2,j_2]$).
We also have $\deg(\SSTT^{-1}(k)) = 0 = \deg(\SSTS^{-1}(k+1))$ in this case.
The conditions  $i_1+1 \in\color{magenta} [i_2,j_2]$ and
$i_1 \not \in \color{magenta} [i_2,j_2]$ together imply that $i_2= i_1+1$.
The diagram $D$ in this case is equal to a double-crossing of a thick $\color{gray}[i_1,j_1]$-strand 
and a thick $\color{magenta}[i_2,j_2]$-strand with no decorations.
We can straighten all the $x$-strands for  $i_2 < x\leq j_2$ in $D$ by  the commutativity relations.
We can straighten the $i_2 $-strand in $D$ using  the fourth case of relation \eqref{rel:quadr} (and the commutativity relations) and in so-doing obtain
\[
D =
(y_1 e_{\SSTT^{-1}(k)}\otimes e_{\SSTT^{-1}(k+1)})
- 
( e_{\SSTT^{-1}(k)}\otimes  y_1 e_{\SSTT^{-1}(k+1)}).
\]
The second term is of the required form.
For the first term we observe that when we plug it back into the wider diagram of \eqref{whatisD} we obtain 
\[
\psi^{\SSTT_\nu}_{\SSTT}
( y_{\SSTT_{<k} }\otimes  
(y_1 e_{\SSTT^{-1}(k)}\otimes e_{\SSTT^{-1}(k+1)} )
\otimes
 y_{\SSTT_{>k+1} })
\psi_{\SSTT_\nu}^{\SSTT} =0
\]
by \cref{toomanydots} and so this first term makes no contribution.
This is illustrated via an example in \cref{tableautileeg2,case2pic}.

\begin{figure}[ht!]
\[
\begin{tikzpicture}[scale=0.6]
\draw[thick] (0,0)--(5,0)--++(-90:1)
--++(180:1)--++(-90:2)--++(180:4)--(0,0);

\clip (0,0)--(5,0)--++(-90:1)
--++(180:1)--++(-90:2)--++(180:4)--(0,0);

\fill [opacity=0.2] (0,0) rectangle (3,-3);

\fill [opacity=0.3,magenta] (3,0) rectangle (5,-3);
\draw[thick] (0,0) rectangle (3,-3);

\draw[thick] (0,-1) --++(0:5);
\draw[thick] (0,-2) --++(0:5);

\draw[thick] (1,0) --++(-90:5);
\draw[thick] (2,0) --++(-90:5);
\draw[thick] (4,0) --++(-90:5);

\draw (0.5,-0.5) node {$0$};
\draw (1.5,-0.5) node {$1$};
\draw (2.5,-0.5) node {$2$};
\draw (3.5,-0.5) node {$3$};
\draw (4.5,-0.5) node {$4$};

\draw (0.5,-1.5) node {$1$};
\draw (1.5,-1.5) node {$0$};
\draw (2.5,-1.5) node {$1$};
\draw (3.5,-1.5) node {$2$};
\draw (4.5,-1.5) node {$3$};

\draw (0.5,-2.5) node {$2$};
\draw (1.5,-2.5) node {$1$};
\draw (2.5,-2.5) node {$0$};
\draw (3.5,-2.5) node {$1$};
\draw (4.5,-2.5) node {$2$};
\end{tikzpicture}
\qquad
\begin{tikzpicture}[scale=0.6]
\draw[thick] (0,0)--(5,0)--++(-90:1)
--++(180:1)--++(-90:2)--++(180:4)--(0,0);

\clip (0,0)--(5,0)--++(-90:1)
--++(180:1)--++(-90:2)--++(180:4)--(0,0);

\fill [opacity=0.2] (0,0) rectangle (3,-3);

\fill [opacity=0.3,magenta] (3,0) rectangle (5,-3);
\draw[thick] (0,0) rectangle (3,-3);

\draw[thick] (0,-1) --++(0:5);
\draw[thick] (0,-2) --++(0:5);

\draw[thick] (1,0) --++(-90:5);
\draw[thick] (2,0) --++(-90:5);
\draw[thick] (4,0) --++(-90:5);

 \draw (0.5,-0.5) node {$1$};
\draw (1.5,-0.5) node {$1$};
 \draw (2.5,-0.5) node {$1$};
 \draw (3.5,-0.5) node {$3$};
 \draw (4.5,-0.5) node {$3$};

 \draw (0.5,-1.5) node {$2$};
\draw (1.5,-1.5) node {$2$};
 \draw (2.5,-1.5) node {$2$};
\draw (3.5,-1.5) node {$5$};
\draw (4.5,-1.5) node {$3$};

\draw (0.5,-2.5) node {$ 4$};
\draw (1.5,-2.5) node {$4$};
\draw (2.5,-2.5) node {$ 4$};
\draw (3.5,-2.5) node {$6$};
\draw (4.5,-2.5) node {$2$};
\end{tikzpicture}
\qquad
\begin{tikzpicture}[scale=0.6]
\draw[thick] (0,0)--(5,0)--++(-90:1)
--++(180:1)--++(-90:2)--++(180:4)--(0,0);

\clip (0,0)--(5,0)--++(-90:1)
--++(180:1)--++(-90:2)--++(180:4)--(0,0);

\fill [opacity=0.2] (0,0) rectangle (3,-3);

\fill [opacity=0.3,magenta] (3,0) rectangle (5,-3);
\draw[thick] (0,0) rectangle (3,-3);

\draw[thick] (0,-1) --++(0:5);
\draw[thick] (0,-2) --++(0:5);

\draw[thick] (1,0) --++(-90:5);
\draw[thick] (2,0) --++(-90:5);
\draw[thick] (4,0) --++(-90:5);

 \draw (0.5,-0.5) node {$1$};
\draw (1.5,-0.5) node {$1$};
 \draw (2.5,-0.5) node {$1$};
 \draw (3.5,-0.5) node {$4$};
 \draw (4.5,-0.5) node {$4$};

 \draw (0.5,-1.5) node {$2$};
\draw (1.5,-1.5) node {$2$};
 \draw (2.5,-1.5) node {$2$};
\draw (3.5,-1.5) node {$5$};
\draw (4.5,-1.5) node {$3$};

\draw (0.5,-2.5) node {$ 3$};
\draw (1.5,-2.5) node {$3$};
\draw (2.5,-2.5) node {$ 3$};
\draw (3.5,-2.5) node {$6$};
\draw (4.5,-2.5) node {$2$};
\end{tikzpicture}
\]
\caption{On the left we depict the $\tt C$-residues of $\rho + \pla$.
The two semistandard tableaux are examples of $\SSTS$ and $\SSTT = \SSTT_{\rho+\pla}
=s_{3}(\SSTS)$ as in Case 2.
Here  $y_\SSTS = y_2 $ and $y_\SSTT= y_2y_7$.}
\label{tableautileeg2}\end{figure}

\begin{figure}[ht!]
\[
\begin{minipage}{5.9cm}
 \begin{tikzpicture}[scale=0.225]
 \draw(-1,0) rectangle (12.5*2,8);
\foreach \i in {0,1,2,...,12}
{\fill (\i*2,0) circle (4pt);
\fill (\i*2,8) circle (4pt);
\path(\i*2,8) coordinate (X\i);
\path(\i*2,0) coordinate (Y\i);
 }
 
 \draw (X0) node [above] {\scalefont{0.7}$0$};
  \draw (X1) node [above] {\scalefont{0.7}$1$};
   \draw (X2) node [above] {\scalefont{0.7}$2$};

 \draw (X3) node [above] {\scalefont{0.7}$1$};
  \draw (X4) node [above] {\scalefont{0.7}$0$};
   \draw (X5) node [above] {\scalefont{0.7}$1$};

  \draw (X6) node [above] {\scalefont{0.7}$2$};
  \draw (X7) node [above] {\scalefont{0.7}$1$};
   \draw (X8) node [above] {\scalefont{0.7}$0$};

  \draw (X9) node [above] {\scalefont{0.7}$3$};
  \draw (X10) node [above] {\scalefont{0.7}$4$};
   \draw (X11) node [above] {\scalefont{0.7}$2$};
   \draw (X12) node [above] {\scalefont{0.7}$1$};

 \fill[opacity =0.2] (X0)to (Y0) to (Y2) to (X2);
 \draw[thick](X0)to (Y0);
  \draw[thick](X1)to (Y1);
   \draw[thick](X2)to (Y2);

  \fill[opacity =0.2] (X3)to (Y3) to (Y5) to (X5);

 \draw[thick](X3)to (Y3);
  \draw[thick](X4)to (Y4);
   \draw[thick](X5)to (Y5);

  \fill[opacity =0.2] (X6)to (Y8) to (Y10) to (X8);
 \draw[thick](X6)to (Y8);
  \draw[thick](X7)to (Y9);
   \draw[thick](X8)to (Y10);

  \fill[opacity =0.3,magenta] (X9)to (Y6) to (Y7) to (X10);

 \draw[thick](X9) to  (Y6);
 \draw[thick](X10) to  (Y7);

  \draw[thick][opacity =0.3,magenta,line width =4] (X11) to (Y11);

 \draw[thick](X11) to  (Y11);

  \draw[thick][opacity =0.3,magenta,line width =4] (X12) to (Y12);
 \draw[thick](X12) to  (Y12);


 \draw(-1,0) rectangle (12.5*2,-8);
\foreach \i in {0,1,2,...,12}
{\fill (\i*2,0) circle (4pt);
\fill (\i*2,-8) circle (4pt);
\path(\i*2,-8) coordinate (X\i);
\path(\i*2,0) coordinate (Y\i);
\path(\i*2,-4) coordinate (Z\i);
 }

\foreach \i in {0,1,2,3,4,5,6,7,8,9,10,12}
{
 \draw[thick](X\i) to  (Y\i);
}

\foreach \i in {1}
{
 \draw[thick,fill=black](Z\i) circle (10pt);
}

  \fill[opacity =0.3,magenta] (X7) rectangle (Y6);

  \fill[opacity =0.2] (X0) rectangle (Y2);

  \fill[opacity =0.2] (X3)to (Y3) to (Y5) to (X5);

 \draw[thick](X3)to (Y3);
  \draw[thick](X4)to (Y4);
   \draw[thick](X5)to (Y5);

  \draw[thick][opacity =0.3,magenta,line width =4] (X11) to [out=-90,in=90] (Z11) to [out=-90,in=90] (Y11);
 \draw[thick](X11) to [out=-90,in=90]  (Y11);

  \draw[thick][opacity =0.3,magenta,line width =4] (X12) to [out=-90,in=90] (Z12) to [out=-90,in=90] (Y12);
 \draw[thick](X12) to [out=-90,in=90]  (Y12);


 \fill[opacity =0.2] (X8)to (Y8) to (Y10) to (X10);
 


 \draw(-1,-8) rectangle (12.5*2,-16);
\foreach \i in {0,1,2,...,12}
{\fill (\i*2,-16) circle (4pt);
\fill (\i*2,-8) circle (4pt);
\path(\i*2,-16) coordinate (X\i);
\path(\i*2,0-8) coordinate (Y\i);
 }
 
 \draw (X0) node [below] {\scalefont{0.7}$0$};
  \draw (X1) node [below] {\scalefont{0.7}$1$};
   \draw (X2) node [below] {\scalefont{0.7}$2$};

 \draw (X3) node [below] {\scalefont{0.7}$1$};
  \draw (X4) node [below] {\scalefont{0.7}$0$};
   \draw (X5) node [below] {\scalefont{0.7}$1$};

  \draw (X6) node [below] {\scalefont{0.7}$2$};
  \draw (X7) node [below] {\scalefont{0.7}$1$};
   \draw (X8) node [below] {\scalefont{0.7}$0$};

  \draw (X9) node [below] {\scalefont{0.7}$3$};
  \draw (X10) node [below] {\scalefont{0.7}$4$};
   \draw (X11) node [below] {\scalefont{0.7}$2$};
   \draw (X12) node [below] {\scalefont{0.7}$1$};

 \fill[opacity =0.2] (X0)to (Y0) to (Y2) to (X2);
 \draw[thick](X0)to (Y0);
  \draw[thick](X1)to (Y1);
   \draw[thick](X2)to (Y2);

  \fill[opacity =0.2] (X3)to (Y3) to (Y5) to (X5);

 \draw[thick](X3)to (Y3);
  \draw[thick](X4)to (Y4);
   \draw[thick](X5)to (Y5);

  \fill[opacity =0.2] (X6)to (Y8) to (Y10) to (X8);
 \draw[thick](X6)to (Y8);
  \draw[thick](X7)to (Y9);
   \draw[thick](X8)to (Y10);

  \fill[opacity =0.3,magenta] (X9)to (Y6) to (Y7) to (X10);

 \draw[thick](X9) to  (Y6);
 \draw[thick](X10) to  (Y7);

  \draw[thick][opacity =0.3,magenta,line width =4] (X11) to (Y11);

 \draw[thick](X11) to  (Y11);

  \draw[thick][opacity =0.3,magenta,line width =4] (X12) to (Y12);
 \draw[thick](X12) to  (Y12);

\end{tikzpicture}\end{minipage}
\;\;= \; - \;
\begin{minipage}{5.9cm}
 \begin{tikzpicture}[scale=0.225]
 \draw(-1,0) rectangle (12.5*2,8);
\foreach \i in {0,1,2,...,12}
{\fill (\i*2,0) circle (4pt);
\fill (\i*2,8) circle (4pt);
\path(\i*2,8) coordinate (X\i);
\path(\i*2,0) coordinate (Y\i);
 }
 
 \draw (X0) node [above] {\scalefont{0.7}$0$};
  \draw (X1) node [above] {\scalefont{0.7}$1$};
   \draw (X2) node [above] {\scalefont{0.7}$2$};

 \draw (X3) node [above] {\scalefont{0.7}$1$};
  \draw (X4) node [above] {\scalefont{0.7}$0$};
   \draw (X5) node [above] {\scalefont{0.7}$1$};

  \draw (X6) node [above] {\scalefont{0.7}$2$};
  \draw (X7) node [above] {\scalefont{0.7}$1$};
   \draw (X8) node [above] {\scalefont{0.7}$0$};

  \draw (X9) node [above] {\scalefont{0.7}$3$};
  \draw (X10) node [above] {\scalefont{0.7}$4$};
   \draw (X11) node [above] {\scalefont{0.7}$2$};
   \draw (X12) node [above] {\scalefont{0.7}$1$};

\foreach \i in {0,1,2,3,4,5,6,7,8,9,10,12}
{
 \draw[thick](X\i) to  (Y\i);
}

\foreach \i in {1,9}
{
 \draw[thick,fill=black](Z\i) circle (10pt);
}

  \fill[opacity =0.3,magenta] (X9) rectangle (Y10);

   \fill[opacity =0.2] (X6)to (Y6) to (Y8) to (X8);

  \fill[opacity =0.2] (X0) rectangle (Y2);

  \fill[opacity =0.2] (X3)to (Y3) to (Y5) to (X5);

 \draw[thick](X3)to (Y3);
  \draw[thick](X4)to (Y4);
   \draw[thick](X5)to (Y5);

  \draw[thick][opacity =0.3,magenta,line width =4] (X11)  --(Y11);
 \draw[thick](X11) --  (Y11);

  \draw[thick][opacity =0.3,magenta,line width =4] (X12)  to (Y12);
 \draw[thick](X12) --  (Y12);


 \draw(-1,0) rectangle (12.5*2,-8);
\foreach \i in {0,1,2,...,12}
{\fill (\i*2,0) circle (4pt);
\fill (\i*2,-8) circle (4pt);
\path(\i*2,-8) coordinate (X\i);
\path(\i*2,0) coordinate (Y\i);
\path(\i*2,-4) coordinate (Z\i);
 }

\foreach \i in {0,1,2,3,4,5,6,7,8,9,10,12}
{
 \draw[thick](X\i) to  (Y\i);
}

\foreach \i in {1}
{
 \draw[thick,fill=black](Z\i) circle (10pt);
}

\fill[opacity =0.3,magenta] (X9) rectangle (Y10);

\fill[opacity =0.2] (X6)to (Y6) to (Y8) to (X8);

\fill[opacity =0.2] (X0) rectangle (Y2);

\fill[opacity =0.2] (X3)to (Y3) to (Y5) to (X5);

 \draw[thick](X3)to (Y3);
  \draw[thick](X4)to (Y4);
   \draw[thick](X5)to (Y5);

\draw[thick][opacity =0.3,magenta,line width =4] (X11) to [out=-90,in=90] (Z11) to [out=-90,in=90] (Y11);
 \draw[thick](X11) to [out=-90,in=90]  (Y11);

  \draw[thick][opacity =0.3,magenta,line width =4] (X12) to [out=-90,in=90] (Z12) to [out=-90,in=90] (Y12);
 \draw[thick](X12) to [out=-90,in=90]  (Y12);


 \draw(-1,-8) rectangle (12.5*2,-16);
\foreach \i in {0,1,2,...,12}
{\fill (\i*2,-16) circle (4pt);
\fill (\i*2,-8) circle (4pt);
\path(\i*2,-16) coordinate (X\i);
\path(\i*2,0-8) coordinate (Y\i);
 }
 
 \draw (X0) node [below] {\scalefont{0.7}$0$};
  \draw (X1) node [below] {\scalefont{0.7}$1$};
   \draw (X2) node [below] {\scalefont{0.7}$2$};

 \draw (X3) node [below] {\scalefont{0.7}$1$};
  \draw (X4) node [below] {\scalefont{0.7}$0$};
   \draw (X5) node [below] {\scalefont{0.7}$1$};

  \draw (X6) node [below] {\scalefont{0.7}$2$};
  \draw (X7) node [below] {\scalefont{0.7}$1$};
   \draw (X8) node [below] {\scalefont{0.7}$0$};

  \draw (X9) node [below] {\scalefont{0.7}$3$};
  \draw (X10) node [below] {\scalefont{0.7}$4$};
   \draw (X11) node [below] {\scalefont{0.7}$2$};
   \draw (X12) node [below] {\scalefont{0.7}$1$};

\foreach \i in {0,1,2,3,4,5,6,7,8,9,10,12}
{
 \draw[thick](X\i) to  (Y\i);
}

\foreach \i in {1}
{
 \draw[thick,fill=black](Z\i) circle (10pt);
}

\fill[opacity =0.2] (X0) rectangle (Y2);

\fill[opacity =0.2] (X3)to (Y3) to (Y5) to (X5);

 \draw[thick](X3)to (Y3);
  \draw[thick](X4)to (Y4);
   \draw[thick](X5)to (Y5);

  \fill[opacity =0.3,magenta] (X9) rectangle (Y10);

\fill[opacity =0.2] (X6)to (Y6) to (Y8) to (X8);

  \draw[thick][opacity =0.3,magenta,line width =4] (X11)  --(Y11);
 \draw[thick](X11) --  (Y11);

  \draw[thick][opacity =0.3,magenta,line width =4] (X12)  to (Y12);
 \draw[thick](X12) --  (Y12);
\end{tikzpicture}\end{minipage}
\]

\caption{We depict 
$\psi^{\SSTT_\nu}_ {\SSTS} y_{\SSTS} \psi_{\SSTT_\nu}^{\SSTS} 
=
- \psi^{\SSTT_\nu}_{\SSTT} y_{\SSTT} \psi_{\SSTT_\nu}^{\SSTT}$ for $\SSTS, \SSTT$ as in \cref{tableautileeg2}.
This is an example of Case 2 where $\deg(\SSTS)= \deg(\SSTT) - 1$.
These diagrams differ only in the region between the 7th, 8th 9th, 10th, and 11th strands; on the left this region is equal to the diagram 
$D = \psi_3\psi_2\psi_4\psi_1\psi_3\psi_2 e(3,4,2,1,0) \psi_2
\psi_3  \psi_1 \psi_4 \psi_2\psi_3$.
On the right this region is equal to $ y_4 e(2,1,0,3,4)$.
This equality can be deduced by many applications of the commutativity relations, with the final double crossing 
 being taken care of by relation \eqref{rel:quadr}, resulting in a $y_2y_7 e_{\SSTT_{\rho}} - y_2y_{10}e_{\SSTT_{\rho}}$, the former of which  is zero by \cref{toomanydots}.}
\label{case2pic}
\end{figure}

\smallskip \noindent
{\bf Case 3: $\deg(\SSTS)= \deg(\SSTT)$.}
There are three subcases to consider:
either $\deg(\SSTT^{-1}(k+1))=0=\deg(\SSTS^{-1}(k))$, $\deg(\SSTT^{-1}(k+1))=1=\deg(\SSTS^{-1}(k))$, or $\deg(\SSTT^{-1}(k))=1=\deg(\SSTS^{-1}(k+1))$.

\smallskip
\noindent{\bf Case 3a. } 
We first suppose that $\deg(\SSTT^{-1}(k+1))=0$ (equivalently  $i_1  \not \in\color{magenta} [i_2,j_2]$)
and $\deg(\SSTS^{-1}(k))=0$ (equivalently  $i_1+1 \not \in \color{magenta} [i_2,j_2]$).
We also have $\deg(\SSTT^{-1}(k)) = 0 = \deg(\SSTS^{-1}(k+1))$ in this case.
The conditions $i_1,i_1+1 \not \in \color{magenta}[i_2,j_2]$ together imply that ${\color{magenta}[i_2,j_2]} \subseteq \color{gray} (j_1,i_1)$ or that $\res(\SSTT^{-1}(k))$ and $\res(\SSTT^{-1}(k+1))$ are well-separated.
In the latter case $y_\SSTS = y_\SSTT$ and all the crossings can be undone trivially and the result follows from the definitions.
In the former case the diagram $D$ is equal to  
 a  double-crossing of a thick $\color{gray}[i_1,j_1]$-strand 
and 
a thick $\color{magenta}[i_2,j_2]$-strand with no decorations.  
For each $x \in {\color{magenta} [i_2, j_2]}$, starting with $x=j_2$ maximal,
 we can pull the $x$-strand from left-to-right using only the commutation relations and a single application of \cref{small-lemma}.

\smallskip

\noindent{\bf Case 3b. } 
We now suppose that $\deg(\SSTT^{-1}(k+1))=1$ (equivalently  $i_1+1 \in\color{magenta} [i_2,j_2]$)
and $\deg(\SSTS^{-1}(k))=1$  (equivalently  $i_1 \in \color{magenta} [i_2,j_2]$).
We also have $\deg(\SSTT^{-1}(k)) = 0 = \deg(\SSTS^{-1}(k+1))$ in this case.
In this case the diagram $D$ is equal to a double-crossing of a thick $\color{gray}[i_1,j_1]$-strand and a thick $\color{magenta}[i_2,j_2]$-strand with a decoration on the  middle of the $i_1$-strand.  
For each $i_1 +1 < x \leq j_2$  starting with $x=j_2$ maximal,
we can pull the $x$-strand from left-to-right  using only the commutation relations.
This allows us to focus on the $j_2 = i_1 + 1$ case.
We first focus on the subdiagram, $D'$,  consisting of the double-crossing of the first two grey strands (with residues $i_1$ and $i_1-1$) and the final two pink strands (with residues $i_1$ and $i_1+1$).
That is,
\begin{align*}
e_{\SSTT_{<k}} \otimes e ({\color{magenta}i_2 \nearrow i_1-1}	)  \otimes 
\psi_2 \psi_1 \psi_3 \psi_2
(y_1e({\color{magenta}i_1,i_1+1} 			, 		{\color{gray}	i_1,i_1-1}))
\psi_2 \psi_3 \psi_1 \psi_2 )
\otimes 
  e({\color{gray} i_1-2\searrow\nearrow 	j_1} 	)\otimes e_{\SSTT_{>k+1}}
\end{align*}
and applying  relation \eqref{rel:quadr} and the commutativity relations we obtain 
\begin{align*}
e_{\SSTT_{<k}} \otimes e ({\color{magenta}i_2 \nearrow i_1-1}	)  \otimes 
 \psi_2 \psi_1 
(y_1e({\color{magenta}i_1} 			, 		{\color{gray}	i_1,i_1-1}, {\color{magenta}i_1+1})(y_2-y_4))
  \psi_1 \psi_2 
  \otimes 
  e({\color{gray} i_1-2\searrow\nearrow 	j_1} 	)\otimes e_{\SSTT_{>k+1}}
\end{align*}
and by $\psi_1 y_1 y_2 \psi_1 e(i,i) = 0$ (by relations \eqref{rel:dotcrossbun1}, \eqref{rel:dotcrossbun2}, and \eqref{rel:quadr}) we obtain 
\begin{align*}
 -
 e_{\SSTT_{<k}} \otimes e ({\color{magenta}i_2 \nearrow i_1-1}	)  \otimes 
  \psi_2 \psi_1 
(y_1 y_4 e({\color{magenta}i_1} 			, 		{\color{gray} i_1,i_1-1}, {\color{magenta}i_1+1}) )
  \psi_1 \psi_2 
   \otimes 
  e({\color{gray} i_1-2\searrow\nearrow 	j_1} 	)\otimes e_{\SSTT_{>k+1}}
\end{align*}
and by \cref{case1lem}  we obtain 
\[
-  
 e_{\SSTT_{<k}} \otimes e ({\color{magenta}i_2 \nearrow i_1-1}	)  \otimes 
y_4 e({\color{gray} i_1,i_1-1}, {\color{magenta}i_1},{\color{magenta}i_1+1})  \otimes 
 e({\color{gray} i_1-2\searrow\nearrow j_1} )\otimes e_{\SSTT_{>k+1}}.
\]
We then substitute this back into the wider diagram $ \psi^{\SSTT_\nu}_{\SSTT}
 ( \psi^{\SSTT }_{\SSTS}
 y_{\SSTS}
  \psi_{\SSTT}^{\SSTS})
  \psi_{\SSTT_\nu}^{\SSTT} $; 
we undo the remaining crossings with \cref{small-lemma} and commutativity relations to obtain 
$\psi^{\SSTT_\nu}_{\SSTT}y_{\SSTT}  \psi_{\SSTT_\nu}^{\SSTT}$, as required.
Examples are depicted in \cref{case2pic42,tableautileeg3}.

\begin{figure}[ht!]
\[
  \begin{tikzpicture}[scale=0.225]
 \draw(-1,0) rectangle (24.5*2,8);
\foreach \i in {0,1,2,...,24}
{\fill (\i*2,0) circle (4pt);
\fill (\i*2,8) circle (4pt);
\path(\i*2,8) coordinate (X\i);
\path(\i*2,0) coordinate (Y\i);
 }
 
 \draw (X0) node [above] {\scalefont{0.7}$0$};
  \draw (X1) node [above] {\scalefont{0.7}$1$};
   \draw (X2) node [above] {\scalefont{0.7}$2$};
 \draw (X3) node [above] {\scalefont{0.7}$3$};

  \draw (X4) node [above] {\scalefont{0.7}$1$};
   \draw (X5) node [above] {\scalefont{0.7}$0$};
   \draw (X6) node [above] {\scalefont{0.7}$1$};
  \draw (X7) node [above] {\scalefont{0.7}$2$};

   \draw (X8) node [above] {\scalefont{0.7}$2$};
  \draw (X9) node [above] {\scalefont{0.7}$1$};
  \draw (X10) node [above] {\scalefont{0.7}$0$};
   \draw (X11) node [above] {\scalefont{0.7}$1$};

   \draw (X12) node [above] {\scalefont{0.7}$3$};
   \draw (X13) node [above] {\scalefont{0.7}$2$};
     \draw (X14) node [above] {\scalefont{0.7}$1$};
       \draw (X15) node [above] {\scalefont{0.7}$0$};

   \draw (X16) node [above] {\color{magenta}\scalefont{0.7}$4$};
     \draw (X17) node [above] {\color{magenta}\scalefont{0.7}$5$};
       \draw (X18) node [above] {\color{magenta}\scalefont{0.7}$6$};

          \draw (X19) node [above] {\color{magenta}\scalefont{0.7}$3$};
     \draw (X20) node [above] {\color{magenta}\scalefont{0.7}$4$};
       \draw (X21) node [above] {\color{magenta}\scalefont{0.7}$5$};

          \draw (X22) node [above] {\color{magenta}\scalefont{0.7}$2$};
     \draw (X23) node [above] {\color{magenta}\scalefont{0.7}$3$};
       \draw (X24) node [above] {\color{magenta}\scalefont{0.7}$4$};

 \fill[opacity =0.2] (X0)rectangle  (Y3); 
 \draw[thick](X0)to (Y0);
  \draw[thick](X1)to (Y1);
   \draw[thick](X2)to (Y2);

 \fill[opacity =0.2] (X4)rectangle  (Y7); 
 \draw[thick](X3)to (Y3);
  \draw[thick](X4)to (Y4);
   \draw[thick](X5)to (Y5);
   \draw[thick](X6)to (Y6);
      \draw[thick](X7)to (Y7);

\fill[opacity =0.2] (X8)to (Y11) to (Y14) to (X11);
  \draw[thick](X8)to (Y11);
  \draw[thick](X9)to (Y12);
   \draw[thick](X10)to (Y13);
   \draw[thick](X11)to (Y14);

\fill[opacity =0.2] (X12)to (Y18) to (Y21) to (X15);
  \draw[thick](X12)to (Y18);
  \draw[thick](X13)to (Y19);
   \draw[thick](X14)to (Y20);
   \draw[thick](X15)to (Y21);

\fill[opacity =0.3,magenta] (X16)to (Y8) to (Y10) to (X18);
  \draw[thick](X16)to (Y8);
  \draw[thick](X17)to (Y9);
   \draw[thick](X18)to (Y10);

\fill[opacity =0.3,magenta] (X19)to (Y15) to (Y17) to (X21);
  \draw[thick](X19)to (Y15);
  \draw[thick](X20)to (Y16);
   \draw[thick](X21)to (Y17);

\fill[opacity =0.3,magenta] (X22)to (Y22) to (Y24) to (X24);
  \draw[thick](X22)to (Y22);
  \draw[thick](X23)to (Y23);
   \draw[thick](X24)to (Y24);

  \draw(-1,0-8) rectangle (24.5*2,8-8);
\foreach \i in {0,1,2,...,24}
{\fill (\i*2,0-8) circle (4pt);
\fill (\i*2,8-8) circle (4pt);
\path(\i*2,8-8) coordinate (X\i);
\path(\i*2,0-8) coordinate (Y\i);
 }

 \fill[opacity =0.2] (X0)rectangle  (Y3); 
 \draw[thick](X0)to (Y0);
  \draw[thick](X1)to (Y1);
   \draw[thick](X2)to (Y2);

 \fill[opacity =0.2] (X4)rectangle  (Y7); 
 \draw[thick](X3)to (Y3);
  \draw[thick](X4)to (Y4);
   \draw[thick](X5)to (Y5);
   \draw[thick](X6)to (Y6);
      \draw[thick](X7)to (Y7);

\fill[opacity =0.3,magenta] (X10)rectangle  (Y8); 
   \draw[thick](X8)to (Y8);
   \draw[thick](X9)to (Y9);
   \draw[thick](X10)to (Y10);

\fill[opacity =0.2 ] (X11)rectangle  (Y14); 
   \draw[thick](X11)to (Y11);
   \draw[thick](X12)to (Y12);
   \draw[thick](X13)to (Y13);
   \draw[thick](X14)to (Y14);

\fill[opacity =0.3,magenta] (X15)rectangle  (Y17); 
   \draw[thick](X15)to (Y15);
   \draw[thick](X16)to (Y16);
   \draw[thick](X17)to (Y17);

   \draw[thick](X22)to [out=-90,in=90](Y18);
   \draw[line width=3,magenta,opacity=0.3](X22)to [out=-90,in=90](Y18);
   \draw[thick](X23)to [out=-90,in=90](Y21);
   \draw[line width=3,magenta,opacity=0.3](X23)to [out=-90,in=90](Y21);
   \draw[thick](X24)to [out=-90,in=90](Y22);
   \draw[line width=3,magenta,opacity=0.3](X24)to [out=-90,in=90](Y22);

  \draw[thick](X18)to [out=-90,in=90](Y19);
   \draw[line width=3, opacity=0.2](X18)to [out=-90,in=90](Y19);
  \draw[thick](X19)to [out=-90,in=90](Y20);
   \draw[line width=3, opacity=0.2](X19)to [out=-90,in=90](Y20);
  \draw[thick](X20)to [out=-90,in=90](Y23);
   \draw[line width=3, opacity=0.2](X20)to [out=-90,in=90](Y23);
  \draw[thick](X21)to [out=-90,in=90](Y24);
   \draw[line width=3, opacity=0.2](X21)to [out=-90,in=90](Y24);

  \draw(-1,0-16) rectangle (24.5*2,8-16);
\foreach \i in {0,1,2,...,24}
{\fill (\i*2,0-16) circle (4pt);
\fill (\i*2,8-16) circle (4pt);
\path(\i*2,8-16) coordinate (X\i);
\path(\i*2,0-16) coordinate (Y\i);
\path(\i*2,0-12) coordinate (Z\i);
 }

 \fill[opacity =0.2] (X0)rectangle  (Y3); 
 \draw[thick](X0)to (Y0);
  \draw[thick](X1)to (Y1);
   \draw[thick](X2)to (Y2);

 \fill[opacity =0.2] (X4)rectangle  (Y7); 
 \draw[thick](X3)to (Y3);
  \draw[thick](X4)to (Y4);
   \draw[thick](X5)to (Y5);
   \draw[thick](X6)to (Y6);
      \draw[thick](X7)to (Y7);

\fill[opacity =0.3,magenta] (X10)rectangle  (Y8); 
   \draw[thick](X8)to (Y8);
   \draw[thick](X9)to (Y9);
   \draw[thick](X10)to (Y10);

\fill[opacity =0.2 ] (X11)rectangle  (Y14); 
   \draw[thick](X11)to (Y11);
   \draw[thick](X12)to (Y12);
   \draw[thick](X13)to (Y13);
   \draw[thick](X14)to (Y14);

\fill[opacity =0.3,magenta] (X15)rectangle  (Y17); 
   \draw[thick](X15)to (Y15);
   \draw[thick](X16)to (Y16);
   \draw[thick](X17)to (Y17);

   \draw[thick](X18)to [out=-90,in=90](Y18);
   \draw[line width=3,magenta,opacity=0.3](X18)to [out=-90,in=90](Y18);
   \draw[thick](X21)to [out=-90,in=90](Z19)to [out=-90,in=90](Y21);

\foreach \i in {1,7,19}
{
 \draw[thick,fill=black](Z\i) circle (10pt);
}

   \draw[line width=3,magenta,opacity=0.3](X21)to [out=-90,in=90](Z19)to [out=-90,in=90](Y21);;
   \draw[thick](X22)to [out=-90,in=90](Z20)
   to [out=-90,in=90](Y22);;
   \draw[line width=3,magenta,opacity=0.3](X22)to [out=-90,in=90](Z20)   to [out=-90,in=90](Y22);;

  \draw[thick](X19)to [out=-90,in=90](Z21)to [out=-90,in=90](Y19);
   \draw[line width=3, opacity=0.2](X19)to [out=-90,in=90](Z21)to [out=-90,in=90](Y19);
  \draw[thick](X20)to [out=-90,in=90](Z22)to [out=-90,in=90](Y20);
   \draw[line width=3, opacity=0.2](X20)to [out=-90,in=90](Z22)to [out=-90,in=90](Y20);

   \draw[thick](X23)to [out=-90,in=90](Y23);
    \draw[line width=3, opacity=0.2](X23)to [out=-90,in=90](Y23);
       \draw[thick](X24)to [out=-90,in=90](Y24);
    \draw[line width=3, opacity=0.2](X24)to [out=-90,in=90](Y24);

  \draw(-1,0-8-8-8) rectangle (24.5*2,8-8-8-8);
\foreach \i in {0,1,2,...,24}
{\fill (\i*2,0-8-8-8) circle (4pt);
\fill (\i*2,8-8-8-8) circle (4pt);
\path(\i*2,8-8-8-8-8) coordinate (X\i);
\path(\i*2,0-8-8) coordinate (Y\i);
 }

 \fill[opacity =0.2] (X0)rectangle  (Y3); 
 \draw[thick](X0)to (Y0);
  \draw[thick](X1)to (Y1);
   \draw[thick](X2)to (Y2);

 \fill[opacity =0.2] (X4)rectangle  (Y7); 
 \draw[thick](X3)to (Y3);
  \draw[thick](X4)to (Y4);
   \draw[thick](X5)to (Y5);
   \draw[thick](X6)to (Y6);
      \draw[thick](X7)to (Y7);

\fill[opacity =0.3,magenta] (X10)rectangle  (Y8); 
   \draw[thick](X8)to (Y8);
   \draw[thick](X9)to (Y9);
   \draw[thick](X10)to (Y10);

\fill[opacity =0.2 ] (X11)rectangle  (Y14); 
   \draw[thick](X11)to (Y11);
   \draw[thick](X12)to (Y12);
   \draw[thick](X13)to (Y13);
   \draw[thick](X14)to (Y14);

\fill[opacity =0.3,magenta] (X15)rectangle  (Y17); 
   \draw[thick](X15)to (Y15);
   \draw[thick](X16)to (Y16);
   \draw[thick](X17)to (Y17);

   \draw[thick](X22)to [out=90,in=-90](Y18);
   \draw[line width=3,magenta,opacity=0.3](X22)to [out=90,in=-90](Y18);
   \draw[thick](X23)to [out=90,in=-90](Y21);
   \draw[line width=3,magenta,opacity=0.3](X23)to [out=90,in=-90](Y21);
   \draw[thick](X24)to [out=90,in=-90](Y22);
   \draw[line width=3,magenta,opacity=0.3](X24)to [out=90,in=-90](Y22);

  \draw[thick](X18)to [out=90,in=-90](Y19);
   \draw[line width=3, opacity=0.2](X18)to [out=90,in=-90](Y19);
  \draw[thick](X19)to [out=90,in=-90](Y20);
   \draw[line width=3, opacity=0.2](X19)to [out=90,in=-90](Y20);
  \draw[thick](X20)to [out=90,in=-90](Y23);
   \draw[line width=3, opacity=0.2](X20)to [out=90,in=-90](Y23);
  \draw[thick](X21)to [out=90,in=-90](Y24);
   \draw[line width=3, opacity=0.2](X21)to [out=90,in=-90](Y24);


 \draw(-1,0-32) rectangle (24.5*2,8-32);
\foreach \i in {0,1,2,...,24}
{\fill (\i*2,0-32) circle (4pt);
\fill (\i*2,8-32) circle (4pt);
\path(\i*2,8-32) coordinate (Y\i);
\path(\i*2,0-32) coordinate (X\i);
 }
 
 \draw (X0) node [below] {\scalefont{0.7}$0$};
  \draw (X1) node [below] {\scalefont{0.7}$1$};
   \draw (X2) node [below] {\scalefont{0.7}$2$};
 \draw (X3) node [below] {\scalefont{0.7}$3$};

  \draw (X4) node [below] {\scalefont{0.7}$1$};
   \draw (X5) node [below] {\scalefont{0.7}$0$};
   \draw (X6) node [below] {\scalefont{0.7}$1$};
  \draw (X7) node [below] {\scalefont{0.7}$2$};

   \draw (X8) node [below] {\scalefont{0.7}$2$};
  \draw (X9) node [below] {\scalefont{0.7}$1$};
  \draw (X10) node [below] {\scalefont{0.7}$0$};
   \draw (X11) node [below] {\scalefont{0.7}$1$};

   \draw (X12) node [below] {\scalefont{0.7}$3$};
   \draw (X13) node [below] {\scalefont{0.7}$2$};
     \draw (X14) node [below] {\scalefont{0.7}$1$};
       \draw (X15) node [below] {\scalefont{0.7}$0$};

   \draw (X16) node [below] {\color{magenta}\scalefont{0.7}$4$};
     \draw (X17) node [below] {\color{magenta}\scalefont{0.7}$5$};
       \draw (X18) node [below] {\color{magenta}\scalefont{0.7}$6$};

          \draw (X19) node [below] {\color{magenta}\scalefont{0.7}$3$};
     \draw (X20) node [below] {\color{magenta}\scalefont{0.7}$4$};
       \draw (X21) node [below] {\color{magenta}\scalefont{0.7}$5$};

          \draw (X22) node [below] {\color{magenta}\scalefont{0.7}$2$};
     \draw (X23) node [below] {\color{magenta}\scalefont{0.7}$3$};
       \draw (X24) node [below] {\color{magenta}\scalefont{0.7}$4$};

 \fill[opacity =0.2] (X0)rectangle  (Y3); 
 \draw[thick](X0)to (Y0);
  \draw[thick](X1)to (Y1);
   \draw[thick](X2)to (Y2);

 \fill[opacity =0.2] (X4)rectangle  (Y7); 
 \draw[thick](X3)to (Y3);
  \draw[thick](X4)to (Y4);
   \draw[thick](X5)to (Y5);
   \draw[thick](X6)to (Y6);
      \draw[thick](X7)to (Y7);

\fill[opacity =0.2] (X8)to (Y11) to (Y14) to (X11);
  \draw[thick](X8)to (Y11);
  \draw[thick](X9)to (Y12);
   \draw[thick](X10)to (Y13);
   \draw[thick](X11)to (Y14);

\fill[opacity =0.2] (X12)to (Y18) to (Y21) to (X15);
  \draw[thick](X12)to (Y18);
  \draw[thick](X13)to (Y19);
   \draw[thick](X14)to (Y20);
   \draw[thick](X15)to (Y21);

\fill[opacity =0.3,magenta] (X16)to (Y8) to (Y10) to (X18);
  \draw[thick](X16)to (Y8);
  \draw[thick](X17)to (Y9);
   \draw[thick](X18)to (Y10);

\fill[opacity =0.3,magenta] (X19)to (Y15) to (Y17) to (X21);
  \draw[thick](X19)to (Y15);
  \draw[thick](X20)to (Y16);
   \draw[thick](X21)to (Y17);

\fill[opacity =0.3,magenta] (X22)to (Y22) to (Y24) to (X24);
  \draw[thick](X22)to (Y22);
  \draw[thick](X23)to (Y23);
   \draw[thick](X24)to (Y24);
\end{tikzpicture}
\]
\caption{We depict  an example of Case (3b) corresponding to the tableaux $\SSTS,\SSTT$ pictured in \cref{tableautileeg3}.
The third layer of the diagram is equal to $D'$ from the proof.
}\label{case2pic42}
\end{figure} 

\begin{figure}[ht!]
\[
\begin{tikzpicture}[scale=0.6]
\draw[thick] (0,0)--(7,0)--++(-90:3)
--++(180:3)--++(-90:1)--++(180:4)--(0,0);

\clip (0,0)--(7,0)--++(-90:3)
--++(180:3)--++(-90:1)--++(180:4)--(0,0);

\fill [opacity=0.2] (0,0) rectangle (4,-4);

\fill [opacity=0.3,magenta] (4,0) rectangle (7,-3);
\draw[thick] (0,0) rectangle (3,-3);

\draw[thick] (0,-1) --++(0:7);
\draw[thick] (0,-2) --++(0:7);
\draw[thick] (0,-3) --++(0:7);

\draw[thick] (1,0) --++(-90:7);
\draw[thick] (2,0) --++(-90:7);
\draw[thick] (4,0) --++(-90:7);
\draw[thick] (3,0) --++(-90:7);
\draw[thick] (5,0) --++(-90:7);
\draw[thick] (6,0) --++(-90:7);

\draw (0.5,-0.5) node {$0$};
\draw (1.5,-0.5) node {$1$};
\draw (2.5,-0.5) node {$2$};
\draw (3.5,-0.5) node {$3$};
\draw (4.5,-0.5) node {$4$};
\draw (5.5,-0.5) node {$5$};
\draw (6.5,-0.5) node {$6$};

\draw (0.5,-1.5) node {$1$};
\draw (1.5,-1.5) node {$0$};
\draw (2.5,-1.5) node {$1$};
\draw (3.5,-1.5) node {$2$};
\draw (4.5,-1.5) node {$3$};
\draw (5.5,-1.5) node {$4$};
\draw (6.5,-1.5) node {$5$};
\draw (7.5,-1.5) node {$6$};

\draw (0.5,-2.5) node {$2$};
\draw (1.5,-2.5) node {$1$};
\draw (2.5,-2.5) node {$0$};
\draw (3.5,-2.5) node {$1$};
\draw (4.5,-2.5) node {$2$};
\draw (5.5,-2.5) node {$3$};
\draw (6.5,-2.5) node {$4$};

\draw (0.5,-3.5) node {$3$};
\draw (1.5,-3.5) node {$2$};
\draw (2.5,-3.5) node {$1$};
\draw (3.5,-3.5) node {$0$};
\end{tikzpicture}
\qquad
\begin{tikzpicture}[scale=0.6]
\draw[thick] (0,0)--(7,0)--++(-90:3)
--++(180:3)--++(-90:1)--++(180:4)--(0,0);

\clip (0,0)--(7,0)--++(-90:3)
--++(180:3)--++(-90:1)--++(180:4)--(0,0);

\fill [opacity=0.2] (0,0) rectangle (4,-4);

\fill [opacity=0.3,magenta] (4,0) rectangle (7,-3);
\draw[thick] (0,0) rectangle (3,-3);

\draw[thick] (0,-1) --++(0:7);
\draw[thick] (0,-2) --++(0:7);
\draw[thick] (0,-3) --++(0:7);

\draw[thick] (1,0) --++(-90:7);
\draw[thick] (2,0) --++(-90:7);
\draw[thick] (4,0) --++(-90:7);
\draw[thick] (3,0) --++(-90:7);
\draw[thick] (5,0) --++(-90:7);
\draw[thick] (6,0) --++(-90:7);

\draw (0.5,-0.5) node {$1$};
\draw (1.5,-0.5) node {$1$};
\draw (2.5,-0.5) node {$1$};
\draw (3.5,-0.5) node {$1$};
\draw (4.5,-0.5) node {$3$};
\draw (5.5,-0.5) node {$3$};
\draw (6.5,-0.5) node {$3$};

\draw (0.5,-1.5) node {$2$};
\draw (1.5,-1.5) node {$2$};
\draw (2.5,-1.5) node {$2$};
\draw (3.5,-1.5) node {$2$};
\draw (4.5,-1.5) node {$5$};
\draw (5.5,-1.5) node {$5$};
\draw (6.5,-1.5) node {$5$};
\draw (7.5,-1.5) node {$6$};

\draw (0.5,-2.5) node {$4$};
\draw (1.5,-2.5) node {$4$};
\draw (2.5,-2.5) node {$4$};
\draw (3.5,-2.5) node {$4$};
\draw (4.5,-2.5) node {$6$};
\draw (5.5,-2.5) node {$6$};
\draw (6.5,-2.5) node {$6$};

\draw (0.5,-3.5) node {$7$};
\draw (1.5,-3.5) node {$7$};
\draw (2.5,-3.5) node {$7$};
\draw (3.5,-3.5) node {$7$};
\end{tikzpicture}
\qquad
\begin{tikzpicture}[scale=0.6]
\draw[thick] (0,0)--(7,0)--++(-90:3)
--++(180:3)--++(-90:1)--++(180:4)--(0,0);

\clip (0,0)--(7,0)--++(-90:3)
--++(180:3)--++(-90:1)--++(180:4)--(0,0);

\fill [opacity=0.2] (0,0) rectangle (4,-4);

\fill [opacity=0.3,magenta] (4,0) rectangle (7,-3);
\draw[thick] (0,0) rectangle (3,-3);

\draw[thick] (0,-1) --++(0:7);
\draw[thick] (0,-2) --++(0:7);
\draw[thick] (0,-3) --++(0:7);

\draw[thick] (1,0) --++(-90:7);
\draw[thick] (2,0) --++(-90:7);
\draw[thick] (4,0) --++(-90:7);
\draw[thick] (3,0) --++(-90:7);
\draw[thick] (5,0) --++(-90:7);
\draw[thick] (6,0) --++(-90:7);

\draw (0.5,-0.5) node {$1$};
\draw (1.5,-0.5) node {$1$};
\draw (2.5,-0.5) node {$1$};
\draw (3.5,-0.5) node {$1$};
\draw (4.5,-0.5) node {$3$};
\draw (5.5,-0.5) node {$3$};
\draw (6.5,-0.5) node {$3$};

\draw (0.5,-1.5) node {$2$};
\draw (1.5,-1.5) node {$2$};
\draw (2.5,-1.5) node {$2$};
\draw (3.5,-1.5) node {$2$};
\draw (4.5,-1.5) node {$5$};
\draw (5.5,-1.5) node {$5$};
\draw (6.5,-1.5) node {$5$};
\draw (7.5,-1.5) node {$6$};

\draw (0.5,-2.5) node {$4$};
\draw (1.5,-2.5) node {$4$};
\draw (2.5,-2.5) node {$4$};
\draw (3.5,-2.5) node {$4$};
\draw (4.5,-2.5) node {$7$};
\draw (5.5,-2.5) node {$7$};
\draw (6.5,-2.5) node {$7$};

\draw (0.5,-3.5) node {$6$};
\draw (1.5,-3.5) node {$6$};
\draw (2.5,-3.5) node {$6$};
\draw (3.5,-3.5) node {$6$};
\end{tikzpicture}
\]
\caption{On the left we depict the $\tt C$-residues of $\rho + \pla$.
The two semistandard tableaux are examples of $\SSTS  
=s_{6}(\SSTT)$ as in Case 3.}
\label{tableautileeg3}
\end{figure}

\smallskip

\noindent{\bf Case 3c. } 
We now suppose that $\deg(\SSTT^{-1}(k))=1$ and $\deg(\SSTS^{-1}(k+1))=1$.
We also have $\deg(\SSTT^{-1}(k+1)) = 0 = \deg(\SSTS^{-1}(k))$ in this case.
In this case $\color{gray}[i_1,j_1]$ and $\color{magenta}[i_2,j_2]$ are well-separated, and the result follows by the commutativity relations.
\end{proof}

Finally, we are able to deduce the main result of the paper (as discussed in the introduction), a graded Morita equivalence between cyclotomic quiver Hecke algebras of types ${\tt A}_\infty$ and ${\tt C}_\infty$.

\begin{thm}\label{cor:Morita+grdec}
The algebras $\scrr^{\La_{\kapc}}_{\beta}(\mathfrak{sp}_\infty)$ and  
$\scrr^{\La_{{\color{magenta}\kappa_1}} + \La_{{\color{cyan}\kappa_2}}}_{\beta - \omega}(\mathfrak{sl}_\infty)$ are graded Morita equivalent.
The graded simple and Specht modules match up naturally via the map 
\[
\D{(\pla,\bbmu)} \longmapsto 
\D{\rho + (\pla,\bbmu{\color{cyan}'})}
\qquad
\spe{(\pla,\bbmu)} \longmapsto 
\spe{\rho + (\pla,\bbmu{\color{cyan}'})}
\]
and hence the graded decomposition matrices are preserved by this Morita equivalence.
\end{thm}

\begin{proof}
We will show that the isomorphism $\varphi$ in \cref{thm:homomorphism,thm:isomorphism} sends cell ideals to cell ideals, and in particular sends the Specht module $ \spe{\rho}\otimes \spe{(\pla,\bbmu)}$ to the Specht module $\spe\nu$, where 
$\nu = \rho + (\pla,\bbmu{\color{cyan}'})$. We will hence deduce (by the construction of simple modules as quotients of Specht modules)  that the map  takes  the simple module $\D{\rho}\otimes \D{(\pla,\bbmu)}$ to the simple module $\D\nu$, where 
$\nu = \rho + (\pla,\bbmu{\color{cyan}'})$. 
The result then follows by \cref{neededlater}, since $\scrr^{\La_{\kapc}}_{\omega}(\mathfrak{sp}_\infty)$ is a simple algebra.

First, we observe that the poset induced by the $\dom'$ dominance order on bipartitions $(\pla,\bbmu) \in \ParblockA$, as defined in \cref{lem:conjugatedomorder}, agrees with that induced by the dominance order on partitions $\nu    = \rho + (\pla,\bbmu{\color{cyan}'}) \in \ParblockC$.
It's   clear that the number of standard tableaux in $\Std(\rho) \times \Std((\pla,\bbmu))$ is the same as the number of standard $\nu$-tableaux (with $\nu  = \rho + (\pla,\bbmu{\color{cyan}'})$) whose residue sequences factorise as $\bfi \otimes \bfj$, for $\bfi \in I^\omega$ and $\bfj \in I^{\beta-\omega}$.
It follows that the dimensions of the cell ideals matched by the combinatorial map $(\rho,(\pla,\bbmu)) \mapsto \rho + (\pla,\bbmu{\color{cyan}'})$ naturally agree, too. 
Therefore it suffices to check that $\varphi$ maps the cell ideal generator for the cell indexed by $(\rho,(\pla,\bbmu))$ into the corresponding cell ideal indexed by $\rho + (\pla,\bbmu{\color{cyan}'})$.
Applying the map $\varphi$ of \eqref{sakljghdlfjhgjldskhgjdflk} to a cell ideal generator, we obtain 
\[
\varphi(c_{\ttt^\rho \ttt^\rho}
 \otimes
c_{\ttt^{(\pla,\bbmu)}, \ttt^{(\pla,\bbmu)}})
=
\varphi(
y_{\ttt^\rho}
\otimes
y_{\ttt^{(\pla,\bbmu)}})= 
 1_{\omega, \beta-\omega} 
(
y_{\ttt^\rho}
\otimes
y_{\ttt^{(\pla,\bbmu)}}) 1_{\omega, \beta-\omega}.
\]
The righthand-side diagram can be factorised as
\[
1_{\omega, \beta-\omega} 
(
y_{\ttt^\rho}
\otimes
y_{\ttt^{(\pla,\bbmu)}}) 1_{\omega, \beta-\omega}  
=
y_{\ttt^\rho} 
\otimes
y_{\ttt^{(\pla,\varnothing)}}
\otimes
e(\bfi^{\bbmu})
=
\psi^{\SSTT_{\rho+\pla}}_{\SSTT_{\rho+\pla}}
 y_{\SSTT_{\rho+\pla}}
  \psi_{\SSTT_{\rho+\pla}}^{\SSTT_{\rho+\pla}}
\otimes
e(\bfi^{\bbmu})
=
\pm \psi^{\SSTT_{\rho+\pla}}_{\SSTT^{\rho+\pla}}
y_{\SSTT^{\rho+\pla}}
  \psi_{\SSTT_{\rho+\pla}}^{\SSTT^{\rho+\pla}}
\otimes
e(\bfi^{\bbmu})
\]
where the first two equalities are by definition, and the third equality follows by applying \cref{BOOM} repeatedly.
We also have that $e(\bfi^{\bbmu}) = \psi_{\ttt_{\bbmu{\color{cyan}'}}} e(\bfi^{\bbmu{\color{cyan}'}}) \psi_{\ttt_{\bbmu{\color{cyan}'}}}^\ast$ using the commutativity relations; 
by definition $y_{\SSTT^{\rho+\pla}} = y_{\ttt^{\rho+\pla}}$ and so we deduce that  
\begin{align*}
\varphi(y_{\ttt^\rho} \otimes y_{\ttt^{(\pla,\bbmu)}}) 
&=
\pm \psi^{\SSTT_{\rho+\pla}}_{\SSTT^{\rho+\pla}}
y_{\ttt^{\rho+\pla}}
  \psi_{\SSTT_{\rho+\pla}}^{\SSTT^{\rho+\pla}}
\otimes
\psi_{\ttt_{\bbmu{\color{cyan}'}}} e(\bfi^{\bbmu{\color{cyan}'}}) \psi_{\ttt_{\bbmu{\color{cyan}'}}}^\ast\\
&= 
\pm
(\psi^{\SSTT_{\rho+\pla}}_{\SSTT^{\rho+\pla}} \otimes \psi_{\ttt_{\bbmu{\color{cyan}'}}})
y_{\ttt^{\rho + (\pla,\bbmu{\color{cyan}'})}}
(\psi_{\SSTT_{\rho+\pla}}^{\SSTT^{\rho+\pla}} \otimes \psi_{\ttt_{\bbmu{\color{cyan}'}}}^\ast),
\end{align*}
an element of the two-sided cell ideal $\mathscr{R}^{\La_{\kapc}}_\beta y_{\ttt^{\rho + (\pla,\bbmu{\color{cyan}'})}} \mathscr{R}^{\La_{\kapc}}_\beta$.
This completes the proof.
\end{proof}

%
%
%
%
%

\begin{rem}
In \cite{lst25}, the third author, with Li and Tan, has generalised \cref{cor:Morita+grdec} to the setting of \emph{core blocks} of $\scrr^{\La_{k}}_{\beta}(\mathfrak{sp}^{(1)}_{2e})$, yielding characteristic-free graded decomposition numbers in that setting.
\end{rem}

\begin{rem}
 By \cref{cor:Morita+grdec} the off-diagonal entries of the 
 graded decomposition matrices of 
 $\scrr^{\La_{\kapc}}_{\beta}(\mathfrak{sp}_\infty)$ 
 are in strictly positive degree (in fact they are given by anti-spherical ($p$-)Kazhdan--Lusztig polynomials for maximal finite parabolics of finite symmetric groups).
This implies that they match the canonical basis coefficients for the type ${\tt C}_\infty$ highest weight irreducible module $V(\La_{\kapc})$.
\end{rem}

\begin{rem}
In \cite{cms25}, the third author, Chung and Mathas have computed the graded decomposition matrices for all $\scrr^{\La_{\kapc}}_{\beta}(\fkg)$ with $\operatorname{ht}(\beta) \leq 12$ for $\fkg = \mathfrak{sp}_\infty$ and $\mathfrak{sp}^{(1)}_{2e}$.
In particular, we found examples where the graded decomposition numbers in characteristic 0 match neither the canonical basis coefficients nor the characteristic $p$ graded decomposition numbers.
Thanks to \cref{cor:Morita+grdec}, we now know that such examples cannot occur (in level 1) in type $\mathfrak{sp}_\infty$, where the canonical basis coefficients match the graded decomposition numbers in any characteristic.
\end{rem}

\bibliographystyle{lspaper}
\phantomsection
\bibliography{master}

 \end{document}